\documentclass[11pt]{article}

\usepackage[english]{babel}
\usepackage[ansinew]{inputenc}
\usepackage{verbatim}
\usepackage{amsmath, amsthm, amsfonts, amssymb}
\usepackage{latexsym}
\usepackage{bbm}
\usepackage{mathrsfs}
\usepackage{tikz}
\usepackage{graphicx}
\usepackage{multirow}
\usepackage{subfig}
\usetikzlibrary{matrix,arrows}

\newtheorem{thm}{Theorem}[section]

\newtheorem{corollary}[thm]{Corollary}

\newtheorem{lemma}[thm]{Lemma}

\newtheorem{proposition}[thm]{Proposition}

\theoremstyle{definition}
\newtheorem{definition}[thm]{Definition}

\newtheorem{remark}[thm]{Remark}

\newtheorem{conjecture}[thm]{Conjecture}

\theoremstyle{remark}

\numberwithin{equation}{section}

\newcommand{\IE}{\mathbb{E}}
\newcommand{\IN}{\mathbb{N}}
\newcommand{\IR}{\mathbb{R}}

\newcommand{\M}{\mathcal{R}}

\newcommand{\mcD}{\mathcal{D}}
\newcommand{\mcH}{\mathcal{M}}
\newcommand{\mcK}{\mathcal{K}}
\newcommand{\mcL}{\mathcal{L}}

\newcommand{\mcO}{\mathcal{O}}
\newcommand{\mcP}{\mathcal{P}}

\newcommand{\mcT}{\mathcal{T}}
\newcommand{\mcW}{\mathcal{W}}

\newcommand{\W}{\mcW}

\newcommand{\CC}{\mathscr{C}}
\newcommand{\hd}{{\textup{\tiny H}}}
\newcommand{\PiS}{\Pi_{\scriptscriptstyle\Sigma}}
\newcommand{\PiM}{\Pi_{\scriptscriptstyle M}}
\newcommand{\PiMh}{\hat{\Pi}_{\scriptscriptstyle M}}
\newcommand{\Pib}{\Pi_b}
\newcommand{\nuS}{\nu_{\scriptscriptstyle \Sigma}}
\newcommand{\Gra}{{\textup{\tiny G}}}

\newcommand{\sm}{{\textup{\tiny sm}}}
\newcommand{\Pri}{{\textup{\tiny P}}}
\newcommand{\Dua}{{\textup{\tiny D}}}
\newcommand{\TP}{\mcT^\Pri}
\newcommand{\TD}{\mcT^\Dua}

\DeclareMathOperator{\vol}{vol}
\DeclareMathOperator{\rvol}{rvol}

\DeclareMathOperator{\cone}{cone}
\DeclareMathOperator{\lin}{lin}
\DeclareMathOperator{\Skew}{Skew}
\DeclareMathOperator{\lpm}{lpm}
\DeclareMathOperator{\prm}{pm}

\DeclareMathOperator{\im}{im}
\DeclareMathOperator{\Gr}{Gr}
\DeclareMathOperator{\rk}{rk}
\DeclareMathOperator{\ndet}{ndet}
\DeclareMathOperator{\trace}{trace}
\DeclareMathOperator{\id}{id}

\DeclareMathOperator{\inter}{int}
\DeclareMathOperator{\Prob}{Prob}
\DeclareMathOperator{\tr}{tr}
\DeclareMathOperator{\Sym}{Sym}
\DeclareMathOperator{\ch}{ch}

\DeclareMathOperator{\diag}{diag}

\newcommand{\tdet}{{\textstyle\det}}
\newcommand{\rbinom}[2]{\genfrac{[}{]}{0pt}{}{#1}{#2}}
\newcommand{\ol}[1]{\overline{#1}}

\newcommand{\vp}{\varphi}
\newcommand{\veps}{\varepsilon}

\newcommand{\exc}{e}

\newcommand{\calc}{
  \scalebox{0.09}{
  \begin{tikzpicture}[scale=4]
    \draw[rounded corners=5mm] (0,-0.06) rectangle (0.95,1);
    \draw[rounded corners=1mm] (0.15,0.7) rectangle (0.8,0.85);
    \draw[rounded corners=0.7mm] (0.15,0.45) rectangle ++(0.15,0.1);
    \draw[rounded corners=0.7mm] (0.4,0.45)  rectangle ++(0.15,0.1);
    \draw[rounded corners=0.7mm] (0.65,0.45) rectangle ++(0.15,0.1);
    \draw[rounded corners=0.7mm] (0.15,0.27) rectangle ++(0.15,0.1);
    \draw[rounded corners=0.7mm] (0.4,0.27)  rectangle ++(0.15,0.1);
    \draw[rounded corners=0.7mm] (0.65,0.27) rectangle ++(0.15,0.1);
    \draw[rounded corners=0.7mm] (0.15,0.09) rectangle ++(0.15,0.1);
    \draw[rounded corners=0.7mm] (0.4,0.09)  rectangle ++(0.15,0.1);
    \draw[rounded corners=0.7mm] (0.65,0.09) rectangle ++(0.15,0.1);
  \end{tikzpicture}
  }
}

\def\CG{\CC}
\def\CR{\mathcal{R}}
\def\R{\IR}
\def\PG{\mathcal{P}}
\def\DG{\mathcal{D}}
\def\SG{\Sigma}
\def\inte{\mathrm{int}}
\def\GR{\Gr_{n,m}}

\usepackage[pdftex, pdfpagelabels]{hyperref}
\hypersetup{
  plainpages=false,
  colorlinks,
  citecolor=black,
  filecolor=black,
  linkcolor=black,
  urlcolor=black,
}

\begin{document}

\title{\bf Probabilistic analysis of the \\ Grassmann condition number}
\author{Dennis Amelunxen\thanks{E-mail: \href{mailto:damelunx@gmail.com}{\nolinkurl{damelunx@gmail.com}}. Partially supported by DFG Research Training Group on Scientific Computation GRK 693 (PaSCo GK) and DFG grants BU 1371/2-2, AM 386/1-1, AM 386/1-2.}
        \and
        Peter B\"urgisser\thanks{E-mail: \href{mailto:pbuerg@upb.de}{\nolinkurl{pbuerg@upb.de}}. Partially supported by DFG grant BU 1371/2-2.}
        }
\date{July 9, 2013}
% \date{\today}
\maketitle

\begin{abstract}
We analyze the probability that a random $m$-dimensional linear subspace of $\IR^n$ 
both intersects a regular closed convex cone~$C\subseteq\IR^n$ 
and lies within distance $\alpha$ of an $m$-dimensional 
subspace not intersecting $C$ (except at the origin). 
The result is expressed in terms of the spherical intrinsic volumes of the cone~$C$. 
This allows us to perform an average analysis of the Grassmann condition number $\CG(A)$
for the homogeneous convex feasibility problem $\exists x\in C\setminus0:Ax=0$.
The Grassmann condition number is a geometric version of Renegar's condition number, 
that we have introduced recently in [SIOPT~22(3):1029--1041, 2012].
We thus give the first average analysis of convex programming that is not restricted to linear programming. 
In particular, we prove that if the entries of $A\in\IR^{m\times n}$ are chosen i.i.d.~standard normal,
then for any regular cone~$C$, we have $\IE[\ln\CG(A)]<1.5\ln(n)+1.5$.
The proofs rely on various techniques from Riemannian geometry applied to 
Grassmann manifolds.  
\end{abstract}

\smallskip


\smallskip


\section{Introduction}

Convex programming is an efficient tool in modern applied mathematics.
In fact, a commonly accepted technique in current scientific computing 
is to ``convexify''supposedly hard problems, 
solve the relaxed convex problem, and then hope
that the result is close to a solution of the original problem.
To quote from~\cite[\S1.3.2]{BV:04}: ``With only a bit of exaggeration,
we can say that, if you formulate a practical problem as a convex
optimization problem, then you have solved the original problem.''

But what is the complexity of convex programming? To specify this
question further, we ask for the number of arithmetic
operations, or the number of iterations of an interior-point
method.
Steve Smale suggested in~\cite{Sm:97} to use the concepts of
condition numbers and probabilistic analysis in a two-part scheme
for the analysis of numerical algorithms:
1.~Establish a bound for the running time, which is  polynomial
in the size of the input and (the logarithm of) a certain
condition number of the input.
2.~Analyze the condition number of a random input in form of tail estimates.

The first step of this scheme, i.e., the analysis of the role of
condition numbers in convex programming, was initialized by Jim Renegar
in~\cite{rene:94,rene:95b,rene:95a}, and is
an active area of research, cf.~\cite{V:96, V:98, F:99, frve:99, 
FV:99, Pe:00, PR:00, ChC:01, cupe:02, EF:02, P:03, FO:05, VRP:07, BF:09}. 
In these references, the role of condition numbers 
is analyzed for linear and nonlinear convex programming, 
for exact arithmetic and for finite-precision arithmetic, 
for ellipsoid and for interior-point methods, etc.

Yet, the second step of Smale's scheme, i.e., the probabilistic analysis
of the condition number, was until now severly restricted to the linear
programming case. See the survey article~\cite{buer:09a} and the references
given therein for more details on probabilistic analyses of condition 
numbers for linear programming.

We will give in this paper the first average analysis of a condition number
for the general homogeneous convex feasibility problem. This includes the
special cases of linear programming, second-order programming, and notably
also the semidefinite programming case. More precisely, we consider the
following problem:

Let $C\subseteq\IR^n$ be a {\em regular cone}, i.e., $C$ is 
a closed convex cone with nonempty interior that 
does not contain a nontrivial linear subspace. 
The \emph{polar cone}
of $C$ is defined as 
$\breve{C} := \{z\in\IR^n\mid \forall x\in C : z^Tx\leq 0\}$.
We call $C$ \emph{self-dual} if $\breve{C}=-C$. 
The \emph{homogeneous convex feasibility problem} is to decide 
for a given matrix $A\in\IR^{m\times n}$, $1\leq m <n$, 
the alternative\footnote{In fact,~\eqref{eq:P} and~\eqref{eq:D} are
only weak alternatives, as it may happen that both~\eqref{eq:P} 
and~\eqref{eq:D} are satisfiable. But the Lebesgue measure of the
set of these ill-posed inputs in $\IR^{m\times n}$ is zero.}
\begin{align}
   &\exists x\in\IR^n\setminus 0 \;  \text{ s.t.} \quad 
            A x=0 ,\; x\in \breve{C} , \label{eq:P}\tag{P}\\
   &\exists y\in\IR^m\setminus   0 \; \text{ s.t. } \quad 
            A^Ty\in C . \label{eq:D}\tag{D}
\end{align}
This problem reduces to the linear feasibility problem if $C=\R_+^n$; 
it reduces to the second-order feasibility problem if $C=\mcL^{n_1}\times\ldots\times \mcL^{n_r}$,
where $\mcL^n:=\{x\in\IR^n\mid x_n\geq (x_1^2 + \cdots + \ldots,x_{n-1}^2)^{1/2}\}$
denotes the $n$-dimensional {\em Lorentz cone}; and it reduces to the
semidefinite feasibility problem if $C=\Sym^k_+=\{M\in\Sym^k\mid M\succeq 0\}$
is the {\em cone of positive semidefinite matrices}, where
$\Sym^k=\{M\in\IR^{k\times k}\mid M^T=M\}$.

\subsection{Grassmann condition number}\label{se:intro-Grassmann}

The condition number, for which we will provide an average analysis,
is the \emph{Grassmann condition number}, that we have introduced
in~\cite{am:thesis,ambu:11a}, cf.~also~\cite{BF:09}. 
Let us recall the necessary definitions from~\cite{ambu:11a}. 
We fix $1\leq m <n$ and consider the 
\emph{Grassmann manifold} $\Gr_{n,m}$, that is defined as the set of 
$m$-dimensional linear subspaces~$W$ of~$\R^n$. 

We also fix a regular cone $C\subseteq\IR^n$. 
The sets of \emph{dual feasible} and \emph{primal feasible} subspaces with respect to~$C$, respectively,  
are defined as follows:
\begin{align*}
  \DG_m(C) & := \{W\in\Gr_{n,m} \mid  W\cap C  \neq \{0\} \} , &
  \PG_m(C) & := \{W\in\Gr_{n,m} \mid W^\bot\cap \breve{C}  \neq \{0\} \} .
\end{align*}
Moreover, we we define the set of \emph{ill-posed subspaces with respect to~$C$} by 
  \[ \SG_m(C) := \DG_m(C)\cap\PG_m(C) . \]

It is known (cf.~\cite{ambu:11a}) that 
$\DG_m(C)$ and $\PG_m(C)$ are compact subsets of $\Gr_{n,m}$ and 
$\Gr_{n,m} = \DG_m(C)\cup\PG_m(C)$.
Moreover, the boundaries
of $\DG_m(C)$ and $\PG_m(C)$ coincide with $\SG_m(C)$.
Furthermore, 
  \[ \SG_m(C) = \{W\in\Gr_{n,m}\mid W\cap C \neq \{0\} \text{ and } 
     W\cap \inte(C) = \emptyset \} . \]
In other words, the set of ill-posed subspaces consists of those
subspaces, which touch the cone~$C$ at the boundary. 
One can show that $\PG_m(C)\setminus\SG_m(C)$ and 
$\DG_m(C)\setminus\SG_m(C)$ are the connected components of 
$\Gr_{n,m}\setminus\SG_m(C)$. 

The {\em projection distance} $d_p(W_1,W_2)$ of two subspaces $W_1,W_2\in \Gr_{n,m}$
is defined as the spectral norm 
$d_p(W_1,W_2) :=\|\Pi_{W_1}-\Pi_{W_2}\|$,
where $\Pi_{W_i}$ denotes the orthogonal projection onto~$W_i$, cf.~\cite[\S2.6]{golloan:83}. 
Clearly, this defines a metric and a corresponding topology on $\Gr_{n,m}$.

In~\cite[Def.~1.2]{ambu:11a} we made the following definition. 

\begin{definition}\label{def:C_G(W)} 
The \emph{Grassmann condition} with respect to the regular cone~$C\subseteq\R^n$ 
is defined as the function
\[ 
 \CG_C\colon \GR \to [1,\infty] ,\quad \CG_C(W) := \frac{1}{d_p(W,\SG_m(C))} \, ,
\]
where $d_p(W,\SG_m(C)):=\min\{ d_p(W,W') \mid W' \in\SG_m(C)\}$.
\end{definition}

Suppose that $W\in\GR$ is represented by $W=\im(A^T)$ 
with a matrix $A\in\IR^{m\times n}$ of full rank. 
Let $\kappa(A)$ denote the matrix condition number, i.e.,
the ratio between the largest and the smallest singular value of~$A$. 
In~\cite[Thm.~1.4]{ambu:11a} (see also~\cite{BF:09}),  
we proved the following basic relation between the 
Grassmann condition number $\CG(W)$ of $W$ 
and Renegar's condition number $\CR(A)$ of the matrix~$A$:
\begin{equation}\label{eq:C<=R<=k*C}
  \CG(W) \;\leq\; \CR(A) \;\leq\; \kappa(A)\, \CG(W) .
\end{equation}
This shows that the Grassmann condition number can be interpreted
as a coordinate-free version of Renegar's condition number. 
The condition of the matrix $A$ representing $W$ enters 
Renegar's condition number~$\CR(A)$, but $\CG(W)$ is independent 
of this representation. 

Our probabilistic analysis of the Grassmann condition crucially relies on 
a geometric interpretation of this quantity that we explain next. 
There is a natural Riemannian metric on the compact Grassmann manifold $\GR$ 
that is invariant under the action of $O(n)$, 
and which is uniquely determined up to a scaling factor~\cite{helga}. 
This induces an $O(n)$-invariant volume form on $\GR$, which 
allows to define the volume $\rvol B$ of Borel measurable subsets $B\subseteq\GR$. 
We assume that $\rvol\GR=1$. 

The \emph{geodesic distance} $d_g(W_1,W_2)$ between $W_1,W_2\in\Gr_{n,m}$ is defined as the minimum 
length of a piecewise smooth curve in $\Gr_{n,m}$ connecting $W_1$ with $W_2$. 
We remark that one can nicely express $d_g(W_1,W_2)$ in terms of the principle angles between these subspaces,
cf.~\cite[equation (8)]{ambu:11a}; in particular, the diameter of $\Gr_{n,m}$ equals $\pi/2$.

In~\cite[Thm.~1.8]{ambu:11a} we proved that 
$d_p(W,\SG_m) = \sin d_g(W,\SG_m)$ for $W\in\Gr_{n,m}$, where we write 
$\SG_m:=\SG_m(C)$ to simplify notation. 
This implies that 
the tube $\mcT(\SG_m,\alpha)$ of radius~$\alpha$ around $\SG_m$, 
defined as follows, 
\begin{equation}\label{eq:basic-tube-interpretation}
 \mcT(\SG_m,\alpha) := \big\{ W\in\GR \mid d_g(W,\SG_m) \le \alpha\big\} 
\end{equation}
equals the set of $W\in\GR$ having Grassmann condition at least $t:=(\sin\alpha)^{-1}$. 

Suppose now that $A\in \IR^{m\times n}$ is a standard Gaussian random matrix. 
Then $A$ almost surely has full rank and
$W=\im A^T$ is uniformly random in $\Gr_{n,m}$ with respect to 
the normalized volume measure~$\rvol$. 
The goal of this work is to prove upper bounds on 
$\Prob [\CG(A) \ge t]$ that hold for any regular cone~$C$, 
thus showing that it is unlikely that $\CG(W)$ is large. 
By~\eqref{eq:basic-tube-interpretation},
this means to bound the volume of the 
tubes around $\SG_m(C)$.

This task is very much in the spirit of the seminal paper~\cite{Demmel88} by Jim Demmel 
and its refinement in~\cite{BCL:08}.  These papers provide upper bounds 
on the volume of tubes around algebraic hypersurfaces in spheres.
The main tool is a general formula, due to Hermann Weyl~\cite{weyl:39},   
on the volume of tubes around smooth hypersurfaces of spheres. 
In Theorem~\ref{thm:tube-form} we will derive a similar, though considerably 
more complicated formula for the volume of the tubes 
around $\SG_m(C)$ in the Grassmann manifold. 
We think that this new tube formula is of independent mathematical interest. 
The proof is reduced to the case where the cone~$C$ has a smooth boundary, 
in which $\SG_m(C)$ turns out to be a smooth hypersurface in~$\GR$.  

An exact tube formula had been previously derived by Glasauer~\cite{Gl,Gl:Summ}
by measure theoretic techniques. 
However, Glasauer's result has the serious drawback that it only holds for radii 
below a certain critical value 
depending on the cone~$C$, and this value 
is zero 
for all cones of interest in convex programming; 
cf.~Remark~\ref{rem:convexity-tubes}.
By contrast, Theorem~\ref{thm:tube-form} states an upper bound on the 
volume of tubes that holds for any radius. 

\subsection{Main results I: probabilistic analysis}

Here is the main result of this paper. 

\begin{thm}\label{thm:estim-conefree}
Let $C\subseteq\IR^n$ be a regular cone with $n\geq 3$. 
If $A\in\IR^{m\times n}$ is a standard Gaussian random matrix, then we have
\begin{align*}
   \Prob [\CG(A)>t] & \;<\; 6 \sqrt{m(n-m)}\, \frac{1}{t} 
     ,\quad \text{if $t>n^{\frac{3}{2}}$} ,
\\ \IE\left[\ln\CG(A)\right] & \;<\; 1.5 \ln(n) + 1.5 \; .
\end{align*}
\end{thm}

In fact, these bounds hold for any probability distribution 
on~$\IR^{m\times n}$, that induces the uniform distribution 
on the Grassmann manifold $\Gr_{n,m}$ via the map 
$A\mapsto \im A^T$. 

Theorem~\ref{thm:estim-conefree} is of relevance for the probabilistic analysis of
algorithms in convex programming. 

\begin{corollary}\label{cor:app-alg}
Let $C$ be a self-scaled cone with a self-scaled barrier function.
Then there exists an interior point algorithm, that solves the general
homogeneous convex feasibility problem for Gaussian random inputs
in expected $O(\sqrt{\nu_C} (\ln\nu_C+\ln n))$ number 
of interior-point iterations. Here, $\nu_C$ denotes the 
complexity parameter of the barrier function for the 
reference cone $C$. For linear programming, second-order programming, 
and semidefinite programming, the expected number of 
interior-point iterations is $O(\sqrt{n}\ln n)$.
\end{corollary}

\begin{proof}
In~\cite{VRP:07} the authors describe an interior-point algorithm 
that solves the general homogeneous convex feasibility problem,
for a self-scaled cone~$C$ with a self-scaled barrier function, in
$O(\sqrt{\nu_C}\ln(\nu_C\CR(A)))$ interior-point iterations.

The typical barrier functions for (LP), (SOCP), and (SDP), respectively, 
have the following complexity parameter $\nu_C$: 
  \[ C=\IR_+^n \;:\; \nu_C = n ,\qquad 
     C=\mcL^{n_1}\times\ldots\times \mcL^{n_r} \;:\; \nu_C = 2\,r ,\qquad 
     C=\Sym^k_+ \;:\; \nu_C = k  . \]
In particular, $\nu_C\leq n$ in all these cases.

By \eqref{eq:C<=R<=k*C}, we have 
$\IE[\ln\CR(A)]  \leq \IE[\ln\kappa(A)] + \IE[\ln\CG(A)]$, 
hence a probabilistic analysis of~$\CR$ reduces to that of $\CG$ and $\kappa$. 
Theorem~\ref{thm:estim-conefree} states that $\IE[\ln\CG(A)]=O(\ln(n))$.
In~\cite{CD:05} it was shown that $\IE\left[\ln\kappa(A)\right] = O(\ln n)$
for Gaussian matrices $A\in\IR^{m\times n}$.
Combining these estimates, we obtain $\IE[\ln(\CR(A))]=O(\ln n)$.
We conclude that the expected number of interior point iterations of the algorithm 
in~\cite{VRP:07} is bounded by $O(\sqrt{\nu_C}\ln(\nu_C + \ln n))$.
\end{proof}

\begin{remark}
In~\cite{VRP:07} it is also shown
that the condition number of the system of equations, that is solved
in each interior-point iteration, is bounded by a factor of~$\CR(A)^2$.
Therefore, our results also imply bounds 
on the expected cost of each iteration in the above-mentioned
algorithm. 
\end{remark}

\subsection{Intrinsic volumes and Weyl's tube formula}\label{se:intro-intrinsic}

The analysis of the homogeneous convex feasibility problem
naturally finds its place in the domain of spherical convex geometry.
Indeed,  a linear subspace intersects a convex cone nontrivially if and only if  
the corresponding subsphere of the unit sphere intersects 
the corresponding spherically convex set. 
While Euclidean convex geometry is a classical and extensively studied subject,
the situation for spherical convex geometry is much less established. 
For more information, we refer to the theses \cite{Gl,Gl:Summ}, \cite{am:thesis}, 
the article~\cite{GHS}, and Section~6.5 in~\cite{SW:08}.

We call a subset $K\subseteq S^{n-1}$ \emph{(spherically) convex} if for all 
$p,q\in K$ with $q\neq\pm p$, the great circle segment between~$p$
and~$q$ is contained in $K$. This is equivalent to the requirement that 
$C:=\cone(K) := \{\lambda p \mid \lambda\geq 0,\; p\in K\}$
is a convex cone. Note that $K=C\cap S^{n-1}$. 
Let $d(p,q):=\arccos(\langle p,q\rangle)$ 
denote the (spherical) distance between points $p,q\in S^{n-1}$.
We are interested in the volume of the tube 
$\mcT(K,\alpha) := \big\{ p\in S^{n-1} \mid d(p,K) \leq \alpha \big\}$
of radius $\alpha$ around~$K$.
In general, unlike in Euclidean space, 
the tubes $\mcT(K,\alpha)$ are not convex, 
which causes technical difficulties; 
see Remark~\ref{rem:convexity-tubes}.

The volumes of the unit sphere $S^{n-1} :=\{x\in\IR^n \mid \|x\| =1\}$  
and the unit ball $B_n:=\{x\in \IR^n\mid  \|x\| \le 1\}$, respectively, 
are given by ($n\ge 1$)
\begin{equation}\label{eq:def-O_k}
   \mcO_{n-1}  := \vol_{n-1}(S^{n-1}) = \frac{2\pi^{n/2}}{\Gamma(n/2)} ,\qquad 
  \omega_n := \vol_n B_n = \frac{\mcO_{n-1}}{n} 
           = \frac{\pi^{\frac{n}{2}}}{\Gamma(\frac{n+2}{2})} 
\end{equation}   
(we also set $\omega_0:=1$). We define for $0\le k \le n-1$ and 
$0\leq\alpha\leq\frac{\pi}{2}$ the functions 
\begin{equation}\label{eq:def-I}
  I_{n,k}(\alpha) := \int_0^\alpha \cos(\rho)^k\cdot \sin(\rho)^{n-2-k} \, d\rho \; .
\end{equation}
If $S^k$ denotes a $k$-dimensional unit subsphere of $S^{n-1}$, then 
it is known that 
(cf.~\cite[Lem.~20.5]{Condition}) 
\begin{equation}\label{eq:O_(n-1,k)(alpha)}
  \mcO_{n-1,k}(\alpha) :=\vol_{n-1} \mcT(S^k,\alpha) = 
  \mcO_k\cdot\mcO_{n-2-k}\cdot I_{n,k}(\alpha) \; .
\end{equation}
More generally, the volume of $\mcT(K,\alpha)$ can be expressed in terms of certain quantities 
assigned to $K$ that we describe next. 

Let $C\subseteq\IR^n$ be a polyhedral cone. For $j=0,1,\ldots,n$, we denote by 
${\cal F}_j$ the set of the relative interiors of the $j$-dimensional faces of~$C$.
Moreover, we consider the canonical projection 
$\Pi_C\colon\IR^d\to C$, $x\mapsto \mathrm{argmin}\{\|x-y\|\mid y\in C\}$. 
Then the {\em $j$-intrinsic volume} $V_j(C)$ of~$C$ is defined by 
\begin{equation}\label{eq:V_j(C)-polyhdrl}
  V_j(C) := \sum_{F\in{\cal F}_j} \; \underset{x\in N(0,I_d)}{\Prob}\big\{\Pi_C(x)\in F\big\} ,
\end{equation}
where $N(0,I_n)$ stands for the standard Gaussian distribution on $\IR^n$.
In particular, we have $V_n(C)=\rvol(C\cap S^{n-1})$ and $V_0(C)=\rvol(\breve{C}\cap S^{n-1})$.
It is easy to check that $V_j(\R^n_+)= \binom{n}{j} 2^{-n}$. 

By continuous extension, one can assign intrinsic volumes $V_0(C),\ldots,V_n(C)$ 
to any closed convex cone $C\subseteq\IR^n$, see Section~\ref{se:intr-vol} 
and \cite{ambu:11c}.  
If $K\subseteq S^{n-1}$ is a closed spherically convex subset, we define 
$V_{j-1}(K) := V_j(C)$, where $C:=\cone(K)$
(the index change has certain advantages).

The following result is essentially due to Weyl~\cite{weyl:39} 
(compare~\cite{GHS} and~\cite{SW:08} for more information):
for $0\leq\alpha\leq\frac{\pi}{2}$ we have 
\begin{equation}\label{eq:def-V_j}
  \vol_{n-1}\mcT(K,\alpha) = \vol_{n-1}(K) + \sum_{j=0}^{n-2} V_j(K)
                                          \cdot \mcO_{n-1,j}(\alpha) \; .
\end{equation}
From \eqref{eq:V_j(C)-polyhdrl} one can derive that for closed convex cones $C_1$ and $C_2$
we have  
\begin{equation}\label{eq:Vj-conv}
V_j(C_1\times C_2) = \sum_{i=0}^j V_i(C_1) V_{j-i}(C_2),
\end{equation}
a fact that apparently was first observed in \cite{am:thesis}.

\subsection{Main results II: Grassmannian tube formula} 

Let us return to the situation of Section~\ref{se:intro-Grassmann}. 
In order to bound the volume of the tube $\mcT(\SG_m(C),\alpha)$, 
we need to introduce some notation. 

Using the analytic extension of the binomial coefficients  
$\binom{x}{y}:=\frac{\Gamma(x+1)}{\Gamma(y+1)\cdot \Gamma(x-y+1)}$, 
for $x>-1$ and $-1<y<x+1$, we have 
\begin{equation}\label{eq:binom(n/2,m/2)=...}
     \binom{n/2}{m/2} = 
    \frac{\Gamma(\frac{n+2}{2})}
              {\Gamma(\frac{m+2}{2})\cdot\Gamma(\frac{n-m+2}{2})} =
  \frac{\omega_m\cdot \omega_{n-m}}{\omega_n} \; .
  \end{equation}
We also define the {\em flag coefficients} (cf.~\cite{KR:97})
\begin{equation}\label{eq:def-flcoeff}
 \rbinom{n}{m} := \binom{n}{m}\binom{n/2}{m/2}^{-1} 
                            = \frac{\sqrt{\pi}\cdot \Gamma(\frac{n+1}{2})}
                                {\Gamma(\frac{m+1}{2})\cdot \Gamma(\frac{n-m+1}{2})} \; .
\end{equation}

We fix a regular cone $C\subseteq\IR^n$ and put $\SG_m:=\SG_m(C)$. 
Recall the definition of the tube $\mcT(\Sigma_m,\alpha)$ from \eqref{eq:basic-tube-interpretation}.  
We define the \emph{primal} and the \emph{dual tube around $\Sigma_m$}, respectively, by 
\begin{equation}\label{eq:TP-TD}
  \TP(\Sigma_m,\alpha) := \mcT(\Sigma_m,\alpha)\cap \PG_m(C) \;,\qquad \TD(\Sigma_m,\alpha) 
                       := \mcT(\Sigma_m,\alpha)\cap \DG_m(C) \; .
\end{equation}
The following tube formula is the technical heart of our work. 
In fact, Theorem~\ref{thm:estim-conefree} will follow by bounding the right-hand side in 
this formula, taking into account that $V_j(C) \le \frac12$. 

\begin{thm}\label{thm:tube-form}
Let $C\subseteq\IR^n$ be a regular cone. Then, for $1\leq m < n$ 
and $0\leq\alpha\leq\frac{\pi}{2}$, we have 
\begin{equation}\label{eq:vol-tube-Sigma-est-nice}
  \rvol \, \TP(\Sigma_m(C),\alpha) \ \leq\  \frac{2 m (n-m)}{n} \binom{n/2}{m/2}\; 
  \sum_{j=0}^{n-2} V_{j+1}(C) \rbinom{n-2}{j} \;
  \sum_{i=0}^{n-2} d_{ij}^{nm}  I_{n,i}(\alpha) \, ,
\end{equation}
where the constants $d_{ij}^{nm}$ are defined for 
$i+j+m\equiv 1 \pmod 2$, 
$0\leq \tfrac{i-j}{2}+\tfrac{m-1}{2}\leq m-1$, and 
$0\leq \tfrac{i+j}{2}-\tfrac{m-1}{2}\leq n-m-1$, by
\begin{equation}\label{eq:form-d_(ij)}
  d_{ij}^{nm} \;:=\; \frac{\binom{m-1}{\frac{i-j}{2}+\frac{m-1}{2}}\cdot 
  \binom{n-m-1}{\frac{i+j}{2}-\frac{m-1}{2}}}{\binom{n-2}{j}} \; ;
\end{equation}
and defined by $d_{ij}^{nm}:=0$ otherwise. 
(See Table~\ref{tab:coeffs-D}.)

The same upper bound holds for the volume of 
$\TD(\Sigma_m,\alpha)$.  
\end{thm}

\begin{remark}\label{re:comm-tube-formula}
\begin{enumerate}
\item It can be shown (see~\cite{am:thesis} for details) 
that~\eqref{eq:vol-tube-Sigma-est-nice} in fact holds 
\emph{with equality} if $\mcT(C\cap S^{n-1},\alpha)$
is convex and if $d_{ij}^{nm}$ is replaced by 
$(-1)^{\frac{i-j}{2}-\frac{m-1}{2}}\cdot d_{ij}^{nm}$.
This reveals that the estimate in Theorem~\ref{thm:tube-form} 
is close to being sharp.

\item In the case $m=1$ we have $\Gr_{n,1} = \mathbb{P}^{n-1}$ and  
$\TP(\SG_1,\alpha)$ is the image of $\mcT(K,\alpha) \setminus K$ 
under the canonical map $S^{n-1}\to\mathbb{P}^{n-1}=\Gr_{n,1}$. 
Hence 
$\rvol\TP(\SG_1,\alpha) = (\vol_{n-1} \mcT(K,\alpha) - \vol_{n-1} K)/(\mcO_{n-1}/2)$. 
Taking into account that $d_{ij}^{n1} = \delta_{ij}$, 
one can check that~\eqref{eq:vol-tube-Sigma-est-nice}
actually gives the expression in Weyl's tube formula~\eqref{eq:def-V_j}.

\item The symmetry relations 
$d_{i,n-2-j}^{n,n-m} = d_{ij}^{nm}$ and 
$d_{n-2-i,n-2-j}^{nm} = d_{ij}^{nm}$ hold.  

\item The bound~\eqref{eq:vol-tube-Sigma-est-nice} has a remarkable symmetry. 
The involution $\iota\colon \Gr_{n,m}\to\Gr_{n,n-m}$, $\mcW\mapsto\mcW^\bot$ 
maps $\PG_m(C)$ to $\DG_{n-m}(\breve{C})$, and therefore
$\iota(\Sigma_m(C))=\Sigma_{n-m}(\breve{C})$, cf.~\cite{ambu:11a}.
In particular, $\iota$ maps the dual tube 
$\mcT^\Dua(\Sigma_m(C),\alpha)$ of $C$ to the primal tube 
$\mcT^\Pri(\Sigma_{n-m}(\breve{C}),\alpha)$ of~$\breve{C}$,  
and so they must have the same volume. 
This symmetry is reflected in~\eqref{eq:vol-tube-Sigma-est-nice}. 
Indeed, if we simplify the upper bound for 
$\rvol \mcT^\Pri(\Sigma_{n-m}(\breve{C}),\alpha)$ 
in~\eqref{eq:vol-tube-Sigma-est-nice}  
by using the duality property $V_{j+1}(\breve{C})=V_{n-2-j+1}(C)$ 
(cf.\ Proposition~\ref{prop:facts-intrvol}),  
by changing the summation via $j\leftarrow n-2-j$, 
and by using the symmetry relations $d_{i,n-2-j}^{n,n-m} = d_{ij}^{nm}$,  
$\binom{n/2}{(n-m)/2} = \binom{n/2}{m/2}$, and  
$\rbinom{n-2}{n-2-j} = \rbinom{n-2}{j}$, 
then we end up with the upper bound for 
$\rvol \mcT^\Pri(\Sigma_{m}(C),\alpha)$ from~\eqref{eq:vol-tube-Sigma-est-nice}.

\end{enumerate}
\end{remark}

\begin{table}[h]
\def\abstand{0.1}
\def\myX{\rule{1mm}{0mm}}
\begin{align*}
     D_{7,1} & = \left(\begin{smallmatrix}
                        1 &  &  &  &  &   \\[1mm]
                          & 1 &  &  &  &  \\[1mm]
                          &  & 1 &  &  &  \\[1mm]
                          &  &  & 1 &  &  \\[1mm]
                          &  &  &  & 1 &  \\[1mm]
                          &  &  &  &  & 1 \end{smallmatrix}\right) \; , &
     D_{7,2} & = \left(\begin{smallmatrix}
                          & \frac{1}{5} \\[\abstand mm]
                         1 & 0 & \frac{2}{5} \\[\abstand mm]
                          & \myX\frac{4}{5}\myX & 0 & \frac{3}{5} \\[\abstand mm]
                          &  & \myX\frac{3}{5}\myX & 0 & \frac{4}{5} \\[\abstand mm]
                          &  &  & \myX\frac{2}{5}\myX & 0 & 1 \\[\abstand mm]
                          &  &  &  & \myX\frac{1}{5}\myX \end{smallmatrix}\right) \; , &
     D_{7,3} & = \left(\begin{smallmatrix}
                          &  & \frac{1}{10} \\[\abstand mm]
                          & \myX\frac{2}{5}\myX & 0 & \frac{3}{10} \\[\abstand mm]
                         1 & 0 & \myX\frac{3}{5}\myX & 0 & \frac{3}{5} \\[\abstand mm]
                          & \frac{3}{5} & 0 & \myX\frac{3}{5}\myX & 0 & 1 \\[\abstand mm]
                          &  & \frac{3}{10} & 0 & \myX\frac{2}{5}\myX \\[\abstand mm]
                          &  &  & \frac{1}{10} \end{smallmatrix}\right) \; ,
\\[5mm]
     D_{7,6} & = \left(\begin{smallmatrix}
                         &  &  &  &  & 1 \\[1mm]
                         &  &  &  & 1 &  \\[1mm]
                         &  &  & 1 &  &  \\[1mm]
                         &  & 1 &  &  &  \\[1mm]
                         & 1 &  &  &  &  \\[1mm]
                        1 &  &  &  &  &  \end{smallmatrix}\right) \; , &
     D_{7,5} & = \left(\begin{smallmatrix}
                          &  &  &  & \myX\frac{1}{5}\myX \\[\abstand mm]
                          &  &  & \myX\frac{2}{5}\myX & 0 & 1 \\[\abstand mm]
                          &  & \myX\frac{3}{5}\myX & 0 & \frac{4}{5} \\[\abstand mm]
                          & \myX\frac{4}{5}\myX & 0 & \frac{3}{5} \\[\abstand mm]
                         1 & 0 & \frac{2}{5} \\[\abstand mm]
                          & \frac{1}{5} \end{smallmatrix}\right) \; , &
     D_{7,4} & = \left(\begin{smallmatrix}
                          &  &  & \frac{1}{10} \\[\abstand mm]
                          &  & \frac{3}{10} & 0 & \myX\frac{2}{5}\myX \\[\abstand mm]
                          & \frac{3}{5} & 0 & \myX\frac{3}{5}\myX & 0 & 1 \\[\abstand mm]
                         1 & 0 & \myX\frac{3}{5}\myX & 0 & \frac{3}{5} \\[\abstand mm]
                          & \myX\frac{2}{5}\myX & 0 & \frac{3}{10} \\[\abstand mm]
                          &  & \frac{1}{10} \end{smallmatrix}\right) \; .
\end{align*}
  \caption{The coefficient matrices $D_{n,m}=(d_{ij}^{nm})_{i,j=0,\ldots,n-2}$ for $n=7$, $m=1,\ldots, 6$.}
  \label{tab:coeffs-D}
\end{table}

\subsection{Basic outline for proof of tube formula}\label{se:basic-outline} 

The basic idea of the proof relies on some Riemannian geometry. 
(We refer to~\cite{dC},~\cite{booth}, or~\cite[Ch.~1]{Chav} for some general background on Riemannian geometry.) 

Let $\M$ be a compact connected Riemannian manifold with distance metric~$d$.
The distance $d(p,\mcH)$ of a point $p\in\M$ to a nonempty closed subset $\mcH$ of $\M$ 
is defined as the minimum of $d(p,q)$ over all $q \in \mcH$. 
For $\alpha\ge 0$ consider the $\alpha$-neighborhood
$$
 \mcT(\mcH,\alpha) := \big\{ p\in \M\mid d(p,\mcH) \leq \alpha \big\}
$$
of $\mcH$ in $\M$.
We also call $\mcT(\mcH,\alpha)$ the \emph{tube of radius $\alpha$ around $\mcH$}. 
It is essential that the tube can be described in terms of the exponential map of~$\M$
if $\mcH$ is a smooth submanifold of~$\M$. 
For stating this, let $T\M$ denote the \emph{tangent bundle} of~$\M$, 
i.e., the disjoint union of all of its tangent spaces. 
The \emph{exponential map} of $\M$ is characterized as the map~$\exp\colon T\M\to \M$ 
such that for $p\in \M$ and $v\in T_p\M$, the curve
$\gamma\colon \IR\to \M, t\mapsto\exp_p(tv)$ 
is the \emph{geodesic} through $p$ in direction~$v$, that is, 
$\gamma(0)=p$ and $\dot{\gamma}(0)=v$.

We will only consider a special kind of geodesics in the Grassmann manifold $\Gr_{n,m}$,
that have a simple geometric meaning: 
Let $W\in\Gr_{n,m}$, and consider a two-dimensional linear subspace $E\subseteq\IR^n$ 
such that $\dim(E\cap W)=\dim(E\cap W^\bot)=1$.
Letting $R_t\in O(n)$ denote the rotation in $E$ 
by the angle~$t$ that keeps the vectors in $E^\perp$ fixed, 
the map $t\mapsto R_t(W)$ is a geodesic through~$W$.

Suppose now that $\mcH\subseteq\M$ is a compact hypersurface in~$\M$ 
with unit normal vector field~$\nu$. Consider the following smooth map 
induced by $\exp$: 
\begin{equation}\label{eq:psi-par}
 \psi\colon\mcH\times\R\to \M,\ (p,\theta)\mapsto \exp_p(\theta \nu(p)) .
\end{equation}
It is a well-known fact that 
\begin{equation}\label{eq:tube-char-exp}
 \mcT(\mcH,\alpha) = \psi\big(\mcH\times [-\alpha,\alpha]\big) \; .
\end{equation}
(For a proof, see~\cite[Addendum to Chap.~9, proof of Thm.~20]{spiv1}.)
Combining~\eqref{eq:tube-char-exp} with the coarea formula 
(compare~\eqref{eq:cor-coarea-Riem-2}), 
we obtain 
\begin{equation}\label{eq:basic-vol-inequ}
  \vol \mcT(\mcH,\alpha) \ \le\  \int_{(p,\theta)\in\mcH\times [-\alpha,\alpha]} |\det D_{(p,\theta)}\psi |\; d(\mcH\times\IR) .
\end{equation}
This inequality is the basis of the proof of Theorem~\ref{thm:tube-form}.
The main difficulty is to make effective use of the right-hand side in the specific situation at hand. 

In the case where the cone $C\subseteq\IR^n$ is such that 
$K:=C\cap S^{n-1}$ has a smooth boundary $M:=\partial K$ with positive curvature, 
we will prove that $\mcH=\SG_m(C)$ is a smooth hypersurface in~$\M=\GR$ and 
that $\mcH$ has a unit normal vector field 
(Proposition~\ref{pro:Sigma}). 
Since any $W\in\Sigma_m(C)$ intersects $K$ in 
a unique point $p\in M$ (cf.\ Lemma~\ref{prop:WcapK=p}), 
this defines a map
$\PiM \colon \SG_m(C)  \to M,\, W\mapsto p$. 
We will prove that $\PiM$ allows us to interpret $\SG_m(C)$ 
as the Grassmann bundle $\Gr(M,m-1)$ over $M$.
The most difficult part of the proof of Theorem~\ref{thm:tube-form}
is to understand the normal Jacobian $|\det D\psi|$ of the parameterization map~$\psi$ in our specific situation;  
see Theorem~\ref{thm:norm-Jac}. 
It turns out that the normal Jacobian is a certain subspace-dependent version of the characteristic polynomial 
of the Weingarten map of~$M$, 
for which we coined the name twisted characteristic polynomal
(cf.\ Section~\ref{sec:tw-charpol}).

\subsection{Main results III: improved probability estimates}

Assuming certain conjectures on the growth of the intrinsic volumes
of special cones~$C$, we can considerably improve the bounds in 
Theorem~\ref{thm:estim-conefree}.

In Section~\ref{se:intr-vol} we compute the intrinsic volumes 
\begin{equation}\label{eq:def-f_j(n)}
 f_j(n) := V_{j}(\mcL^n) = 2^{-n} \binom{\frac{n-2}{2}}{\frac{j-1}{2}}  \quad 
 \mbox{for $1\le j \le n-1$}  
\end{equation}
of the $n$-dimensional {\em Lorentz cone}
$$
 \mcL^n := \{x\in\IR^n\mid x_n\geq (x_1^2 + \cdots + \ldots,x_{n-1}^2)^{1/2}\} \; .
$$ 
(See \eqref{eq:f_0} 
for a formula for $f_n(n)=f_0(n)$.) 
Note that the sequence $f(n):=(f_0(n),\ldots,f_n(n))$ is symmetric, i.e., $f_{n-j}(n)=f_j(n)$, 
since $\mcL^n$ is self-dual. 

It will be convenient to compare the intrinsic volumes of a self-dual cone $C\subseteq\IR^n$ 
with the intrinsic volumes of the Lorentz cone $\mcL^n$. 
We thus define the \emph{excess~$\exc(C)$ of $C$ over the Lorentz cone} as
\begin{equation}\label{eq:def-v(C)}
  \exc(C) := \min_{0\leq j\leq n} \frac{V_{j}(C)}{f_j(n)} \; .
\end{equation}
In other words, $\exc(C)$ is the smallest constant such that the inequality 
$V_j(C) \leq \exc(C) f_j(n)$ is satisfied for all~$j$. 
Clearly, $\exc(\mcL^n) = 1$ by definition.
One can check that $\exc(\R^n_+) < \sqrt{2}$ and 
$\lim_{n\to\infty} \exc(\R^n_+) =1$.

Based on experiments, we set up the following conjecture. 

\begin{conjecture}\label{conj:v(C_1timesC_2)} 
The convolution of $f(n_1)$ with $f(n_2)$ satisfies 
$f(n_1) * f(n_2) \le 2f(n_1+n_2)$ for $n_1,n_2\ge 2$. 
\end{conjecture}

From this conjecture and~\eqref{eq:Vj-conv} 
it easily follows that $\exc(C_1\times C_2) \leq 2\, \exc(C_1) \exc(C_2)$ 
for closed convex cones $C_1,C_2$. 
Therefore, 
$\exc(\mcL^{n_1}\times\ldots\times\mcL^{n_r}) \le 2^{r-1}$.
For the cone of positive semidefinite matrices, we make the following conjecture;
see \cite{ambu:11c} for motivation and experiments. 

\begin{conjecture}\label{conj:est-SDP}
We have  $\exc(\Sym^k_+) < 2$ for the cone~$\Sym^k_+$ of positive semidefinite matrices.
\end{conjecture}

The following theorem refines Theorem~\ref{thm:estim-conefree} in terms of the excess of the Lorentz cone.
We remark that, conditional on the above conjectures, the obtained
bounds on the condition number are \emph{independent of the dimension~$n$ of the ambient space},  
for second-order cone programming 
($C=\mcL^{n_1}\times\ldots\times\mcL^{n_r}$ with $r$ fixed ),
and for semidefinite programming, respectively. 

\begin{thm}\label{thm:estim-conedep}
Let $C\subseteq\IR^n$ be a self-dual cone.
If $A\in\IR^{m\times n}$, with $m\geq8$, is a standard Gaussian random matrix, 
then we have
\begin{align*}
   \Prob [\CG(A)>t] & \;<\; 20 \, \exc(C)\,\sqrt{m}\,\frac{1}{t}
      ,\quad \text{if $t>m$}  ,
\\ \IE\left[\ln\CG(A)\right] & \;\;<\;\; \ln m + \max\{\ln \exc(C),0\} + 3 .
\end{align*}
\end{thm}

To conclude, let us point out that the shape of the distribution of the intrinsic volumes 
of a convex cone is currently a subject with a wealth of open questions. 
We make the following conjecture, that can be considered 
a spherical analog of the \emph{Alexandrov-Fenchel inequality}; 
cf.~\cite{Stanley, Schn:book}.

\begin{conjecture}\label{conj:log-conc}
For every closed convex cone $C\subseteq \IR^n$, 
the sequence $V(C)$ 
of its intrinsic volumes is
\emph{log-concave}, i.e., $V_j(C)^2\geq V_{j-1}(C)\, V_{j+1}(C)$
for $1\leq j < n$.
\end{conjecture}

We have verified this conjecture for products of circular cones
(in particular for the positive orthant and for products of Lorentz cones); see~\cite{am:thesis, ambu:11c} for more information.

\subsection{Outline of paper}

We recall some basic facts from Riemannian geometry in Section~\ref{se:prelim}.  
Section~\ref{sec:spher-conv-geom} is a short review of spherical convex geometry,  
with particular emphasis on spherical intrinsic volumes.
In Section~\ref{se:tube-form} we derive the tube formula in Theorem~\ref{thm:tube-form}
from a main technical lemma about the normal Jacobian of the exponential map that parametrizes the tube around $\SG_m(C)$,
the proof of which we defer to Section~\ref{sec:proof-tube-form}.
Along the way, we introduce in Section~\ref{se:tube-form} an algebraic object of possible independent interest: 
to an endomorphism $\vp$ of a Euclidean vector space $V$
together with a linear subspace~$Y\subseteq V$, we assign 
its {\em twisted characteristic polynomial} $\ch_Y(\vp,t)$. 
For the proof of the tube formula we will use the fact that 
the expectation of $\ch_Y(\vp,t)$, 
with respect to a randomly chosed subspace~$Y$, 
is an explicit function of the characteristic polynomial of~$\vp$. 
In the interest of clarity, we defer the proof of this last fact to the appendix.

In Section~\ref{se:proof_main_re}, based on the tube formula, 
we provide the proofs of the main results 
(Theorem~\ref{thm:estim-conefree} and Theorem~\ref{thm:estim-conedep}), 
which provide the desired average-case analyses of the Grassmann condition number. 

The goal of the last part of the paper is to provide the proofs 
that have been left open in Section~\ref{se:tube-form}. 
For this, a good understanding of the metric properties of Grassmann manifolds, 
is required:  Section~\ref{sec:prelim-Riemgeom} presents the necessary background. 
Finally, the remaining proofs are provided, the most delicate being the one of 
Theorem~\ref{thm:norm-Jac} that expresses the normal Jacobian of the parameterization map 
of the tube around $\SG_m(C)$ in terms of twisted characteristic polynomials 
of the Weingarten maps of the boundary of the cone~$C$. 

\medskip

\noindent{\bf Acknowledgments.} 
We are grateful to one of the anonymous referees for criticism that has led to a substantial improvement of the paper's presentation.

\section{Differential geometric preliminaries}\label{se:prelim}

\subsection{Coarea formula}

We define the \emph{normal determinant} of a surjective linear map
$A\colon V\to W$ between Euclidean vector spaces $V$ and $W$ by 
\[ 
\ndet(A) := |\det(A|_{\ker(A)^\bot})| \; , 
\]
where $A|_{\ker(A)^\bot}\colon \ker(A)^\bot\to W$ denotes the restriction of $A$ to the orthogonal complement of the kernel of $A$. 
Obviously, if $A$ is bijective, then $\ndet(A)=|\det(A)|$. 
Thus the normal determinant provides a natural generalization of the absolute value of the determinant.

The {\em smooth coarea formula}, stated below, is our main tool for volume computations.

\begin{thm}
Let $\vp\colon \M_1\to \M_2$ be a smooth surjective map between Riemannian 
manifolds~$\M_1,\M_2$. Then for any $f\colon \M_1\to\IR$ that is integrable 
with respect to~$d\M_1$, we have
\begin{align}
   \underset{\M_1}{\int} f\,d\M_1 & = \underset{q\in \M_2}{\int}\;\underset{p\in \vp^{-1}(q)}{\int} 
                \frac{f}{\ndet(D_p\vp)} \,d\vp^{-1}(q)\,d\M_2 \; . 
                \label{eq:cor-coarea-Riem-1} \\
\intertext{If additionally $\dim \M_1=\dim \M_2$, then}
   \vol \M_2 := \underset{\M_2}{\int} \,d\M_2 & \leq \underset{q\in \M_2}{\int} \#\vp^{-1}(q) \,d\M_2 = 
                \underset{p\in \M_1}{\int} \ndet(D_p\vp) \,d\M_1 \; , 
                \label{eq:cor-coarea-Riem-2}
\end{align}
where $\#\vp^{-1}(q)$ denotes the number of elements in the fiber $\vp^{-1}(q)$. \hfill $\Box$
\end{thm}

See~\cite[3.8]{Mor:95} or~\cite[3.2.11]{Fed} for proofs of the coarea formula where $\M_1,\M_2$ 
are submanifolds of Euclidean space. 
For a proof of the coarea formula in the above stated form, see~\cite[Appendix]{H:93}. 
One calls $\ndet(D_p\vp)$ the \emph{normal Jacobian} of~$\vp$ at~$p$.

\begin{remark}
The inner integral in~\eqref{eq:cor-coarea-Riem-1} over the fiber
$\vp^{-1}(q)$ is well-defined for almost all $q\in \M_2$, 
which can be seen as follows: 
Sard's lemma (cf.~\cite[Thm.~3-14]{spiv:65}) implies that almost all
$q\in \M_2$ are regular values, 
i.e., the derivative $D_p\vp$ has full rank for all
$p\in\vp^{-1}(q)$. The fibers $\vp^{-1}(q)$ of regular values $q$ are
smooth submanifolds of $\M_1$ 
and therefore the integral over $\vp^{-1}(q)$ is well-defined.
\end{remark}

\subsection{Exponential maps and tubes}\label{se:exp-map-tubes}

Suppose we are in the situation of Section~\ref{se:basic-outline}.  
Thus let $\M$ be a compact connected Riemannian manifold and 
$\mcH\subseteq\M$ be a compact hypersurface with unit normal vector field~$\nu$.
Recall the map $\psi\colon\mcH\times\R\to \M$ from \eqref{eq:psi-par}, 
which satisfies $\mcT(\mcH,\alpha) = \psi(\M\times [-\alpha,\alpha])$
according to~\eqref{eq:tube-char-exp}. 

We additionally assume that there are 
closed subsets $\PG,\DG\subseteq\M$ such that 
$\M=\PG\cup\DG$ and $\mcH=\PG\cap\DG$. 
Moreover, we assume that 
there exists $\epsilon>0$ such that 
$\psi(\mcH\times [0,\epsilon]) \subseteq \PG$ and 
$\psi(\mcH\times [-\epsilon,0]) \subseteq \DG$.
(We express this property by saying that 
{\em $\nu$ points into $\PG$} and 
{\em $-\nu$ points into $\DG$}.) 
In Proposition~\ref{pro:Sigma} we will show that 
$\PG_m(C)$ and $\DG_m(C)$ satisfy these assumptions. 

\begin{lemma}\label{le:PD-tubes}
We have 
$$
\mcT(\mcH,\alpha) \cap \PG \subseteq \psi(\mcH\times [0,\alpha])\quad
\mbox{ and }\quad 
\mcT(\mcH,\alpha) \cap \DG \subseteq \psi(\mcH\times [-\alpha,0])
$$ 
for all $\alpha\ge 0$.
\end{lemma}

\begin{proof}
By symmetry, it suffices to prove the first inclusion. 
Let $q\in\PG$ with $\alpha :=d(q,\mcH)>0$. 
From~\eqref{eq:tube-char-exp} 
we get that a minimum length geodesic
between some $p\in\mcH$ and~$q$ 
has the form 
$$
 \gamma\colon [0,\alpha]\to\M,\, \rho\mapsto \exp_p(\rho\,\delta\,\nu(p)) \; ,
$$
where $\delta\in\{1,-1\}$. 
It is sufficient to prove that $\delta=1$. 
By way of contradiction, assume that $\delta=-1$. 
Then $q=\gamma(\alpha)=\exp_p(-\alpha\nu(p))\in\PG$. 
Since by our assumption, we have 
$\gamma(\rho)\in\DG$ for sufficiently small $0<\rho<\alpha$, 
there exists $0 \le \alpha_0<\alpha$ 
such that $\gamma(\alpha_0)\in\PG\cap \DG=\mcH$. 
Therefore, $d(q,\mcH)\leq d(\gamma(\alpha),\gamma(\alpha_0))\leq \alpha-\alpha_0<\alpha$, which is a contradiction.
\end{proof}

\subsection{Weingarten maps}

We recall a further basic notion from differential geometry (cf.~\cite[Ch.~9]{thor:94}).
Let $M$ be a (smooth) hypersurface of $S^{n-1}$ with a unit normal vector field~$\nu$. 
The {\em Weingarten map} $\W_p$ at $p\in M$ is defined as 
\begin{equation}\label{eq:weing-hypersurf}
  \W_p\colon T_pM\to T_pM ,\quad \W_p(\zeta) = -D_p\nu(\zeta) \; ,
\end{equation}
where $D_p\nu$ denotes the derivative of $\nu\colon M\to S^{n-1}$ at $p$.
It can be shown that the Weingarten map is self-adjoint; we denote its eigenvalues by 
$\kappa_1(p),\ldots,\kappa_{n-2}(p)$. These are called the 
\emph{principal curvatures of $M$ at $p$}. Furthermore, we denote by 
$\sigma_k(p)$ the $k$th elementary symmetric function 
in the eigenvalues of~$\W_p$. We call the product 
$\sigma_{n-2}(p) := \kappa_1(p)\cdots\kappa_{n-2}(p)$ the {\em Gaussian curvature} of $M$ at $p$.

\subsection{Volume of Grassmann manifold}

Let $O(n):=\{Q\in\IR^{n\times n} \mid QQ^T=I_n\}$ denote the orthogonal group. 
Recall that $\mcO_i$ denotes the volume of $S^{i-1}$, 
cf.~\eqref{eq:def-O_k}. The following result is well-known. 

\begin{lemma}\label{le:volGr}
We have 
$$
\vol O(n) = \prod_{i=0}^{n-1} \mcO_i \quad\mbox{ and }\quad 
\vol \Gr_{n,m} = \frac{\vol O(n)}{\vol O(m)\cdot \vol O(n-m)}  \;.
$$
\end{lemma}

\begin{proof}
The map $\vp\colon O(n)\to S^{n-1}$, $Q\mapsto Q e_1$, 
where $e_1\in\IR^n$ denotes the first canonical basis vector, is a Riemannian 
submersion. In particular, we have $\ndet(D_Q\vp)=1$ for all $Q\in O(n)$.
Moreover, each fiber of~$\vp$ is isometric to $O(n-1)$. So an application 
of the coarea formula~\eqref{eq:cor-coarea-Riem-1} shows that 
$\vol O(n) = \mcO_{n-1}\cdot \vol O(n-1)$ and 
the formula for $\vol O(n)$ follows by induction.
 
For $\Gr_{n,m}$, it suffices to note that we have a Riemannian 
submersion $\Pi\colon O(n)\to\Gr_{n,m}$, whose fibers are isometric to the 
direct product $O(m)\times O(n-m)$. 
\end{proof}

For later use, we record the following 
\begin{align}\label{eq:vol(Gr_(n-2,m-1))/vol(Gr_(n,m))=...}
  \frac{\vol\Gr_{n-2,m-1}}{\vol\Gr_{n,m}} & = 
  \frac{\prod_{i=n-m-1}^{n-3}\mcO_i}{\prod_{i=0}^{m-2}\mcO_i}\cdot 
  \frac{\prod_{i=0}^{m-1}\mcO_i}{\prod_{i=n-m}^{n-1}\mcO_i} = 
  \frac{\mcO_{m-1}\cdot \mcO_{n-m-1}}{\mcO_{n-2}\cdot \mcO_{n-1}} 
  \stackrel{\eqref{eq:def-O_k}}{=} \frac{m(n-m)}{n}\cdot 
  \frac{\omega_m\cdot \omega_{n-m}}{\mcO_{n-2}\cdot \omega_n} \nonumber
\\ & \hspace{-1mm}\stackrel{\eqref{eq:binom(n/2,m/2)=...}}{=} 
  \frac{m(n-m)}{n}\cdot \binom{n/2}{m/2}\cdot \frac{1}{\mcO_{n-2}} \; .
\end{align}

\section{Background from spherical convex geometry}\label{sec:spher-conv-geom}

This section extends the background from spherical convex geometry given in Section~\ref{se:intro-intrinsic} with special emphasis on spherically convex sets with smooth boundary.

\subsection{The metric space of spherically convex sets}\label{sec:some-general}

Recall that a subset $K\subseteq S^{n-1}$ is called (spherically) convex iff 
$C:=\cone(K)$ is a convex cone. We have $K=C\cap S^{n-1}$. 
We denote the family of closed spherically convex sets by $\mcK(S^{n-1})$. 
The {\em duality map}, which sends a closed convex cone $C$ to its polar cone~$\breve{C}$,  
naturally defines an involution on $\mcK(S^{n-1})$: 
it maps $K\in\mcK(S^{n-1})$ to 
$\breve{K}:=\breve{C}\cap S^{n-1}$. 
It is easy to check  that 
$\breve{K}=\{p\in S^{n-1}\mid d(p,K)\geq\frac{\pi}{2}\}$, 
where $d(p,q)$ denotes the (spherical) distance between $p,q\in S^{n-1}$. 

It is easily seen that a closed convex cone is regular if and only if both $C$ and $\breve{C}$ have nonempty interior.
We say that $K\in\mcK(S^{n-1})$ is regular if $\cone(K)$ is regular and denote by 
$\mcK^r(S^{n-1})$ the set of regular $K\in\mcK(S^{n-1})$.

The set $\mcK(S^{n-1})$ is a compact metric space with respect to the 
{\em Hausdorff metric}~$d_\hd$, which for $K_1,K_2\in\mcK(S^{n-1})$ is defined by  
\[ d_\hd(K_1,K_2) := \max\big\{\min\{ \alpha\geq 0\mid K_2\subseteq 
     \mcT(K_1,\alpha)\} \,,\; \min\{ \beta\geq 0\mid K_1\subseteq 
     \mcT(K_2,\beta)\}\big\} \; ; 
\]
cf.~\cite[\S1.2]{M:06}.  
A spherically convex set $K\in\mcK(S^{n-1})$ is called \emph{polyhedral} 
if $\cone(K)$ is the intersection of finitely many closed half-spaces 
(containing the origin).
One can show that the set $\mcK^p(S^{n-1})$ of polyhedral convex sets in $S^{n-1}$ is dense in $\mcK(S^{n-1})$ 
with respect to the Hausdorff metric. 
This is seen by an easy adaption 
(cf.~\cite[Hilfssatz~2.5]{Gl} or~\cite[Prop.~3.3.4]{am:thesis})
of the proof for the corresponding Euclidean statement in~\cite[\S2.4]{Schn:book}.

\begin{remark}\label{rem:convexity-tubes}
One important difference between Euclidean and spherical 
convex geometry are the convexity properties of tubes. In the Euclidean case, 
the tubes around a convex set are again convex. In the spherical case this is rarely true.
Suppose that for $K\in\mcK(S^{n-1})$ the cone $C$ has a supporting 
hyperplane~$H\subseteq\IR^n$, $H\cap\inter(C)=\emptyset$, 
such that the face $H\cap C$ has dimension at least two.
Then one can show that the tube $\mcT(K,\alpha)$ is not convex, 
unless $\alpha=0$ or $\mcT(K,\alpha)=S^{n-1}$. 
This implies that for $n\ge 3$ and $k\ge 3$,
the cones $\IR_+^n$ and  $\Sym_+^k$ do not have convex tubes, respectively. 
\end{remark}

\subsection{Smooth convex sets}

We call $K\in\mcK(S^{n-1})$ {\em smooth} if $K$ is regular and 
its boundary $M=\partial K$ is a smooth hypersurface of $S^{n-1}$ 
with nowhere vanishing Gaussian curvature. 
Let $\mcK^\sm(S^{n-1})$ denote the family of smooth $K\in\mcK(S^{n-1})$. 
If we chose for $\nu(p)$ the unit vector in $T_pS^{n-1}$ 
normal to $T_pM$ and pointing inwards~$K$, 
then all principal curvatures of $M$ at~$p$ are positive, 
i.e., $\W_p$ is positive definite. 

\begin{lemma}\label{le:dense}
$\mcK^\sm(S^{n-1})$ is dense in $\mcK^r(S^{n-1})$.
\end{lemma}

An Euclidean analogue of this result was first shown by Minkowski (cf.~\cite[\S6]{bofe:74}).
Using this result, it is not hard to derive the above lemma; 
we refer to~\cite[Prop.~4.1.10]{am:thesis} for a proof.

\subsection{Intrinsic volumes}\label{se:intr-vol} 

In Section~\ref{se:intro-intrinsic}, we defined the intrinsic volumes $V_0(C),\ldots,V_n(C)$ 
of a polyhedral convex cone $C\subseteq\IR^n$. 
This allows to define the intrinsic volumes $V_j(K) := V_{j+1}(\cone(K))$ for $K\in\mcK^p(S^{n-1})$, $-1\leq j\leq n-1$. 
Recall from Section~\ref{sec:some-general} that $\mcK^p(S^{n-1})$ is dense in $\mcK(S^{n-1})$. 
It is a well-known fact that the functions $V_j\colon\mcK^p(S^{n-1}) \to \R$ have a unique continuous 
extension to $\mcK(S^{n-1})$ with respect to the Hausdorff metric $d_H$, cf.~\cite[Sec.~6.5]{SW:08}.

The following well-known facts~\cite[Sec.~6.5]{SW:08} about the intrinsic volumes are easily verified for polyhedral cones 
and therefore, by continuity, hold for any closed convex cone.  

\begin{proposition}\label{prop:facts-intrvol} 
Let $C\subseteq\IR^n$ be a closed convex cone.
\begin{enumerate}
\item The intrinsic volumes $V_0(C),\ldots,V_n(C)$ form a probability distribution: 
 $V_j(C)\ge 0$ and $\sum_{j=0}^n V_j(C) = 1$.
\item $V_j(QC) = V_j(C)$ for $Q\in O(n)$ (orthogonal invariance). 
\item We have $V_j(C) = V_{n-j}(\breve{C})$.
\end{enumerate}
\end{proposition}

For a regular cone $C$ one can slightly improve the bound $V_j(C) \le 1$. 

\begin{lemma}\label{le:bound-IV}
For any regular cone $C\subseteq\IR^n$ we have 
$V_j(C) \le \frac12$ for all $0\le j\le n$. 
\end{lemma}

\begin{proof}
It is known that 
\begin{equation*}\label{eq:Gauss-Bonnet}
  V_1(C)+V_3(C)+V_5(C) + \ldots = V_0(C)+V_2(C)+V_4(C) + \ldots = \tfrac{1}{2} \chi(C\cap S^{d-1}) ,
\end{equation*}
where $\chi$ denotes the \emph{Euler characteristic}, cf.~\cite[Sec.~4.3]{Gl} or~\cite[Thm.~6.5.5]{SW:08}. 
Moreover, $C$ is contained in an open halfspace since $C$ is regular. 
This implies $\chi(C\cap S^{d-1})=1$.  
\end{proof}

We next state a well-known formula for the spherical intrinsic volumes of smooth 
spherically convex sets, which directly follows from Weyl's tube formula~\cite{weyl:39}. 
(See~\cite[Ch.~4]{am:thesis} for the proof of a more general statement.)

\begin{proposition}\label{pro:V_j(K)--smooth}
Let $K\in \mcK^\sm(S^{n-1})$ and $0\leq j\leq n-2$. Then the intrinsic 
volumes of~$K$ are given by
\begin{equation*}
  V_j(K) = \frac{1}{\mcO_j\cdot \mcO_{n-2-j}}\,
                        \int_{p\in M} \sigma_{n-2-j}(p)\,dM \; ,
\end{equation*}
where $M:=\partial K$ denotes the boundary of $K$, and $\sigma_k(p)$ 
denotes the $k$th elementary symmetric function in the principal 
curvatures of~$M$. \hfill $\Box$
\end{proposition}

For an important special case consider a circular cap 
$K:=B(z,\beta)=\{p\in S^{n-1}\mid d(z,p)\leq \beta\}$ of radius $\beta\in(0,\pi)$. 
The boundary $M=\partial K$ is a sphere of dimension $n-2$ and radius $\sin\beta$.
Hence $\vol_{n-2}(M)=\sin(\beta)^{n-2}\cdot \mcO_{n-2}$. 
The principal curvatures of $M$ at any of its points~$p$
are given by $\kappa_1(p)= \ldots = \kappa_{n-2}(p) = \cot(\beta)$. 
In particular, we have $\sigma_i(p) = \binom{n-2}{i} \cot(\beta)^i$.
Proposition~\ref{pro:V_j(K)--smooth} implies that for $1\leq j < n$, 
\begin{align*}
   V_{j}(K) & =  
             \frac{\mcO_{n-2}}{\mcO_{j-1}\cdot \mcO_{n-j-1}}\cdot 
        \binom{n-2}{j-1}\cdot \sin(\beta)^{j-1}\cdot \cos(\beta)^{n-j-1} 
\\[2mm] & \stackrel{\eqref{eq:form-halfbinom}}{=} \binom{(n-2)/2}{(j-1)/2}
              \cdot \frac{\sin(\beta)^{j-1}\cdot \cos(\beta)^{n-j-1}}{2} \; .
\end{align*}
Furthermore, 
\begin{equation}\label{eq:f_n}
 V_{n}(K) = \frac{\mcO_{n-2}}{\mcO_{n-1}}\cdot 
                   \int_0^\beta \sin(\rho)^{n-2} \, d\rho 
  \stackrel{\eqref{eq:binom(n/2,m/2)=...}}{=} 
    \binom{(n-2)/2}{(n-1)/2}\cdot 
    \frac{n-1}{2}\cdot \int_0^\beta \sin(\rho)^{n-2} \, d\rho \; . 
\end{equation}
Recalling $V_{0}(K)=V_{n}(\breve{K})$, we get from this 
\begin{equation}\label{eq:f_0}
      V_{0}(K) = \frac{\mcO_{n-2}}{\mcO_{n-1}}\cdot 
                   \int_0^\beta \cos(\rho)^{n-2} \, d\rho 
 = \binom{(n-2)/2}{-1/2}\cdot \frac{n-1}{2}\cdot 
                  \int_0^{\frac{\pi}{2}-\beta} \sin(\rho)^{n-2} \, d\rho \; .
\end{equation} 

The intersection of a Lorentz cone $\mcL^n$ with $S^{n-1}$ is a 
circular cap of radius $\beta=\frac{\pi}{4}$.
From the above, we obtain the formula for 
$f_j(n) = V_{j}(\mcL^n)$ that was already stated in \eqref{eq:def-f_j(n)}.

\section{Deriving the Grassmannian tube formula}\label{se:tube-form} 

Let $C\subseteq\IR^n$ be a regular cone and put $K:=C\cap S^{n-1}$.
In a first step, we reduce the proof of Theorem~\ref{thm:tube-form} 
to the case where $K$ is smooth. 
Consider both sides of the inequality~\eqref{eq:vol-tube-Sigma-est-nice}:  
the intrinsic volumes $V_{j+1}(C)=V_j(K)$ are continuous in~$K$ with respect to the Hausdorff metric. 
Furthermore, it is straightforward to check that, for fixed~$\alpha$,  
$\rvol \TP(\Sigma_m(K),\alpha)$ depends continuously on $K$, cf.~\cite[Lemma~6.1.7]{am:thesis}. 
Moreover, by Lemma~\ref{le:dense},  $\mcK^\sm(S^{n-1})$ is dense in $\mcK^r(S^{n-1})$. 
Therefore, for proving Theorem~\ref{thm:tube-form}, 
we can assume without loss of generality that  $K\in\mcK^\sm(S^{n-1})$. 

\subsection{Parameterizing the tube around \texorpdfstring{$\SG_m$}{Sigma}}

In this subsection, we assume that $C=\cone(K)$ for some $K\in\mcK^\sm(S^{n-1})$. 
Fix $1\leq m< n$.
In Section~\ref{se:intro-Grassmann}, we assigned to $C$ 
the compact subsets $\PG_m(C)$ and $\DG_m(C)$ of the Grassmann manifold~$\Gr_{n,m}$. 
The set $\SG_m(C)=\mcP_m(C)\cap\mcD_m(C)$ consists of 
the subspaces~$W$ touching $C$. It is known that $\SG_m:=\SG_m(C)$ is the common boundary 
of $\PG_m(C)$ and $\DG_m(C)$, cf.~\cite{ambu:11c}. 
We analyze now the geometry of $\SG_m(C)$.

We begin with the basic observation that any $W\in\Sigma_m(C)$ intersects $K$ in 
a unique point~$p$, which moreover lies in the boundary of $K$. 
This allows us to define the map
$\PiM \colon \SG_m  \to \partial K,\, W\mapsto p$.

\begin{lemma}\label{prop:WcapK=p}
If $W\in\SG_m$, then $W\cap K=\{p\}$ for some $p\in\partial K$.
\end{lemma}

\begin{proof}
As $W\in\SG_m$, we have $\SG_m\cap C=\Sigma_m\cap \partial C\neq\{0\}$. It follows that there exists $p\in W\cap K$. 
To prove that $p$ is the only element in $W\cap K$, we assume that there exists $q\in W\cap K$, $p\neq q$. 
As~$K$ is regular, we have $p\neq -q$, so that there exists a unique great circle segment 
between~$p$ and~$q$. 
By convexity of $W$ and $K$, this arc lies in $W\cap K$, and thus in the boundary of~$K$. 
But this implies that along this arc, 
$M$~has zero Gaussian curvature, which contradicts the assumption $K\in\mcK^\sm(S^{n-1})$.
\end{proof}

Write $M:=\partial K$. Later on, we will see that $\PiM \colon \SG_m  \to M,\, W\mapsto p$ can be 
interpreted as the Grassmann bundle $\Gr(M,m-1)$ over $M$. 
This will lead us to the following basic result, 
whose proof is postponed to Section~\ref{sec:proof-tube-form}.  

\begin{proposition}\label{pro:Sigma} 
Let $K\in\mcK^\sm(S^{n-1})$ and $C:=\cone(K)$. 
Then $\Sigma_m(C)$ is a connected hypersurface of $\Gr_{n,m}$. 
Moreover, $\Sigma_m(C)$ has a unit normal vector field $\nu_\Sigma$ such that 
$\nu_\Sigma$ points into~$\PG_m(C)$ and $-\nu_\Sigma$ points into~$\DG_m(C)$ 
(compare Section~\ref{se:exp-map-tubes} for this notion). 
\end{proposition}

Proposition~\ref{pro:Sigma} says that we are in the situation of Section~\ref{se:basic-outline}, 
where $\M=\Gr_{n,m}$ and $\mcH=\SG_m(C)$. 
We reparameterize the map 
$\psi\colon\SG_m\times\R,(W,\theta)\mapsto \exp_W(\theta\,\nu_\Sigma(W))$ 
by setting $t=\tan\theta$ and thus obtain the smooth map
\begin{equation}\label{eq:def-Psi}
\Psi\colon \Sigma_m\times \IR \to \Gr_{n,m},\ W \mapsto \exp_W(\arctan t\, \nuS(W)) \; .
\end{equation}

The main technical difficulty is to understand the normal Jacobian of $\Psi$. 
It turns out that the normal Jacobian is a certain subspace-dependent version of the characteristic polynomial 
of the Weingarten map of $M=\partial K$. 
In the next section we define this purely algebraic notion, for which we coined the name 
twisted characteristic polynomal. 

\subsection{Twisted characteristic polynomials}\label{sec:tw-charpol}

Let $\vp$ be an endomorphism of a $k$-dimensional Euclidean vector space~$V$. 
We denote by $\sigma_j(\vp)$, $0\leq j\leq k$, the coefficients of 
the characteristic polynomial of $\vp$ (up to sign). More precisely,
  \[ \det(\vp-t\cdot \id_V) = \sum_{i=0}^k (-1)^{k-i}\cdot 
     \sigma_i(\vp)\cdot t^{k-i} \; . \]
Note that we have $\sigma_k(\vp)=\det(\vp)$, $\sigma_0(\vp)=1$, 
and $\sigma_1(\vp)=\trace(\vp)$.
In the following we denote by $\Gr(V,\ell)$ the set of all $\ell$-dimensional linear subspaces of~$V$.

\begin{definition}\label{def:tw-char-pol}
Let $Y\in\Gr(V,\ell)$, and denote by $\Pi_Y$ and $\Pi_{Y^\bot}$ the orthogonal 
projections onto~$Y$ and~$Y^\bot$, respectively.
The \emph{twisted characteristic polynomial $\ch_Y$ with respect to $Y$} is defined as 
$$
 \ch_Y(\vp,t)  := t^{k-\ell} \cdot\det\left(\vp-\left(t\cdot\Pi_Y-\tfrac{1}{t}
                         \cdot\Pi_{Y^\bot}\right)\right) \; ,
$$
We denote by $\vp_Y\colon Y\to Y$ the \emph{restriction} of~$\vp$, 
i.e., $\vp_Y(y):=\Pi_Y(\vp(y))$, and we use the notation $\tdet_Y(\vp):=\det(\vp_Y)$.
\end{definition}

Note that for $\ell=k$ we get $\ch_V(\vp,t)=\det(\vp-t\cdot \id_V)$, 
the usual characteristic polynomial, whereas for $\ell=0$ we get 
$\ch_0(\vp,t)=\det(t\cdot \vp+\id_V)$. 

We claim that 
\begin{equation}\label{eq:ch_Y(phi,0)=...}
  \ch_Y(\vp,0) 
   =\tdet_Y(\vp) \; .
\end{equation}
In order to see this, 
let us express the twisted characteristic polynomial in coordinates.
For $A\in\IR^{k\times k}$ and $0\leq\ell\leq k$ we use the notation
$\ch_\ell(A,t) := \ch_{\IR^\ell\times0}(\vp,t)$, 
where $\vp\colon\IR^k\to\IR^k$, $x\mapsto Ax$. Note that if $A$ has the block decomposition 
$A=\left(\begin{smallmatrix} A_1 & A_2 \\ A_3 & A_4\end{smallmatrix}\right)$, 
where $A_1\in\IR^{\ell\times\ell}$, and the other blocks accordingly, then
\begin{equation}\label{eq:ch_l(A,t)=...}
  \ch_\ell(A,t) = \det \begin{pmatrix} A_1-tI_\ell & A_2 
                          \\ tA_3 & tA_4+I_{k-\ell} \end{pmatrix} \; . 
\end{equation}
From this description we get for $Y=\IR^\ell\times0$ the identity 
$\ch_Y(\vp,0) =\det(A_1)=\tdet_Y(\vp)$, which proves~\eqref{eq:ch_Y(phi,0)=...}.

The next result expresses the expectation of the twisted characteristic polynomial~$\ch_Y(\vp,t)$, 
taken over a random subspace $Y\in\Gr(V,\ell)$, concisely in terms of the 
(coefficients of) the characteristic polynomial of $\vp$. 
We postpone the proof of this result to Appendix~\ref{sec:ave-tw-charpol}. 

\begin{thm}\label{thm:ave-twist-char-poly}
Let $V$ be a $k$-dimensional Euclidean vector space and let $\vp$ be 
an endomorphism of $V$. If $Y\in\Gr(V,\ell)$ is chosen uniformly 
at random, then
\begin{equation}\label{eq:E_Y(det_Y)=...}
  \underset{Y}{\IE}\big[\tdet_Y(\vp)\big] = \frac{1}{\binom{k}{\ell}}
                                            \cdot \sigma_\ell(\vp) \; .
\end{equation}
Moreover, the expectation of the twisted characteristic polynomial
 is given by
\begin{equation}\label{eq:E_Y[ch_Y]=...}
   \underset{Y}{\IE}\left[\ch_Y(\vp,t)\right] = \sum_{i,j=0}^k d_{ij}
                      \cdot \sigma_{k-j}(\vp)\cdot t^{k-i} \;, 
\end{equation}
where the coefficients $d_{ij}$ are given by 
\begin{equation*}\label{eq:form-d_(ij)-tw-pol}
  d_{ij} \;:=\; (-1)^{\frac{i-j}{2}-\frac{\ell}{2}}\cdot 
                \frac{\binom{\ell}{\frac{i-j}{2}+\frac{\ell}{2}}\cdot 
                \binom{k-\ell}{\frac{i+j}{2}-\frac{\ell}{2}}}{\binom{k}{j}} \; ,
\end{equation*}
for $i+j+\ell\equiv 0 \bmod 2$,  
$0\leq \tfrac{i-j}{2}+\tfrac{\ell}{2} \leq \ell$, and 
$0\leq \tfrac{i+j}{2}-\tfrac{\ell}{2} \leq k-\ell$; 
and given by $d_{ij}:=0$ otherwise. 
{\rm (}Note that $|d_{ij}|=d_{ij}^{k+2,\ell+1}$ in the notation of~\eqref{eq:form-d_(ij)}.{\rm )}

If $\vp$ is positive semidefinite, then we have for $t\in\IR$ that  
\begin{equation}\label{star}
\underset{Y}{\IE}\Big[\big|\ch_Y(\vp,t)\big|\Big] \;\leq\; 
     \sum_{i,j=0}^k \left|d_{ij}\right|\cdot \sigma_{k-j}(\vp)\cdot |t|^{k-i} \; . 
\end{equation}
\end{thm}

We are now in a position to state the announced formula that expresses 
the Jacobian of the parameter map $\Psi$ 
in terms of the twisted characteristic polynomial of the Weingarten map of $\partial K$.
We postpone the rather difficult proof to Section~\ref{sec:proof-tube-form}.  

\begin{thm}\label{thm:norm-Jac}
Let $K\in\mcK^\sm(S^{n-1})$ and write 
$M:=\partial K$, $C:=\cone(K)$, and $\SG_m:=\SG_m(C)$.
For $W\in\Sigma_m$ let $\nuS(W)\in T_W\Gr_{n,m}$ denote the unit 
normal vector pointing inside $\PG_m(K)$. 
We define the maps $\PiM$ and~$\Psi$ via
  \begin{align}\label{eq:def-Pi_m-Psi}
     \PiM \colon \Sigma_m & \to M & \Psi\colon \Sigma_m\times \IR & \to \Gr_{n,m}
  \\[1mm] W & \mapsto p \;\text{, where } W\cap K=\{p\} \; , & (W,t) & 
     \mapsto \exp_W(\arctan t\cdot \nuS(W)) \; . \nonumber
  \end{align}
Then the normal Jacobians of $\PiM$ and $\Psi$ are given by
\begin{align}\label{eq:ndet(D_wPiM)...}
    \ndet(D_W\PiM) & = \tdet_Y(\W_p)^{-1} , & |\det(D_{(W,t)}\Psi)| 
    & = (1+t^2)^{-n/2}\cdot \frac{\left|\ch_Y(\W_p,-t)\right|}{\tdet_Y(\W_p)} ,
\end{align}
where $W\cap K=\{p\}$, $Y:=p^\bot\cap W$,
and $\W_p$ denotes the Weingarten map of $M$ at $p$.
\end{thm}

\medskip

After these preparations, we are now ready to present the proof of the tube formula 
stated in Theorem~\ref{thm:tube-form}. 

\subsection{The Grassmannian tube formula}

\begin{proof}[Proof of Theorem~\ref{thm:tube-form}]
By Proposition~\ref{pro:Sigma},
$\PG_m(C)$ and $\DG_m(C)$ satisfy the assumptions of  Lemma~\ref{le:PD-tubes}, 
hence $\TP(\Sigma_m,\alpha) \subseteq \Psi\big(\SG_m\times [0,\alpha]\big)$. 
Let $0\le\alpha\le \pi/2$ and put $\tau:=\tan\alpha$. 
Applying the coarea formula to the map~$\Psi$ 
and using the formula in~\eqref{eq:ndet(D_wPiM)...} for the Jacobian of~$\Psi$, we obtain
\begin{align*}
   \vol \TP(\Sigma_m,\alpha) & \stackrel{\eqref{eq:cor-coarea-Riem-2}}{\leq} 
     \underset{(W,t)\in \Sigma_m \times [0,\tau]}{\int} |\det(D_{(W,t)}\Psi)| \,d(\Sigma_m\times \IR)
\\ & \stackrel{\eqref{eq:ndet(D_wPiM)...}}{=} 
     \underset{W\in \Sigma_m}{\int}\;\int_0^\tau (1+t^2)^{-n/2}\cdot 
     \frac{\left|\ch_Y(\W_p,-t)\right|}{\tdet_Y(\W_p)} \,dt\,dW .
\end{align*}
Changing the integration via the coarea formula applied to the map
$\PiM\colon\Sigma_m\to M$, 
and using the formula in~\eqref{eq:ndet(D_wPiM)...} for the normal Jacobian of $\PiM$, we get
\begin{align*}
   \vol \TP(\Sigma_m,\alpha) \ & \leq \underset{p\in M}{\int}\;
   \underset{Y\in\Gr(T_pM,m-1)}{\int}\;\int_0^\tau (1+t^2)^{-n/2}
   \cdot \left|\ch_Y(\W_p,-t)\right| \,dt\,dY\,dp
\\ & = \underset{p\in M}{\int} \int_0^\tau \frac{\vol\Gr_{n-2,m-1}}{(1+t^2)^{n/2}}
   \cdot \underset{Y}{\IE} \Big[\left|\ch_Y(\W_p,-t)\right|\Big] \,dt\,dp \; ,
\end{align*}
where the expectation is with respect to~$Y$ chosen uniformly at random in $\Gr(T_pM,m-1)$. 
Using the bound~\eqref{star} in Theorem~\ref{thm:ave-twist-char-poly} with $k=n-2$ and $\ell=m-1$, 
we obtain
\begin{align*}
   \vol \TP(\Sigma_m,\alpha) & \leq \underset{p\in M}{\int} \int_0^\tau 
   \frac{\vol\Gr_{n-2,m-1}}{(1+t^2)^{n/2}}\cdot \sum_{i,j=0}^{n-2} 
   \left|d_{ij}^{nm}\right|\cdot \sigma_{n-2-j}(\W_p)\cdot t^{n-2-i} \,dt\,dp
\\ & = \vol\Gr_{n-2,m-1}\cdot \sum_{i,j=0}^{n-2} \left|d_{ij}^{nm}\right|
   \cdot \int_0^\tau \frac{t^{n-2-i}}{(1+t^2)^{n/2}} \,dt\cdot 
   \underset{p\in M}{\int} \sigma_{n-2-j}(\W_p)\,dp
\\ &  = 
   \vol\Gr_{n-2,m-1}\cdot \sum_{i,j=0}^{n-2} \left|d_{ij}^{nm}\right|\cdot 
   I_{n,i}(\alpha)\cdot \mcO_j\cdot \mcO_{n-2-j}\cdot V_{j+1}(K) \; ,
\end{align*}
where for the last equality, 
we have used Proposition~\ref{pro:V_j(K)--smooth} 
and the definition~\eqref{eq:def-I} 
of the functions $I_{n,i}(\alpha)$. 
Finally, we get for $\rvol \, \TP (\Sigma_m,\alpha)=
\frac{\vol \TP (\Sigma_m,\alpha)}{\vol \Gr_{n,m}}$, 
using the formulas~\eqref{eq:binom(n/2,m/2)=...} and \eqref{eq:vol(Gr_(n-2,m-1))/vol(Gr_(n,m))=...} 
\begin{align*}
  \rvol \, \TP (\Sigma_m,\alpha) \ & \leq \frac{\vol\Gr_{n-2,m-1}}{\vol\Gr_{n,m}}
  \cdot \sum_{i,j=0}^{n-2} \left|d_{ij}^{nm}\right|\cdot I_{n,i}(\alpha)
  \cdot \mcO_j\cdot \mcO_{n-2-j}\cdot V_{j+1}(K)
\\ & \hspace{-1mm}\stackrel{\eqref{eq:vol(Gr_(n-2,m-1))/vol(Gr_(n,m))=...}}{=} 
  \frac{2m(n-m)}{n}\cdot \binom{n/2}{m/2}\cdot \sum_{j=0}^{n-2} V_{j+1}(K)\cdot 
  \frac{\mcO_j\cdot \mcO_{n-2-j}}{2\cdot \mcO_{n-2}}\cdot \sum_{i=0}^{n-2} 
  \left|d_{ij}^{nm}\right|\cdot I_{n,i}(\alpha)
\\ & \hspace{-1mm}\stackrel{\eqref{eq:binom(n/2,m/2)=...}}{=} \frac{2 m (n-m)}{n}
  \cdot \binom{n/2}{m/2} \cdot \sum_{j=0}^{n-2} V_{j+1}(K)\cdot \rbinom{n-2}{j}
  \cdot \sum_{i=0}^{n-2} |d_{ij}^{nm}|\cdot I_{n,i}(\alpha) \; . 
\end{align*}
The upper bound on the volume of $\TD(\Sigma_m,\alpha)$ 
is proved analogously.
\end{proof}

\section{Condition number estimates} \label{se:proof_main_re}

We derive here Theorem~\ref{thm:estim-conefree} and Theorem~\ref{thm:estim-conedep} 
from the tube formula in Theorem~\ref{thm:tube-form}. 

\subsection{Some technical estimations}

In the following proposition we collect some useful identities involving the 
binomial coefficients and the flag coefficients introduced in~\eqref{eq:def-flcoeff}. 

\begin{proposition}\label{prop:idents-flcoeffs}
\begin{enumerate}
\item For $n,m\in\IN$, $n\geq m$,
$$
    \binom{n}{m} \ =\  \rbinom{n}{m}\, \binom{n/2}{m/2}  \ =\ 
   \frac{\sqrt{\pi}\cdot \Gamma(\frac{n+1}{2})\cdot
      \Gamma(\frac{n+2}{2})}{\Gamma(\frac{m+1}{2})\cdot \Gamma(\frac{m+2}{2})
                 \cdot \Gamma(\frac{n-m+1}{2})\cdot \Gamma(\frac{n-m+2}{2})} \; . 
    \label{eq:binom(n,m)-decomp-flcoeff}
$$
\item For $n,m\in\IN$, $n\geq m$,
  \begin{align}
     \binom{n/2}{m/2} & = \frac{\omega_m\cdot \omega_{n-m}}{\omega_n} 
\quad,\qquad \rbinom{n}{m} = \frac{\mcO_m\cdot \mcO_{n-m}}{2\, \mcO_n} \; .
  \end{align}
  In particular,
  \begin{equation}
     \binom{n}{m}\cdot \frac{\omega_n}{\omega_m\cdot \omega_{n-m}}
      = \rbinom{n}{m} \quad,\qquad \binom{n}{m}\cdot 
         \frac{2\, \mcO_n}{\mcO_m\cdot \mcO_{n-m}}
      = \binom{n/2}{m/2} \; . \label{eq:form-halfbinom}
  \end{equation}
  \item For $n,m\to\infty$ such that $(n-m)\to\infty$,
  \begin{equation}\label{eq:flcoeff-asymptot}
     \rbinom{n}{m} \;\sim\; \sqrt{\frac{\pi}{2}}\cdot 
     \sqrt{\frac{m(n-m)}{n}}\cdot \binom{n/2}{m/2} \; ,
  \end{equation}
        where the symbol $\sim$ means that the quotient of the two sides
tends to one.
\end{enumerate}
\end{proposition}

\begin{proof}
The first equation in~\eqref{eq:binom(n,m)-decomp-flcoeff} follows 
from applying the duplication formula of the~$\Gamma$-function 
$\Gamma(2x) = \tfrac{1}{\sqrt{\pi}}\cdot 2^{2x-1}\cdot \Gamma(x)
\cdot \Gamma(x+\tfrac{1}{2})$ to the term 
$\binom{n}{m} = \frac{\Gamma(n+1)}{\Gamma(m+1)\cdot\Gamma(n-m+1)}$.
The remaining equations follow by plugging in the definitions of the 
corresponding quantities.
As for the asymptotics stated in~\eqref{eq:flcoeff-asymptot}, we compute
\begin{align*}
   \sqrt{\frac{\pi}{2}}\cdot & \sqrt{\frac{m(n-m)}{n}} \cdot \binom{n/2}{m/2}
    = \sqrt{\frac{\pi}{2}}\cdot \sqrt{\frac{m(n-m)}{n}}
       \cdot \frac{\Gamma(\frac{n}{2}+1)}{\Gamma(\frac{m}{2}+1)
       \cdot \Gamma(\frac{n-m}{2}+1)}
\\ & \qquad = \frac{\sqrt{\pi}\cdot \sqrt{\frac{n}{2}}
       \cdot \Gamma(\frac{n}{2})}{\sqrt{\frac{m}{2}}\cdot \Gamma(\frac{m}{2})
       \cdot \sqrt{\frac{n-m}{2}}\cdot \Gamma(\frac{n-m}{2})} \sim \frac{\sqrt{\pi}
       \cdot \Gamma(\frac{n+1}{2})}{\Gamma(\frac{m+1}{2})\cdot \Gamma(\frac{n-m+1}{2})}
            = \rbinom{n}{m} \; ,
\end{align*}
where we have used the asymptotics $\sqrt{x}\cdot \Gamma(x) 
\sim \Gamma(x+\tfrac{1}{2})$ for $x\to\infty$.
\end{proof}

We now rewrite the upper bound in Theorem~\ref{thm:tube-form}.

\begin{lemma}\label{cor:tube-form}
Let $C\subseteq\IR^n$ be a regular cone,
$1\leq m <n$, and $0\leq\alpha\leq\frac{\pi}{2}$.
Using the convention $\binom{k}{\ell}:=0$ if $\ell<0$ or $\ell>k$, we have
\begin{align*}
  \rvol \, \mcT(\Sigma_m(C),\alpha) & \;\leq\; 8 \, \sum_{i,k=0}^{n-2} V_{n-m-i+2k}(C) 
     \cdot \frac{\Gamma(\frac{m+i-2k+1}{2})}{\Gamma(\frac{m}{2})}\cdot 
     \frac{\Gamma(\frac{n-m-i+2k+1)}{2})}{\Gamma(\frac{n-m}{2})} \nonumber
\\ & \hspace{3cm} \cdot \binom{m-1}{k} \binom{n-m-1}{i-k}\cdot I_{n,n-2-i}(\alpha) \, .
\end{align*}
\end{lemma}

\begin{proof}

Theorem~\ref{thm:tube-form} implies that 
\begin{equation}\label{eq:est-rvolT(Sigma)}
   \rvol \, \mcT(\Sigma_m(C),\alpha) \;\leq\; \frac{4 m (n-m)}{n} 
   \binom{n/2}{m/2} \cdot \sum_{j=0}^{n-2} V_{j+1}(C)\cdot \rbinom{n-2}{j}\cdot 
   \sum_{i=0}^{n-2} d_{ij}^{nm}\cdot I_{n,i}(\alpha) \; .
\end{equation}
Using Proposition~\ref{prop:idents-flcoeffs}, we derive the following identity: 
\begin{align*}
   & \frac{4m (n-m)}{n} \cdot \frac{\binom{n/2}{m/2} \cdot 
   \rbinom{n-2}{j}}{\binom{n-2}{j}} \stackrel{\eqref{eq:binom(n,m)-decomp-flcoeff}}{=} 
   \frac{4m (n-m)}{n}\cdot \frac{\binom{n/2}{m/2}}{\binom{(n-2)/2}{j/2}}
\\ & \stackrel{\eqref{eq:binom(n/2,m/2)=...}}{=} 
   \frac{4m (n-m)}{n}\cdot
   \frac{\Gamma(\frac{n+2}{2})}{\Gamma(\frac{m+2}{2})\cdot
   \Gamma(\frac{n-m+2}{2})}\cdot \frac{\Gamma(\frac{j+2}{2})\cdot
   \Gamma(\frac{n-j}{2})}{\Gamma(\frac{n}{2})}
 = 8\cdot \frac{\Gamma(\frac{j+2}{2})}{\Gamma(\frac{m}{2})}\cdot
   \frac{\Gamma(\frac{n-j}{2})}{\Gamma(\frac{n-m}{2})} \; .
\end{align*}
Using this in~\eqref{eq:est-rvolT(Sigma)}, changing the summation via 
$i\leftarrow n-2-i$ and $j\leftarrow n-2-j$, and taking into account 
the symmetry relations for $d_{ij}^{nm}$ stated in 
Remark~\ref{re:comm-tube-formula}, 
we get
\begin{align}\label{eq:estim-before-sumchng}
   & \rvol \, \mcT(\Sigma_m(C),\alpha) \;\leq\; 
   8 \, \sum_{i,j=0}^{n-2} \frac{\Gamma(\frac{j+2}{2})}{\Gamma(\frac{m}{2})}\cdot 
   \frac{\Gamma(\frac{n-j}{2})}{\Gamma(\frac{n-m}{2})} \cdot V_{j+1}(C) \cdot 
   \binom{n-2}{j} \cdot d_{ij}^{nm}\cdot I_{n,i}(\alpha) \nonumber
\\ & = 
   8 \, \sum_{i,j=0}^{n-2} \frac{\Gamma(\frac{j+2}{2})}{\Gamma(\frac{m}{2})}\cdot 
   \frac{\Gamma(\frac{n-j}{2})}{\Gamma(\frac{n-m}{2})} \cdot V_{n-1-j}(C) \cdot 
   \binom{n-2}{j} \cdot d_{ij}^{nm}\cdot I_{n,n-2-i}(\alpha)
\\ & = 8 \, \hspace{-1mm}\sum_{\substack{i,j=0 \\ i+j+m\equiv 1 \\ \pmod 2}}^{n-2} 
   V_{n-1-j}(C) \cdot \frac{\Gamma(\frac{j+2}{2})}{\Gamma(\frac{m}{2})}\cdot 
   \frac{\Gamma(\frac{n-j}{2})}{\Gamma(\frac{n-m}{2})} \cdot 
   \binom{m-1}{\frac{i-j}{2}+\frac{m-1}{2}} \binom{n-m-1}{\frac{i+j}{2}-\frac{m-1}{2}}\cdot 
   I_{n,n-2-i}(\alpha) \; . \nonumber
\end{align}
Here we interpret $\binom{k}{\ell}=0$ if $\ell<0$ or $\ell>k$, i.e., 
the above summation over $i,j$ in fact only runs over the rectangle determined 
by the inequalities $0\leq \frac{i-j}{2}+\frac{m-1}{2}\leq m-1$ and 
$0\leq \frac{i+j}{2}-\frac{m-1}{2}\leq n-m-1$ (cf.~Figure~\ref{fig:chng-summation}). 
As the summation runs only 
over those $i,j$, for which $i+j+m\equiv 1\bmod 2$, we may replace the 
summation over $j$ by a summation over $k=\frac{i-j}{2}+\frac{m-1}{2}$. 
The above inequalities then transform into $0\leq k\leq m-1$ and $0\leq i-k\leq n-m-1$. 
So we get from~\eqref{eq:estim-before-sumchng}
\begin{align*}
   \rvol \, \mcT(\Sigma_m(C),\alpha) & = 8 \, \sum_{i,k=0}^{n-2} 
     V_{n-m-i+2k}(C) \cdot \frac{\Gamma(\frac{m+i-2k+1}{2})}{\Gamma(\frac{m}{2})}
     \cdot \frac{\Gamma(\frac{n-m-i+2k+1)}{2})}{\Gamma(\frac{n-m}{2})}
\\ & \hspace{3cm} \cdot \binom{m-1}{k} \binom{n-m-1}{i-k}\cdot 
     I_{n,n-2-i}(\alpha) \; . \qedhere
\end{align*}
\end{proof}

In the next lemma we provide a number of technical estimates. 
For the ease of presentation we defer the proof of this lemma to Appendix~\ref{sec:proof_technicalities}.

\begin{lemma}\label{lem:estimates}
Let $i,k,\ell,m,n\in\IN$ with $n\geq 2$ and $1\leq m\leq n-1$.
\begin{enumerate}
  \item We have
  \begin{equation}\label{eq:lem-estimates--1}
    \frac{\Gamma(\frac{m+\ell+1}{2})}{\Gamma(\frac{m}{2})} \leq 
    \sqrt{\frac{m}{2}}\cdot \left(\frac{m+\ell}{2}\right)^{\frac{\ell}{2}} \; .
  \end{equation}
  \item For $0\leq k\leq m-1$ and $0\leq i-k\leq n-m-1$ we have
  \begin{equation}\label{eq:lem-estimates--2}
    \left(\frac{m+i-2k}{n-m-i+2k}\right)^{\frac{i-2k}{2}} < n^{\frac{i}{2}} \; .
  \end{equation}
  \item For $0\leq\alpha\leq\frac{\pi}{2}$, $t:=\sin(\alpha)^{-1}$, 
        and $n\geq 3$, we have
  \begin{align}
     \sum_{i=0}^{n-2} \binom{n-2}{i}\cdot n^{\frac{i}{2}} \cdot 
     I_{n,n-2-i}(\alpha) & \;\;<\;\; \frac{3}{t} \;,\quad 
     \text{if $t>n^{\frac{3}{2}}$} \; , \label{eq:lem-estimates--4}
  \\ \sum_{i=0}^{n-2} \binom{n-2}{i}\cdot I_{n,n-2-i}(\alpha) & 
     \;\;<\;\; \exp\left(\frac{n}{m}\right)\cdot \frac{1}{t} \;,\quad 
     \text{if $t>m$} \; . \label{eq:lem-estimates--5}
  \end{align}
\end{enumerate}
\end{lemma}

\subsection{Proof of the general bound}

In the following proof, in the remainder of this section as well as in Appendix~\ref{sec:proof_technicalities},  
we will mark estimates that are easily checked with a computer algebra system 
with the symbol~$\calc$.

\begin{proof}[Proof of Theorem~\ref{thm:estim-conefree}] 
We have $V_j(C)\leq \frac{1}{2}$ by Lemma~\ref{le:bound-IV}.
Using this bound in Lemma~\ref{cor:tube-form} and 
applying Lemma~\ref{lem:estimates}, we get 
\begin{align*}
   \rvol \, \mcT(\Sigma_m,\alpha) & \ \leq\  
       8 \,\sum_{i,k=0}^{n-2} V_{n-m-i+2k}(C) \cdot 
      \frac{\Gamma(\frac{m+i-2k+1}{2})}{\Gamma(\frac{m}{2})}
      \cdot \frac{\Gamma(\frac{n-m-i+2k+1)}{2})}{\Gamma(\frac{n-m}{2})}
\\ & \hspace{3cm} \cdot \binom{m-1}{k} \binom{n-m-1}{i-k}
      \cdot I_{n,n-2-i}(\alpha)
\\ & \stackrel{\eqref{eq:lem-estimates--1}}{\leq} 2 \, \sqrt{m(n-m)}
      \, \sum_{i,k=0}^{n-2} \left(\frac{m+i-2k}{2}\right)^{\frac{i-2k}{2}} 
      \, \left(\frac{n-m-i+2k}{2}\right)^{-\frac{i-2k}{2}}
\\ & \hspace{3cm} \cdot \binom{m-1}{k} \binom{n-m-1}{i-k}
      \cdot I_{n,n-2-i}(\alpha)
\\ & \stackrel{\eqref{eq:lem-estimates--2}}{\leq} 2 \, \sqrt{m(n-m)}
      \, \sum_{i=0}^{n-2} n^{\frac{i}{2}} \cdot I_{n,n-2-i}(\alpha)
      \cdot \sum_{k=0}^{n-2} \binom{m-1}{k} \binom{n-m-1}{i-k} \; .
\end{align*}
By Vandermonde's identity we have $\sum_{k=0}^{n-2} \binom{m-1}{k}
\binom{n-m-1}{i-k}=\binom{n-2}{i}$. So we get
\begin{align*}
   \rvol \, \mcT(\Sigma_m,\alpha) & \leq 2\, \sqrt{m(n-m)}\, 
   \sum_{i=0}^{n-2} \binom{n-2}{i}\cdot n^{\frac{i}{2}} \cdot I_{n,n-2-i}(\alpha)
\\ & \stackrel{\eqref{eq:lem-estimates--4}}{<} 6\, \sqrt{m(n-m)}
   \, \frac{1}{t} \;,\quad \text{if $t>n^{\frac{3}{2}}$ and $n\geq3$} \; .
\end{align*}
This is the tail bounded stated in Theorem~\ref{thm:estim-conefree}. 

As for the expectation of the logarithm of the Grassmann condition, we compute
\begin{align*}
   \IE\left[\ln \CC_\Gra(A) \right] & = \int_0^\infty \Prob[\ln \CC_\Gra(A) > s] \,ds
\\ & < 1.5\, \ln(n) + r + \int_{\ln(n^{3/2}) + r}^\infty 6 \, \sqrt{m(n-m)}
       \cdot \exp(-s)\,ds
\\ & = 1.5\, \ln(n) + r + \underbrace{6 \,
       \frac{\sqrt{m(n-m)}}{n^{3/2}}}_{\leq \frac{2}{3}\sqrt{6} 
       \text{ , if $n\geq3$}}\cdot \exp(-r)
\\[-2mm] & \stackrel{\calc}{<} 1.5\, \ln(n) + 1.5 \; ,
\end{align*}
if we choose $r:=\frac{1}{2}\ln\left(\frac{8}{3}\right)$. 
This finishes the proof of Theorem~\ref{thm:estim-conefree}.
\end{proof}

\subsection{A certain log-concave sequence} 

Before we finish this section with the proof of Theorem~\ref{thm:estim-conedep}, 
we need yet another technical lemma. 
A sequence $(a_n)$ of nonnegative real numbers is called \emph{log-concave} iff $a_n^2\geq a_{n-1}\, a_{n+1}$ for all~$n$. 
See~\cite{Stanley} for a survey on log-concave sequences and their appearances in diverse areas of mathematics.

\begin{lemma}\label{lem:g_m(n)}
For $n\geq 2$ and $1\leq m\leq n-1$ let 
$g_m(n) := \frac{\Gamma(\frac{n}{2})\cdot \exp(\tfrac{n}{m})}{\Gamma(\frac{m}{2})\cdot \Gamma(\frac{n-m}{2})\cdot 2^{n/2}}$.
\begin{enumerate}
  \item The sequence $(g_m(n))_n$ is log-concave, i.e., $g_m(n)^2\geq g_m(n-1)\, g_m(n+1)$ for $n\geq m+2$.
  \item For fixed $m\geq8$ we have $\max\{g_m(n)\mid n> m\} = \max\{ g_m(2m+k)\mid k\in\{5,6,7\} \}$.
  \item For $m\geq8$ we have
  \begin{equation}\label{eq:g_m(n)<2.5*sqrt(m)}
    g_m(n) < 2.5\,\sqrt{m} \; .
  \end{equation}
\end{enumerate}
\end{lemma}

\begin{proof}
(1) We have
\begin{align*}
   \frac{g_m(n)^2}{g_m(n-1)\, g_m(n+1)} & = \frac{\Gamma(\frac{n}{2})^2}{\Gamma(\frac{n-m}{2})^2}\cdot
   \frac{\Gamma(\frac{n-1-m}{2})}{\Gamma(\frac{n-1}{2})}\cdot
   \frac{\Gamma(\frac{n+1-m}{2})}{\Gamma(\frac{n+1}{2})} \; .
\end{align*}
In order to show that this expression is greater or equal than~$1$, 
we use induction on~$m$. For $m=0$ this is trivially true,
and for $m=1$ this is easily checked with a computer algebra system.
For $m\geq2$ we have, using $\Gamma(x+1)=x\cdot \Gamma(x)$,
\begin{align*}
   & \frac{\Gamma(\frac{n}{2})^2}{\Gamma(\frac{n-m}{2})^2}\cdot
   \frac{\Gamma(\frac{n-1-m}{2})}{\Gamma(\frac{n-1}{2})}\cdot
   \frac{\Gamma(\frac{n+1-m}{2})}{\Gamma(\frac{n+1}{2})}
\\ & = \underbrace{\frac{(n-m)^2}{(n-1-m)\cdot (n+1-m)}}_{>1}\cdot 
   \underbrace{\frac{\Gamma(\frac{n}{2})^2}{\Gamma(\frac{n-(m-2)}{2})^2}\cdot
   \frac{\Gamma(\frac{n-1-(m-2)}{2})}{\Gamma(\frac{n-1}{2})}\cdot
   \frac{\Gamma(\frac{n+1-(m-2)}{2})}{\Gamma(\frac{n+1}{2})}}_{\geq1 \text{ by ind.~hyp.}}
   > 1 \; .
\end{align*}

(2) As the sequence $(g_m(n))_n$ is log-concave and positive, it follows that it is \emph{unimodal}. 
This means that there exists an index~$N$ such that $g_m(n-1)\leq g_m(n)$ for all $n\leq N$, and $g_m(n)\geq g_m(n+1)$ for all $n\geq N$ (cf.~\cite{Stanley}). 
Moreover, for $m\geq8$ we have $N\in \{ 2m+k\mid k\in\{5,6,7\}\}$, as
\begin{align*}
   \frac{g_m(2m+4)}{g_m(2m+5)} & = \frac{\Gamma(\frac{2m+4}{2})}
   {\Gamma(\frac{2m+5}{2})}\cdot \frac{\Gamma(\frac{m+5}{2})}
   {\Gamma(\frac{m+4}{2})} \cdot \frac{\sqrt{2}}
   {\exp\left(\frac{1}{m}\right)}
    \;\stackrel{\calc}{<}\; 1 \; , 
\\[2mm] \frac{g_m(2m+7)}{g_m(2m+8)} & = \frac{\Gamma(\frac{2m+7}{2})}{\Gamma(\frac{2m+8}{2})}\cdot \frac{\Gamma(\frac{m+8}{2})}{\Gamma(\frac{m+7}{2})} 
     \cdot \frac{\sqrt{2}}{\exp\left(\frac{1}{m}\right)}
    \;\stackrel{\calc}{>}\; 1 \;,\qquad\text{for $m\geq8$} \; .
\end{align*}

(3) For fixed~$k$, the following asymptotics is easily verified:
  \[ \frac{g_m(2m+k)}{\sqrt{m}} \quad \stackrel{m\to\infty}{\longrightarrow} \quad
     \frac{4\,\exp(2)}{\sqrt{2\pi}} < 12 \; . \]
In particular, it follows by~(2) that for $m\geq8$ we have an asymptotic estimate of $g_m(n)=O(\sqrt{m})$. More precisely, it is straightforward to check that 
for $k\in\{5,6,7\}$ and $m\geq8$ we have $g_m(2m+k)<2.5\,\sqrt{m}$. It follows by~(2) that for $m\geq8$ we have $g_m(n)<2.5\,\sqrt{m}$.
\end{proof}

\subsection{Proof of the refined bounds}

\begin{proof}[Proof of Theorem~\ref{thm:estim-conedep}]
We will estimate the intrinsic volumes of~$C$ via $V_{j}(C)\leq \exc(C)\, f_j(n)$, 
where $f_j(n)=V_{j}(\mcL^n)$ and $\exc(C)$ denotes the excess over the Lorentz cone introduced in \eqref{eq:def-v(C)}. 
Note that for $j=m+i-2k$ with $1\leq j\leq n-1$, we get from~\eqref{eq:def-f_j(n)} using~\eqref{eq:binom(n/2,m/2)=...} 
\begin{equation}\label{eq:f_(n-(m+i-2k))}
  f_{n-(m+i-2k)}(n)
      = \frac{\binom{(n-2)/2}{(n-m-i+2k-1)/2}}{2^{n/2}}
      = \frac{\Gamma(\frac{n}{2})}{\Gamma(\frac{n-m-i+2k+1}{2})\cdot 
              \Gamma(\frac{m+i-2k+1}{2})\cdot 2^{n/2}} \; .
\end{equation}
We thus obtain from Lemma~\ref{cor:tube-form}, using Lemma~\ref{lem:estimates} for 
the last equality: 
\begin{align*}
   & \rvol \, \mcT(\Sigma_m,\alpha) \leq 
     8 \, \sum_{i,k=0}^{n-2} V_{n-m-i+2k}(C) \cdot 
     \frac{\Gamma(\frac{m+i-2k+1}{2})}{\Gamma(\frac{m}{2})}\cdot 
     \frac{\Gamma(\frac{n-m-i+2k+1)}{2})}{\Gamma(\frac{n-m}{2})}
\\ & \hspace{6cm} \cdot \binom{m-1}{k} \binom{n-m-1}{i-k}
     \cdot I_{n,n-2-i}(\alpha)
\\ & \stackrel{\eqref{eq:f_(n-(m+i-2k))}}{\leq} 
     8 \, \exc(C)\cdot \frac{\Gamma(\frac{n}{2})}
        {\Gamma(\frac{m}{2})\cdot \Gamma(\frac{n-m}{2})\cdot 2^{n/2}}
     \cdot \sum_{i=0}^{n-2} I_{n,n-2-i}(\alpha)\cdot 
     \sum_{k=0}^{n-2} \binom{m-1}{k} \binom{n-m-1}{i-k}
\\ & = 8 \, \exc(C)\cdot \frac{\Gamma(\frac{n}{2})}
        {\Gamma(\frac{m}{2})\cdot \Gamma(\frac{n-m}{2})\cdot 2^{n/2}}
     \cdot \sum_{i=0}^{n-2} I_{n,n-2-i}(\alpha)\cdot \binom{n-2}{i}
\\ & \stackrel{\eqref{eq:lem-estimates--5}}{<} 8 \, \exc(C)\cdot 
     \frac{\Gamma(\frac{n}{2})\cdot \exp(\tfrac{n}{m})}
          {\Gamma(\frac{m}{2})\cdot \Gamma(\frac{n-m}{2})\cdot 2^{n/2}}
     \cdot \frac{1}{t} ,\quad \text{if $t>m$ and $n\geq3$}  .
\end{align*}
Using the notation from Lemma~\ref{lem:g_m(n)}, this implies 
the desired tail estimate
\begin{align*}
   \rvol \, \mcT(\Sigma_m,\alpha) & < 8 \, \exc(C)\, g_m(n)\, \frac{1}{t}
     \stackrel{\eqref{eq:g_m(n)<2.5*sqrt(m)}}{<} 20\, \exc(C) \, \sqrt{m} \, \frac{1}{t} ,\quad \text{for $t>m\geq 8$} \; .
\end{align*}

Analogous to the proof of Theorem~\ref{thm:estim-conefree}, we estimate 
the expectation of the logarithm of the Grassmann condition. Defining 
$\tilde{\exc}(C):=\max\{\exc(C),1\}$, so that in particular $\ln \tilde{\exc}(C)\geq0$, we get
\begin{align*}
   \IE\left[\ln \CC_\Gra(A)\right] & = \int_0^\infty \Prob[\ln \CC_\Gra(A) > s] \,ds
\\ & < \ln(m)+\ln(\tilde{\exc}(C)) + r + \int_{\ln(m)+\ln(\tilde{\exc}(C))+r}^\infty 
       20\cdot \exc(C) \cdot \sqrt{m}\cdot \exp(-s)\,ds
\\ & = \ln(m)+\ln(\tilde{\exc}(C)) + r + \frac{20 \, \exc(C)}{\sqrt{m}\cdot \tilde{\exc}(C)\cdot \exp(r)}
\\ & < \ln(m) + \ln(\tilde{\exc}(C)) + 3 \;,\quad \text{if $m\geq8$} \; ,
\end{align*}
where the last inequality follows from choosing $r:=\ln\big(\frac{20}{\sqrt{8}}\big)$.
This completes the proof of Theorem~\ref{thm:estim-conedep}.
\end{proof}
  
\section{Geometry of Grassmann manifolds}\label{sec:prelim-Riemgeom}

It remains to prove Proposition~\ref{pro:Sigma}, which describes $\Sigma_m$ 
as a smooth hypersurface in $\Gr_{n,m}$, as well as  
Theorem~\ref{thm:norm-Jac} on the Jacobian of the parameterization $\Psi$ 
of the tube around $\Sigma_m$. 
We provide these proofs in Section~\ref{sec:proof-tube-form}. 
As they require a good understanding of the metric properties of 
orthogonal groups and Grassmann manifolds, we collect here 
the necessary notions needed for these proofs. 
For a general background on Riemannian geometry we refer to~\cite{dC},~\cite{booth}, or~\cite[Ch.~1]{Chav};
as necessary we will give more specific references.

\subsection{Orthogonal groups}\label{sec:O(n)-Gr_(n,m)}

The orthogonal group $O(n)$ is the compact Lie group 
consisting of the 
$U\in\IR^{n\times n}$ such that $U^T U = I_n$. 
Its Lie algebra is given by 
  \[ T_{I_n}O(n) = \Skew_n := \{U\in\IR^{n\times n}\mid U^T=-U\} . \]
The tangent space of an element $Q\in O(n)$ is 
thus given by $T_QO(n) = Q\cdot \Skew_n$.

As for the Riemannian metric on $O(n)$, it is convenient to scale the (Euclidean) metric
induced from $\IR^{n\times n}$ by a factor of~$\frac{1}{2}$. The Riemannian 
metric $\langle .,.\rangle_Q$ on $T_Q O(n)=Q\cdot \Skew_n$ is thus given by
\begin{equation}\label{eq:def-metric-O(n)}
  \langle QU_1,QU_2\rangle_Q = \tfrac{1}{2}\tr(U_1^T U_2) .
\end{equation}
for $U_1,U_2\in \Skew_n$. Observe that we have a canonical basis for $\Skew_n$ 
given by $\{E_{ij}-E_{ji} \mid 1\leq j<i\leq n\}$, where $E_{ij}$ denotes the 
$(i,j)$th elementary matrix, i.e.,~the matrix whose entries are zero everywhere except 
for the $(i,j)$th entry, which is~$1$. This basis is orthogonal and by the 
choice of the scaling factor it is also orthonormal.

For the orthogonal group~$O(n)$, the exponential map $\exp_Q\colon T_Q O(n)\to O(n)$ at $Q\in O(n)$ is given by the usual 
matrix exponential, i.e., for $U\in\Skew_n$ we have 
$\exp_Q(QU)=Q\cdot e^U=Q\cdot \sum_{k=0}^\infty \frac{U^k}{k!}$. 
For example, for 
\begin{equation}\label{eq:def-N}
 N := \left(\begin{array}{c|c|c}
   \scriptstyle0 & 0 & \scriptstyle 1
\\\hline 0 & \raisebox{-3mm}{\rule{0mm}{9mm}}\rule{3mm}{0mm}0\rule{3mm}{0mm} & 0
\\\hline \scriptstyle-1 & 0 & \scriptstyle0
\end{array}\right)
\end{equation}
the exponential map at $Q$ in direction $QN$ is given by the rotation 
\begin{equation}\label{eq:exp_Q(QU)}
  \exp_Q(\rho\cdot QN) = Q\cdot Q_\rho \;,\quad Q_\rho := \left(\begin{array}{c|c|c}
   \scriptstyle\cos\rho & 0 & \scriptstyle\sin\rho
\\\hline 
   0 & \raisebox{-3mm}{\rule{0mm}{9mm}}\rule{2mm}{0mm}I_{n-2}\rule{2mm}{0mm} & 0
\\\hline \scriptstyle-\sin\rho & 0 & \scriptstyle\cos\rho
\end{array}\right) .
\end{equation}
Note that $\frac{d}{d\rho} Q_\rho = Q_\rho\cdot N$.

\subsection{Grassmann manifolds}

We defined the \emph{Grassmann manifold} $\Gr_{n,m}$, $1\leq m\leq n$ as the 
set of all $m$-dimensional subspaces of~$\IR^n$. 
For our explicit explicit calculations, it is essential to view the Grassmann manifold as a quotient of the orthogonal group. 
Namely, we can describe $\Gr_{n,m}$ 
as the quotient $O(n)/H:=\{QH\mid Q\in O(n)\}$ of $O(n)$ by the subgroup
  \[ H := \left\{ \left. \begin{pmatrix} Q' & 0 \\ 0 & Q'' 
                         \end{pmatrix} \right| Q'\in O(m) \,,\; 
               Q''\in O(n-m)\right\} \simeq O(m)\times O(n-m) \]
in the following way:
\begin{equation}\label{eq:Pi:O(n)->Gr_(n,m)}
\begin{array}[c]{c}
\begin{tikzpicture}[>=stealth]
\matrix (m) [matrix of math nodes, column sep=10mm, row sep=3mm, 
             text height=1.5ex, text depth=0.25ex]
  {    O(n) & \Gr_{n,m} \\[6mm]
       O(n)/H \\
  };
\path[->>]
  (m-1-1) edge node[auto]{$\Pi$} (m-1-2)
  (m-1-1) edge (m-2-1);
\path[->, dashed]
  (m-2-1) edge node[auto,below]{$\sim$} (m-1-2);
\end{tikzpicture}
\end{array} ,\quad \text{with } \Pi\colon Q \mapsto Q\,(\IR^m\times 0) .
\end{equation}
In other words, the identification of $\Gr_{n,m}$ with the quotient $O(n)/H$ 
amounts to identifying an $m$-dimensional linear subspace~$W$ of $\IR^n$ 
with the set of all orthogonal matrices, whose first $m$ columns span~$W$. 
Note that this identification endows $\Gr_{n,m}$ with the quotient topology on $O(n)/H$.

In the following paragraphs, we will give a concrete description of the tangent space $T_W\Gr_{n,m}$ for $W\in\Gr_{n,m}$, 
and we will describe the Riemannian metric and the exponential map on $\Gr_{n,m}$.
Note that the natural action of $O(n)$ on $\IR^n$ induces a corresponding action 
on $\Gr_{n,m}$. It can be shown that, up to scaling, there exists a unique 
Riemannian metric on $\Gr_{n,m}$, which is $O(n)$-invariant. 
This is the Riemannian metric, that we will next describe explicitly.

Let $W\in\Gr_{n,m}$ and let $Q\in \Pi^{-1}(W)$, i.e., the first $m$ columns 
of $Q$ span $W$. The coset $QH$ is a submanifold of $O(n)$, and 
its tangent space at~$Q$ is given by
\begin{align}
   T_Q QH & = Q\cdot T_{I_n}H = Q\cdot \left\{ \left. \begin{pmatrix} U' & 0 
               \\ 0 & U'' \end{pmatrix} \right| U'\in \Skew_m \,,\; 
               U''\in \Skew_{n-m}\right\} .
\nonumber
\intertext{Note that the orthogonal complement $(T_Q QH)^\bot$ of $T_Q QH$ in $T_QO(n)$ is given by}
   (T_Q QH)^\bot & = Q\cdot \ol{\Skew_n} \; \text{, where } \; \ol{\Skew_n} := 
     \left\{ \left. \begin{pmatrix} 0 & -X^T \\ X & 0 \end{pmatrix} \right| 
                                  X\in \IR^{(n-m)\times m} \right\} .
\label{eq:bar(Skew_n)}
\end{align}

It can be shown (cf.~\cite[Lemma~II.4.1]{helga}) that there exists an open ball~$B$ around the origin in~$T_QO(n)$ 
such that the restriction of~$\Pi\circ \exp_Q$ to $B\cap (T_Q QH)^\bot$ is a diffeomorphism onto an open neighborhood of~$W=\Pi(Q)$ in $\Gr_{n,m}$. 
The derivative of $\Pi\circ\exp_Q$ therefore yields a linear isomorphism
  \[ (T_Q QH)^\bot \stackrel{\sim}{\longrightarrow} T_W\Gr_{n,m} ,\qquad QU\mapsto \xi:=D(\Pi\circ\exp_Q)(QU) , \]
where $U\in\ol{\Skew_n}$.

The pair $(Q,U)$ with $U\in\ol{\Skew_n}$ thus 
defines the tangent vector~$\xi:=D(\Pi\circ\exp_Q(QU))$ in~$T_W\Gr_{n,m}$. However, we may also represent the subspace~$W$ by a group element $Qh$, 
for any $h\in H$. A small computation shows that the pair $(Qh,h^{-1}Uh)$ represents the same tangent vector~$\xi$. 
More generally, if we define an equivalence relation~$\sim$ on $O(n)\times H$ via
  \[ (Q,U)\sim(Q',U') \;:\!\iff \exists h\in H \colon (Q',U')=(Qh,h^{-1}Uh) , \]
then it follows that the tangent vector $\xi\in T_W\Gr_{n,m}$ 
can be identified with the equivalence class $[Q,U]:=\{ (Qh,h^{-1}U h)\mid h\in H\}\in (O(n)\times H)/\sim$. 
We have thus obtained the following 
{\em model for $T_W\Gr_{n,m}$}:
\begin{equation}\label{eq:model-TS-Gr} 
T_W\Gr_{n,m} \simeq \left\{[Q,U]\mid Q\in \Pi^{-1}(W) \,,\; 
                                  U\in\ol{\Skew_n}\right\} . 
\end{equation}
Note that in this model, the derivative~$D_Q\Pi$ of~$\Pi$ at~$Q$ is given by
\begin{equation}\label{eq:D_QPi}
  D_Q\Pi(QU)=[Q,\pi(U)] ,
\end{equation}
where $\pi\colon\Skew\to\ol{\Skew_n}$ denotes the orthogonal projection.

We can now define the Riemannian metric on $\Gr_{n,m}$ by setting
\begin{equation}\label{eq:Riem-metr-Gr}
  \big\langle [Q,U_1] \,,\; [Q,U_2] \big\rangle := \langle U_1,U_2\rangle_{I_n}
      \stackrel{\eqref{eq:def-metric-O(n)}}{=} 
        \tfrac{1}{2} \cdot \tr(U_1^T\cdot U_2) ,
\end{equation}
for $U_1,U_2\in\ol{\Skew_n}$ (this is clearly a well-defined $O(n)$-invariant Riemannian metric). 
This way, $\Pi\colon O(n)\to\Gr_{n,m}$ becomes a \emph{Riemannian submersion}, i.e., 
for every $Q\in O(n)$ the restriction of~$D\exp_Q$ to the orthogonal complement of its kernel is 
an isometry. Using the above description of~$D\Pi$ and elementary properties of geodesics (cf.~for example~\cite[Sec.~3.2]{dC}), 
one can show that the exponential map on~$\Gr_{n,m}$ is given by
\begin{equation}\label{eq:exp-Gr(n,m)}
  \exp_W(\xi) = \Pi( \exp_Q(QU) ) = \Pi\big( Q\cdot e^U \big) ,\quad 
                 \xi=[Q,U]\in T_W\Gr_{n,m} .
\end{equation}

\subsection{Frame bundle and Grassmann bundle}\label{sec:bundles}

In the remainder of this section, we discuss some 
properties of frame bundles and Grassmann bundles. 

Let $V$ be a $k$-dimensional Euclidean vector space. 
A {\em frame of $V$} is an orthonormal basis of $V$.
We denote by 
\[ 
 F(V) = \{ (q_1,\ldots,q_k)\in V^k \mid \forall i,j: \langle q_i,q_j\rangle = \delta_{ij} \} 
\]
the set of frames of~$V$.
We can identify $F(\IR^k)$ with the orthogonal group $O(k)$ 
by identifying the frame $(q_1,\ldots,q_k) \in F(\IR^k)$ 
with the orthogonal matrix having the columns $q_1,\ldots,q_k$. 
This way, $F(V)$ gets the structure of a smooth Riemannian manifold. 

Recall the $\ell$th Grassmann manifold $\Gr(V,\ell)$ consisting of the 
$\ell$-dimensional linear subspace of $V$ ($1\leq \ell\leq k$). 
It inherits from $\Gr(\IR^k,\ell)$ the structure of a smooth Riemannian manifold,
since $\Gr(\IR^k,\ell) = \Gr_{k,\ell}$.
Moreover, the smooth surjective map 
$F(V)\to\Gr(V,\ell)$, $(q_1,\ldots,q_k)\mapsto \lin\{q_1,\ldots,q_\ell\}$
provides a close connection between these two manifolds. 

Let $M$ be a $k$-dimensional smooth Riemannian manifold and $1\leq \ell\leq k$.
The \emph{orthonormal frame bundle} and the $\ell$th \emph{Grassmann bundle} 
over~$M$,  respectively,  are defined as
\[ 
 F(M) = \bigcup_{p\in M} \{p\}\times F(T_pM) ,\qquad \Gr(M,\ell) 
          = \bigcup_{p\in M} \{p\}\times \Gr(T_pM,\ell) . 
\]
Using charts of $M$ to connect the manifolds $F(T_pM)$, one can show that 
$F(M)$ has the structure of a smooth manifold. 
Moreover, $F(M)\to M$ is a fiber bundle (cf.~for example~\cite[\S I.5]{KobNom1}). The same holds true
for the Grassmann bundle $\Gr(M,\ell)\to M$. 
Furthermore, the maps of the fibers $F(T_pM)\to \Gr(T_pM,\ell)$ 
may be combined to a smooth surjective bundle map
\begin{equation}\label{eq:def-Pi_b}
  \Pib\colon F(M)\to\Gr(M,\ell) ,\qquad \Pib(p,\bar{Q})=(p,\lin\{q_1,\ldots,q_\ell\}) ,
\end{equation}
where $\bar{Q}=(q_1,\ldots,q_k)$.

For $(p,\bar{Q})\in F(M)$ 
one has a natural decomposition of the tangent space of $F(M)$ at $(p,\bar{Q})$ 
into a \emph{vertical space} and a \emph{horizontal space},
\begin{equation}\label{eq:decomp-horiz-vert}
  T_{(p,\bar{Q})} F(M) = T^v_{(p,\bar{Q})} F(M) \oplus T^h_{(p,\bar{Q})} F(M) .
\end{equation}
The vertical space is defined by
\begin{equation}\label{eq:def-vert}
  T^v_{(p,\bar{Q})} F(M) := T_{(p,\bar{Q})} ( \{p\}\times F(T_pM) ) .
\end{equation}
For defining the horizontal space we consider the following construction. 
Let $\zeta\in T_pM$ and let $c\colon \IR\to M$ be such that $c(0)=p$ and 
$\dot{c}(0)=\zeta$. Furthermore, let $\bar{Q}_t$ denote the \emph{parallel transport}
(cf.~for example~\cite[Thm.~VII.3.12]{booth})
of the frame $\bar{Q}$ along~$c$ at time~$t$. Then the map 
$c_F\colon \IR\to F(M)$, $t\mapsto (c(t),\bar{Q}_t)$, defines a smooth curve in $F(M)$ and thus 
a tangent vector $\dot{c_F}(0)\in T_{(p,\bar{Q})} F(M)$. It can be shown (cf.~for example~\cite[\S II.3]{KobNom1})
that this tangent vector does not depend on the specific choice of the curve~$c$. Moreover, 
one can show that the induced map $T_pM\to T_{(p,\bar{Q})} F(M)$ sending
each tangent vector $\zeta$ to the above defined tangent vector of $F(M)$ 
at $(p,\bar{Q})$, is an injective linear map. The image of this injective 
linear map is defined as the \emph{horizontal space}. One can then show 
the decomposition~\eqref{eq:decomp-horiz-vert} of the tangent space of $F(M)$ at $(p,\bar{Q})$ 
into the direct sum of the vertical and the horizontal space.

Analogous statements about the decomposition of the tangent spaces 
also hold for the Grassmann bundle~$\Gr(M,\ell)$.
Moreover, the derivative of the bundle map $\Pib\colon F(M)\to\Gr(M,\ell)$ certainly maps the vertical spaces of $F(M)$ to the vertical spaces of $\Gr(V,\ell)$. 
As for the horizontal spaces, note that the parallel transport of a linear subspace may be achieved by parallel transporting a basis for this subspace. 
This implies that also the horizontal spaces of $F(M)$ are mapped to the horizontal spaces of $\Gr(V,\ell)$ via the derivative of the bundle map~$\Pib$.

\section{Computation of Normal Jacobians}\label{sec:proof-tube-form}

For the sake of clarity, we first study a more general situation.
Here is a brief outline: 
Let $M\subseteq S^{n-1}$ be a compact hypersurface
with a distinguished unit normal vector field $\nu\colon M\to S^{n-1}$. 
We embed $F(M)$ into $O(n)$ via the map 
$\hat{\Phi}\colon F(M) \to F(\IR^n)=O(n)$ 
that sends a frame $(q_1,\ldots,q_{n-2})$ of $T_pM$ 
to the frame $(p,q_1,\ldots,q_{n-2},\nu(p))$ of $\IR^n$. 
In Section~\ref{se:lifting} we prove that $\hat{\Phi}$ is an 
injective immersion of $F(M)$ into $O(n)$. Hence 
the image $\hat{\Sigma}(M)$ of $F(M)$ under $\hat{\Phi}$ 
is a smooth submanifold  of $O(n)$. 

We are interested in the set $\Sigma(M)$, which is defined as the image of the map 
\[ 
 \Phi_m\colon \Gr(M,m-1)\to \Gr_{n,m} ,\quad (p,Y) \mapsto \IR p + Y . 
\]
The point is that $\Sigma(M) = \SG_m(C)$ in the special case, where $M=\partial K$ 
is the boundary of a smooth $K\in\mcK^\sm(S^{n-1})$ and $C=\cone(K)$. 
The idea is that $\Sigma(M)$ can also be obtained as the image of $\hat{\Sigma}(M)$ 
under the canonical projection 
$\Pi\colon O(n) \to \Gr_{n,m}$ defined in~\eqref{eq:Pi:O(n)->Gr_(n,m)},
see the commutative diagram~\eqref{eq:comm-diag-Sigma_m}. 

In Section~\ref{se:finish-proof-Psi} we prove, via the detour over $\hat{\Phi}$, 
that $\Sigma(M)$ is a smooth hypersurface of $\Gr_{n,m}$ 
in the  case where $M=\partial K$ with $K\in\mcK^\sm(S^{n-1})$. 
This provides the proof of Proposition~\ref{pro:Sigma}.  
We then go on working in this framework to complete the proof of Theorem~\ref{thm:norm-Jac}.

\subsection{Analyzing the situation lifted to the frame bundle}\label{se:lifting}

We fix a compact hypersurface $M\subseteq S^{n-1}$ 
with a unit normal vector field $\nu\colon M\to S^{n-1}$. 

In the following, we identify the tangent space $T_pM$ for $p\in M$ with the linear subspace $T_pM=p^\bot\cap \nu(p)^\bot$ of~$\IR^n$, 
which has codimension~$2$. If $(p,\bar{Q})\in F(M)$, then $\bar{Q}=(q_1,\ldots,q_{n-2})$ is an orthonormal basis of $T_pM$, 
and in the following we shall interpret~$\bar{Q}$ as the matrix in~$\IR^{n\times (n-2)}$ with the columns $q_i\in\IR^n$. 
Note that $(p,\bar{Q},\nu(p))=(p,q_1,\ldots,q_{n-2},\nu(p))\in O(n)$. This defines the map
  \[ \hat{\Phi}\colon F(M)\to O(n) ,\quad (p,\bar{Q}) \mapsto Q = (p,\bar{Q},\nu(p)) . \]
We define the lifted Sigma set~$\hat{\Sigma}(M)$ as the image of the map $\hat{\Phi}$: 
  \[ \hat{\Sigma} := \hat{\Sigma}(M) := \hat{\Phi}(F(M)) = \{ Q\in O(n)\mid \exists p\in M : Q\cdot e_1=p \,,\; Q\cdot e_n=\nu(p) \} , \]
where $e_i\in\IR^n$ denote the canonical basis vectors. Analogous to the definition of~$\hat{\Phi}$, 
we define for $1\leq m\leq n-1$ the map 
  \[ \Phi_m\colon \Gr(M,m-1)\to \Gr_{n,m} ,\quad (p,Y) \mapsto \IR p + Y , \]
where~$Y\in\Gr(T_pM,m-1)$ is interpreted as an $(m-1)$-dimensional subspace of~$\IR^n$. Furthermore, we define
  \[ \Sigma_m := \Sigma_m(M) := \Phi_m(\Gr(M,m-1)) = \{ W\in\Gr_{n,m}\mid \exists p\in M : p\in W \,,\; W\subseteq \nu(p)^\bot \} . \]
Note that we have $\Pi(\hat{\Sigma}(M))=\Sigma_m$, where $\Pi\colon O(n)\to\Gr_{n,m}$ denotes the canonical projection (cf.~\eqref{eq:Pi:O(n)->Gr_(n,m)}), 
but we have a strict inclusion $\hat{\Sigma}(M)\subsetneq \Pi^{-1}(\Sigma_m(M))$. 

The following commutative diagram provides an overview over the relations, 
which are central for the understanding of $\Sigma_m$:
\begin{equation}\label{eq:comm-diag-Sigma_m}
\begin{array}[c]{c}
\begin{tikzpicture}[>=stealth]
\matrix (m) [matrix of math nodes, column sep=10mm, row sep=12mm, 
             text height=1.5ex, text depth=0.25ex]
  {    F:=F(M) & \hat{\Sigma} & O(n) \\
       G:=\Gr(M,m-1) & \Sigma_m & \Gr_{n,m} \\
  };
\path[->>]
  (m-1-1) edge node[auto]{$\hat{\Phi}$} (m-1-2)
  (m-2-1) edge node[auto]{$\Phi_m$} (m-2-2);
\path[->>]
  (m-1-1) edge node[auto]{$\Pib$} (m-2-1)
  (m-1-2) edge node[auto]{$\PiS$} (m-2-2)
  (m-1-3) edge node[auto]{$\Pi$} (m-2-3);
\path[right hook->]
  (m-1-2) edge (m-1-3)
  (m-2-2) edge (m-2-3);
\end{tikzpicture}
\end{array} .
\end{equation}
Here $\Pib$ is defined as in~\eqref{eq:def-Pi_b}, and $\PiS$ is defined as the restriction of $\Pi$ (cf.~\eqref{eq:Pi:O(n)->Gr_(n,m)}) to $\hat{\Sigma}$. 
Note that if $(p,\bar{Q})\in F(M)$ with $\bar{Q}=(q_1,\ldots,q_{n-2})$, then $\hat{\Phi}(p,\bar{Q})=Q=(p,q_1,\ldots,q_{n-2},\nu(p))\in\hat{\Sigma}$ 
and $\Pib(p,\bar{Q})=(p,Y)$, where $Y$ is the span of $q_1,\ldots,q_{m-1}$.

The tangent spaces of the fiber bundles $F(M)$ and $\Gr(M,m-1)$ have natural 
decompositions into \emph{vertical} and \emph{horizontal components} (cf.~Section~\ref{sec:bundles}): 
\begin{equation}\label{eq:vert-horiz}
  T_{(p,\bar{Q})} F = T_{(p,\bar{Q})}^v F \oplus T_{(p,\bar{Q})}^h F ,\qquad T_{(p,Y)} G = T_{(p,Y)}^v G \oplus T_{(p,Y)}^h G .
\end{equation}
We denote the images of the vertical and the horizontal space of the frame bundle at $(p,\bar{Q})$ under the derivative $D\hat{\Phi}$ by
\begin{equation}\label{eq:T(hat(Sigma))-v-h}
   T_Q^v \hat{\Sigma} := D\hat{\Phi}\left(T_{(p,\bar{Q})}^v F\right) ,\qquad T_Q^h \hat{\Sigma} := D\hat{\Phi}\left(T_{(p,\bar{Q})}^h F\right) .
\end{equation}
We will give next concrete descriptions of these spaces.

Before we state the next result about the derivative of $\hat{\Phi}$, 
we define the following linear maps for $Q=(p,\bar{Q},\nu(p))\in \hat{\Sigma}$
\begin{align}\label{eq:hat(E)^v-hat(E)^h_Q}
   \hat{E}^v & \colon \Skew_{n-2} \to\Skew_n , & \hat{E}^h & \colon \IR^{n-2} \to\Skew_n ,
\\[1mm] \hat{E}^v(\bar{U}) & :=
    \left(\begin{array}{c|c|c}
   \scriptstyle0 & 0 & \scriptstyle0
\\\hline 0 & \raisebox{-3mm}{\rule{0mm}{9mm}}\rule{3mm}{0mm}\bar{U}\rule{3mm}{0mm} & 0
\\\hline \scriptstyle0 & 0 & \scriptstyle0
\end{array}\right) , & \hat{E}^h(a) & :=
   \left(\begin{array}{c|c|c}
   \scriptstyle0 & -a^T & \scriptstyle0
\\\hline a & \raisebox{-3mm}{\rule{0mm}{9mm}}\rule{3mm}{0mm}0\rule{3mm}{0mm} & -b
\\\hline \scriptstyle0 & b^T & \scriptstyle0
\end{array}\right) ,\; \text{where } b:=\bar{Q}^T\cdot \W_p(\bar{Q}a) . \nonumber
\end{align}
Here $\W_p\colon T_pM\to T_pM$ is the 
Weingarten map of~$M$ at~$p$, cf.~\eqref{eq:weing-hypersurf}, and we recall that the image of $\bar{Q}\in\IR^{n\times (n-2)}$ equals $T_pM$, 
since $(p,\bar{Q},\nu(p))\in\hat{\Sigma}$. Note that $\hat{E}^v$ is an 
isometric embedding independent of the manifold~$M$ and the element $Q\in\hat{\Sigma}$, 
while $\hat{E}^h$ is a linear injection, which depends both on~$M$ and on~$Q$. For the sake of simplicity, we do not reflect this dependence in the notation.
Note further that the images of $\hat{E}^v$ and $\hat{E}^h$ are 
orthogonal subspaces of $\Skew_n$.

\begin{lemma}\label{prop:Dhat(Phi)}
Let $(p,\bar{Q})\in F(M)$ and $Q:=\hat{\Phi}(p,\bar{Q})=
(p,\bar{Q},\nu(p))\in\hat{\Sigma}$. 
Then we have
\begin{equation}\label{eq:Dhat(Phi)}
  T_Q^v \hat{\Sigma} = Q\cdot \im(\hat{E}^v) ,\qquad T_Q^h \hat{\Sigma} 
                     = Q\cdot \im(\hat{E}^h) .
\end{equation}
\end{lemma}

\begin{proof}
Since $p=Qe_1$ and $\nu(p)=Qe_n$, the image of the fiber $\{p\}\times F(T_pM)$ under the map $\hat{\Phi}$ is given by 
  \[ \hat{\Phi}(\{p\}\times F(T_pM)) = \left\{ \left. Q\cdot 
     \left(\begin{array}{c|c|c}
   \scriptstyle1 & 0 & \scriptstyle0
\\\hline 0 & \raisebox{-3mm}{\rule{0mm}{9mm}}\rule{3mm}{0mm}\tilde{Q}\rule{3mm}{0mm} & 0
\\\hline \scriptstyle0 & 0 & \scriptstyle1
\end{array}\right)\right| \tilde{Q}\in O(n-2)\right\} . \]
This implies that the image of the vertical space $T_{(p,\bar{Q})}^v F(M)$ under the derivative 
$D\hat{\Phi}$ is given by $T_Q^v \hat{\Sigma} = Q\cdot \im(\hat{E}^v)$, cf.~\eqref{eq:def-vert}.

As for the horizontal space, let $\xi\in T_{(p,\bar{Q})}^h F(M)$ be 
represented by the curve $c_F\colon\IR\to F(M)$, i.e., $c_F(0)=(p,\bar{Q})$ 
and $\dot{c_F}(0)=\xi$, which is given in the following way: 
$c_F(t)=(c(t),\bar{Q}_t)$, where $\bar{Q}_t$ is the parallel transport of 
$\bar{Q}$ along the curve~$c$ (cf.~Section~\ref{sec:bundles}). 
Let us denote by~$\zeta:=\dot{c}(0)\in T_pM$ the tangent vector at~$p$ defined by~$c$. The image of the curve~$c_F$ under~$\hat{\Phi}$ is thus given by 
$\hat{\Phi}(c_F(t))=(c(t), \bar{Q}_t, \nu(c(t)))=:Q(t)$, 
and the image of $\xi$ is given by $D\hat{\Phi}(\xi)=\dot{Q}(0)$. 
It is sufficient to show that $\dot{Q}(0)$ is given by
\begin{equation}\label{eq:dot(Q)(0)}
  \dot{Q}(0) = Q\cdot \left(\begin{array}{c|c|c}
   \scriptstyle0 & -a^T & \scriptstyle0
\\\hline a & \raisebox{-3mm}{\rule{0mm}{9mm}}\rule{3mm}{0mm}0\rule{3mm}{0mm} & -b
\\\hline \scriptstyle0 & b^T & \scriptstyle0
\end{array}\right) ,\; \text{where} \; 
   a = \bar{Q}^T \zeta ,\quad b = \bar{Q}^T \mcW_p(\zeta) .
\end{equation}

To prove~\eqref{eq:dot(Q)(0)}, note first that as $Q(t)\in O(n)$ and $Q(0)=Q$, 
we have $\dot{Q}(0)=Q\cdot U$ with $U\in\Skew_n$. Recall that the columns of $\bar{Q}$ form an orthonormal basis of $T_pM=p^\bot\cap \nu(p)^\bot$, 
hence $Q^T T_pM=e_1^\bot\cap e_n^\bot$. Therefore, $\dot{c}(0)=\zeta=\bar{Q}a$ for some $a\in\IR^{n-2}$, which implies $a=\bar{Q}^T\zeta$. 
The first column of $U$ is thus given by $Ue_1=Q^T\dot{Q}(0)e_1=Q^T\zeta 
=\left(\begin{smallmatrix} 0 \\ \textstyle a \\ 0\end{smallmatrix}\right)$.
By skew-symmetry of~$U$, this also gives us the first row of~$U$. The zero matrix 
in the middle follows from the fact that the frame $\bar{Q}$ is parallel 
transported along~$c$ (cf.~\cite[\S VII.3]{booth}). 
Finally, the last column of $\dot{Q}(0)$ is given by $D_p\nu(\zeta)$.
This implies that the last column of $U$ is given by
  \[ Q^T D_p\nu(\zeta) \stackrel{\eqref{eq:weing-hypersurf}}{=} -Q^T \mcW_p(\zeta) . \]
As $\mcW_p(\zeta)\in T_pM = p^\bot\cap \nu(p)^\bot$, the last column of $U$ has the form 
$\left(\begin{smallmatrix} 0 \\ \textstyle -b \\ 0\end{smallmatrix}\right)$.  
Therefore $\mcW_p(\zeta) 
= Q \left(\begin{smallmatrix} 0 \\ \textstyle b \\ 0\end{smallmatrix}\right) = \bar{Q}b$, which implies $b=\bar{Q}^T \mcW_p(\zeta)$.
The last row follows again by skew-symmetry of~$U$.
\end{proof}

\begin{corollary}\label{cor:T-hat(Sigma)}
The map $\hat{\Phi}$ is an injective immersion of $F(M)$ into $O(n)$. 
In particular, $\hat{\Sigma}$ is a smooth submanifold of $O(n)$ of codimension $n-1$, and the tangent space $T_Q\hat{\Sigma}$ at $Q\in O(n)$ 
has the orthogonal decomposition $T_Q\hat{\Sigma}=T_Q^v \hat{\Sigma} \oplus T_Q^h \hat{\Sigma}$.
\end{corollary}

\begin{proof}
The fact that $\hat{\Phi}$ is a smooth injective map is obvious. Furthermore, 
by Lemma~\ref{prop:Dhat(Phi)}, it follows that the derivative 
$D\hat{\Phi}$ at $(p,\bar{Q})$ has full rank. The map $\hat{\Phi}$ is thus an injective immersion, 
and as the domain $F(M)$ is compact, it is also an embedding. 
As for the dimension, we compute
\begin{align*}
   \dim \hat{\Sigma} & = \dim F(M) = \dim M + \dim O(n-2) 
     = n-2+\tfrac{(n-2)(n-3)}{2} = \tfrac{(n-2)(n-1)}{2}
\\ & = \dim O(n)-(n-1) .
\end{align*}
The decomposition of the tangent space into the direct sum $T_Q\hat{\Sigma}=T_Q^v \hat{\Sigma} \oplus T_Q^h \hat{\Sigma}$ 
follows from the decomposition of the tangent space of the fiber bundle~\eqref{eq:vert-horiz}. 
The fact that $T_Q^v \hat{\Sigma}$ and $T_Q^h \hat{\Sigma}$ are orthogonal follows from the description given in Lemma~\ref{prop:Dhat(Phi)}.
\end{proof}

Now that we have a clear description of the lifted Sigma set $\hat{\Sigma}$, 
we will transfer this description to $\Sigma_m$ via the 
projection map~$\PiS\colon\hat{\Sigma}\to\SG_m$ defined in~\eqref{eq:comm-diag-Sigma_m}.  
Analogously to the lifted case, 
we denote the images under the derivative $D\Phi_m$ 
of the vertical and the horizontal space, cf.~\eqref{eq:vert-horiz}, 
of the Grassmann bundle at $(p,Y)$ by
\begin{equation}\label{eq:T(Sigma_m)-v-h}
   T_W^v \Sigma_m := D\Phi_m\left(T_{(p,Y)}^v G\right) ,\qquad T_W^h \Sigma_m := D\Phi_m\left(T_{(p,Y)}^h G\right) .
\end{equation}

Recall that the subspace $\ol{\Skew_n}$ defined in~\eqref{eq:bar(Skew_n)} serves as a model for the tangent spaces of~$\Gr_{n,m}$. 
We consider the linear map
\begin{equation}\label{eq:def-alpha}
  \alpha\colon \IR^{n-2}\to\ol{\Skew_n} ,\quad x=\begin{pmatrix} x_1 \\ x_2\end{pmatrix} \mapsto \left(\begin{array}{c|c|c|c}
   \multicolumn{2}{c|}{\multirow{2}{*}{$0$}} & -x_2^T & \scriptstyle0 
   \\\cline{3-4} \multicolumn{2}{c|}{} & 0 & -x_1
\\\hline x_2 & 0 & \multicolumn{2}{c}{\multirow{2}{*}{$0$}}
   \\\cline{1-2} \rule{0mm}{4mm}\scriptstyle0 & \rule{1mm}{0mm}x_1^T\rule{1mm}{0mm}
\end{array}\right) ,
\end{equation}
where $x_1\in\IR^{m-1}$ and $x_2\in\IR^{n-m-1}$, and we denote its image by
  \[ S_n := \im(\alpha)\subseteq\ol{\Skew_n} . \]
Note that $\dim S_n = n-2$. We now define the following linear maps for $Q=(p,\bar{Q},\nu(p))\in \hat{\Sigma}$: 
\begin{align}\label{eq:E^v-E^h_Q}
   E^v & \colon \IR^{(n-m-1)\times(m-1)} \to\ol{\Skew_n} , & E^h & \colon \IR^{n-2} \to S_n , 
\\[1mm] E^v(X) & :=
    \left(\begin{array}{c|c|c|c}
   \multicolumn{2}{c|}{\multirow{2}{*}{$0$}} & 0 & \scriptstyle0 
   \\\cline{3-4} \multicolumn{2}{c|}{} & \rule{0mm}{4mm}-X^T & 0
\\\hline 0 & \rule{0mm}{4mm}X & \multicolumn{2}{c}{\multirow{2}{*}{$0$}}
   \\\cline{1-2} \scriptstyle0 & 0
\end{array}\right) , & E^h(a) & :=
   \left(\begin{array}{c|c|c|c}
   \multicolumn{2}{c|}{\multirow{2}{*}{$0$}} & -a_2^T & \scriptstyle0 
   \\\cline{3-4} \multicolumn{2}{c|}{} & 0 & -b_1
\\\hline a_2 & 0 & \multicolumn{2}{c}{\multirow{2}{*}{$0$}}
   \\\cline{1-2} \rule{0mm}{4mm}\scriptstyle0 & \rule{1mm}{0mm}b_1^T\rule{1mm}{0mm}
\end{array}\right) , \nonumber
\end{align}
where again $b:=\bar{Q}^T\cdot \W_p(\bar{Q}a)$, and $a
=\left(\begin{smallmatrix} \textstyle a_1\\[1pt]\textstyle a_2\end{smallmatrix}\right)$, $b
=\left(\begin{smallmatrix} \textstyle b_1\\[1pt]\textstyle b_2\end{smallmatrix}\right)$ 
with $a_1,b_1\in\IR^{m-1}$, $a_2,b_2\in\IR^{n-m-1}$. The map $E^v$ is an isometry if the scalar product on $\ol{\Skew_n}$ 
induced from $\IR^{n\times n}$ is scaled by~$\frac{1}{2}$. Note that $E^v$ is independent of~$M$ and~$Q$. 
On the other hand, as in the lifted case, the map $E^h$ depends on both~$M$ and~$Q$. Unlike in the lifted case, the map $E^h$ need not be injective. 
We will see below in Proposition~\ref{prop:TSigma_m} that it is injective iff the restriction~$\W_{p,Y}$ of 
the Weingarten map~$\W_p$ to the subspace $Y=\lin\{q_1,\ldots,q_{m-1}\}\subseteq T_pM$, where $Q=(p,q_1,\ldots,q_{n-2},\nu(p))$, has full rank. 
If this is the case, then $\im E^h=S_n$. 

Note that we have an orthogonal decomposition of $\ol{\Skew_n}$ into
\begin{equation}\label{eq:decomp-ol(Skew_n)}
  \ol{\Skew_n} = \im E^v \,\oplus\, S_n \,\oplus\, \IR N ,
\end{equation}
where $N\in\ol{\Skew_n}$ was defined in~\eqref{eq:def-N}. If we define $\iota\colon\IR^{(n-m-1)\times(m-1)} \to\Skew_{n-2}$, 
via $X\mapsto \big(\begin{smallmatrix} 0 & -X^T \\ X & 0\end{smallmatrix}\big)$, and 
if we denote by $\pi\colon\Skew_n\to\ol{\Skew_n}$ the orthogonal projection, then we have the crucial relation
\begin{equation}\label{eq:E^h=pi(hat(E^h))}
  E^v(X) = \hat{E}^v(\iota(X)) ,\qquad E^h(a) = \pi(\hat{E}^h(a)) .
\end{equation}

\begin{proposition}\label{prop:TSigma_m}
\begin{enumerate}
  \item Let $(p,Y)\in \Gr(M,m-1)$ and $W:=\Phi_m(p,Y)$, and let $Q\in\hat{\Sigma}$ 
        be such that $\PiS(Q)=W$. Then we have, using the model of tangent spaces of $\Gr_{n,m}$ 
       described in~\eqref{eq:model-TS-Gr}: 
  \begin{equation}\label{eq:T_W^v(Sigma_m)}
    T_W^v \Sigma_m = \big\{ [Q,U]\mid U\in\im(E^v)\big\} ,\qquad 
    T_W^h \Sigma_m = \big\{ [Q,U]\mid U\in\im(E^h)\big\} .
  \end{equation}
  \item The rank of the map $E^h$ is given by 
     \[ \rk(E^h) = n-m-1+\rk \W_{p,Y} \, , \]
        where $\W_{p,Y}\colon Y\to Y$ is the restriction of the Weingarten map $\W_p$ to the subspace~$Y$ of~$T_pM$. 
Moreover, if $\W_{p,Y}$ has full rank and if we consider $S_n$ to be endowed with the scalar product 
$\langle U_1,U_2\rangle := \frac{1}{2}\cdot \tr(U_1^T\cdot U_2)$, then
  \begin{equation}\label{eq:ndet(E^h)}
    |\det(E^h)|=|\tdet_Y(\W_p)| ,
  \end{equation}
        where $\det_Y(\W_p)$ denotes the determinant of $\W_{p,Y}$ (cf.~Definition~\ref{def:tw-char-pol}).
  \item The rank of the derivative $D_{(p,Y)}\Phi_m$ is given by
     \[ \rk D_{(p,Y)}\Phi_m = m(n-m) - m + \rk \W_{p,Y} , \]
\end{enumerate}
\end{proposition}

\begin{proof}
(1) Since the derivative of the bundle projection $\Pib\colon F(M)\to\Gr(M,\ell)$ 
defined in~\eqref{eq:def-Pi_b} maps the vertical spaces of 
the frame bundle to the vertical spaces of the Grassmann bundle 
(cf.~Section~\ref{sec:bundles}), we have
\begin{align*}
   T_W^v \Sigma_m & \stackrel{\eqref{eq:T(Sigma_m)-v-h}}{=} D\Phi_m\left(T_{(p,Y)}^v G\right) \stackrel{\eqref{eq:comm-diag-Sigma_m}}{=} D\PiS(T_Q^v \hat{\Sigma}) 
\stackrel{\eqref{eq:Dhat(Phi)}}{=} D\PiS(Q\cdot\im(\hat{E}^v)) .
\end{align*}
Using $D_Q\PiS(QU)=[Q,\bar{U}]$ (cf.~\eqref{eq:D_QPi}), we may continue as
\begin{align*}
   T_W^v \Sigma_m & = \big\{ [Q,U]\mid U\in \pi(\im(\hat{E}^v))\big\}
 = \big\{ [Q,U]\mid U\in\im(E^v)\big\} .
\end{align*}
Similar arguments show the description~\eqref{eq:T_W^v(Sigma_m)} of the horizontal space. 

(2) By the above choice of scalar product in $S_n$, the map $\alpha$ defined in~\eqref{eq:def-alpha} is an isometry on its image. 
Hence, instead of considering $E^h$ we may focus on the map
  \[ \alpha^{-1}\circ E^h\colon\IR^{n-2}\to\IR^{n-2} ,\quad a=\begin{pmatrix} a_1 \\ a_2\end{pmatrix} \mapsto \begin{pmatrix} b_1 \\ a_2\end{pmatrix} , \]
where as usual $b=\bar{Q}^T\cdot \W_p(\bar{Q}a)$, and $a
=\left(\begin{smallmatrix} \textstyle a_1\\[1pt]\textstyle a_2\end{smallmatrix}\right)$, $b
=\left(\begin{smallmatrix} \textstyle b_1\\[1pt]\textstyle b_2\end{smallmatrix}\right)$ with $a_1,b_1\in\IR^{m-1}$, $a_2,b_2\in\IR^{n-m-1}$.
 Let $\Lambda\in\IR^{(n-2)\times (n-2)}$ be the representation matrix of $\W_p$ with respect to the orthonormal basis $q_1,\ldots,q_{n-2}$ of $T_pM$, 
i.e., $\Lambda a =\bar{Q}^T\cdot \W_p(\bar{Q}a)=b$. Furthermore, let $\Lambda$ be decomposed as $\Lambda = \begin{pmatrix} \Lambda_1 & \Lambda_2 
\\ \Lambda_3 & \Lambda_4 \end{pmatrix}$, where $\Lambda_1\in\IR^{(m-1)\times (m-1)}$
and the other blocks accordingly. Then the representation matrix of $\alpha^{-1}\circ E^h$ is given by 
$\begin{pmatrix} \Lambda_1 & \Lambda_2 \\ 0 & I_{n-m-1}\end{pmatrix}$. As $\Lambda_1$ is 
the representation matrix of $\W_{p,Y}$ with respect to the orthonormal basis $q_1,\ldots,q_{m-1}$ of $Y$, we obtain
  \[ \rk(E^h)=\rk(\alpha^{-1}\circ E^h)=n-m-1+\rk(\Lambda_1)=n-m-1+\rk(\W_{p,Y}) . \]
Furthermore, if $\W_{p,Y}$ has full rank $\rk \W_{p,Y}=m-1$, then
  \[ |\det(E^h)|=|\det(\alpha^{-1}\circ E^h)|=|\det(\Lambda_1)|=|\tdet_Y(\W_p)| . \]

(3) From~(1) and~(2) we obtain
\begin{align*}
   \dim T_W^v\Sigma_m & = \rk(E^v)=(m-1)(n-m-1) ,
\\ \dim T_W^h\Sigma_m & = \rk(E^h)=n-m-1 + \rk \W_{p,Y} .
\end{align*}
This implies that the rank of the derivative of $\Phi_m$ is given by
  \[ \rk D_{(p,Y)}\Phi_m = \dim T_W^v\Sigma_m + \dim T_W^h\Sigma_m
                       = m(n-m-1) + \rk \W_{p,Y} . \qedhere \]
\end{proof}

\subsection{Specializing to the boundary of a convex set}\label{se:finish-proof-Psi}

We assume now that  $M=\partial K$ for some $K\in\mcK^\sm(S^{n-1})$. 

The subsequent corollary implies Proposition~\ref{pro:Sigma},  
as well as the claim about the normal Jacobian of $\PiM$ 
in Theorem~\ref{thm:norm-Jac}. 

\begin{corollary}\label{cor:Sigma_m}
Let $M=\partial K$ with $K\in\mcK^\sm(S^{n-1})$.
\begin{enumerate}
  \item Then $\Phi_m$ is an injective immersion of $\Gr(M,m-1)$ into $\Gr_{n,m}$, 
  and $\Sigma_m$ is a smooth hypersurface of $\Gr_{n,m}$. Furthermore, for $W\in\Sigma_m$, the tangent space of $\Sigma_m$ at $W$ has 
  the orthogonal decomposition $T_W\Sigma_m=T_W^v\Sigma_m\oplus T_W^h\Sigma_m$.

  \item Recall $N\in\ol{\Skew_n}$ defined in~\eqref{eq:def-N}. 
  If $W=\PiS(Q)\in\Sigma_m$ for $Q\in\hat{\Sigma}$, then 
  $\nuS(W) = [Q,N] \in T_W\Gr_{n,m}$ is a unit normal vector of $\Sigma_m$ at $W$ 
  that points into~$\PG_m(C)$. Moreover, $-\nuS(W)$ points into $\DG_m(C)$. 

  \item Consider $\PiM \colon \Sigma_m \to M, \, W \mapsto p$, where $W\cap K=\{p\}$, cf.~\eqref{eq:def-Pi_m-Psi}. 
  Then the normal Jacobian of $\PiM$ at $W\in\Sigma_m$ is given by 
  $\ndet(D_W\PiM) = \tdet_Y(\W_p)^{-1}$, where $Y:=W\cap p^\bot\in\Gr(T_pM,m-1)$. 
\end{enumerate}
\end{corollary}

\begin{proof}
(1) The injectivity of~$\Phi_m$ follows from Proposition~\ref{prop:WcapK=p}.
Furthermore, the Weingarten map $\mcW_p$ of $M$ is positive definite for every $p\in M$. 
Therefore, the restriction $\mcW_{p,Y}$ of $\W_p$ to any subspace $Y\in\Gr(T_pM,m-1)$ 
is positive definite and in particular has full rank. It follows from 
Proposition~\ref{prop:TSigma_m} that $\Phi_m$ is an injective immersion. 
By compactness of the domain $\Gr(M,m-1)$, it follows that $\Phi_m$ is 
an embedding. The image $\Sigma_m$ is thus a smooth submanifold of 
dimension $\dim\Sigma_m=\dim\Gr(M,m-1)=m(n-m)-1$. 
The claimed orthogonal decomposition of $T_W\Sigma_m$ is a consequence of~\eqref{eq:T_W^v(Sigma_m)} and~\eqref{eq:decomp-ol(Skew_n)}.

(2) It is clear from the orthogonal decompositions $T_W\Sigma_m=T_W^v\Sigma_m\oplus T_W^h\Sigma_m$ 
and~\eqref{eq:decomp-ol(Skew_n)} that $\nuS(W)$ lies 
in the orthogonal complement of $T_W\Sigma_m$ in $T_W\Gr_{n,m}$. 
Consider the geodesic   
\[ 
 W_\rho := \exp_W(\rho\cdot\nuS(W)) \stackrel{\eqref{eq:exp-Gr(n,m)}}{=} 
     \Pi( \exp_Q(QN) ) \stackrel{\eqref{eq:exp_Q(QU)}}{=} \Pi( Q\cdot Q_\rho )
\]
through $W$ in direction $\nuS(W)$. 

The first column of $Q\cdot Q_\rho$ is given by 
$p_\rho := p\cos\rho - \nu(p)\sin\rho$.
In particular, $p_\rho \in W_\rho$. 
By assumption, the normal vector $\nu(p)$ points inwards~$K$. 
Therefore, for all $\rho<0$ with $|\rho|$ small enough, we have $p_\rho\in K$,  
which implies $W_\rho \in \DG_m(C)$. 
Hence $-\nuS(W)$ points into $\DG_m(C)$.

On the other hand, the orthogonal complement of $\nu(p)$ is a supporting hyperplane 
of $\cone(K)$. From this it follows that $W_\rho \cap K =\{0\}$ 
for all  sufficiently small $\rho>0$.
Therefore,  $W_\rho \in \Gr_{n,m}\setminus\DG_m(C) \subseteq \PG_m(C)$. 
Hence $\nuS(W)$ points into $\PG_m(C)$.

(3) We may lift the function $\PiM\colon \Sigma_m\to M$ to the function $\PiMh\colon \hat{\Sigma}\to M$, $Q\mapsto p$, 
where $Q=(p,q_1,\ldots,q_{n-2},\nu(p))$. So we have the following commutative diagram
\begin{equation}\label{eq:PiMh-PiM}
\begin{array}[c]{c}
\begin{tikzpicture}[>=stealth]
\matrix (m) [matrix of math nodes, column sep=10mm, row sep=3mm, 
             text height=1.5ex, text depth=0.25ex]
  {    \hat{\Sigma} & \Sigma_m \\[6mm]
       & M \\
  };
\path[->>]
  (m-1-1) edge node[auto]{$\PiS$} (m-1-2)
  (m-1-2) edge node[auto]{$\PiM$} (m-2-2)
  (m-1-1) edge node[below left=-1.5mm]{$\PiMh$} (m-2-2);
\end{tikzpicture}
\end{array} ,\quad 
\begin{array}[c]{c}
\begin{tikzpicture}[>=stealth]
\matrix (m) [matrix of math nodes, column sep=13mm, row sep=3mm, 
             text height=1.5ex, text depth=0.25ex]
  {    Q & W \\[6mm]
       & p \\
  };
\path[|->]
  (m-1-1) edge (m-1-2)
  (m-1-2) edge (m-2-2)
  (m-1-1) edge (m-2-2);
\end{tikzpicture}
\end{array}  ,
\end{equation}
where $W=\lin\{p,q_1,\ldots,q_{m-1}\}$. We next show that the orthogonal complement of $\ker D_W\PiM$ in $T_W\Sigma_m$ 
is given by the horizontal space $T_W^h\Sigma_m$.

As the map $\PiMh$ is just the projection onto the first column, the kernel of $D_Q\PiMh$ is given by
  \[ \ker D_Q\PiMh = \{ U\in T_Q\hat{\Sigma} \mid 
QU e_1=0 \} . \]
It follows that $T_Q^v\hat{\Sigma}\subseteq \ker D_Q\PiMh$ and $T_Q^h\hat{\Sigma}\cap \ker D_Q\PiMh=0$.
Using the commutative diagram~\eqref{eq:PiMh-PiM}, we obtain $T_W^v\Sigma_m\subseteq \ker D_W\PiM$ and $T_W^h \Sigma_m\cap \ker D_W\PiM=0$. 
The fact that $T_W\Sigma_m=T_W^v\Sigma_m\oplus T_W^h\Sigma_m$ is an orthogonal decomposition implies that $T_W^v\Sigma_m=\ker D_W\PiM$ 
and hence $T_W^h\Sigma_m=(\ker D_W\PiM)^\bot$.

In order to compute $D_W\PiM$ on the horizontal space $T_W^h\Sigma_m$, we consider for fixed $Q=(p,q_1,\ldots,q_{n-2},\nu(p))\in\hat{\Sigma}$ 
the following diagram
  \[ \begin{array}[c]{c}
\begin{tikzpicture}[>=stealth]
\matrix (m) [matrix of math nodes, column sep=10mm, row sep=3mm, 
             text height=1.5ex, text depth=0.25ex]
  {    \IR^{n-2} & S_n & T_W^h\Sigma_m \\
       \rotatebox{90}{$=$} \\
       \IR^{n-2} & & T_pM \\
  };
\path[->]
  (m-1-1) edge node[auto]{$E^h$} (m-1-2)
  (m-1-2) edge node[auto]{$\beta$} (m-1-3)
  (m-1-3) edge node[auto]{$D_W\PiM$} (m-3-3)
  (m-3-1) edge node[auto]{$\gamma$} (m-3-3);
\end{tikzpicture}
\end{array} , \]
where $\beta(U)=[Q,U]$ for $U\in S_n$, and $\gamma(a)=\sum_{i=1}^{n-2} a_i\cdot q_i$ for $a=(a_1,\ldots,a_{n-2})^T\in\IR^{n-2}$. 
Let us check that this diagram is commutative: using
\[ D_Q\PiS(Q\cdot \hat{E}^h(a)) \stackrel{\eqref{eq:D_QPi}}{=} [Q,\pi(\hat{E}^h(a))] \stackrel{\eqref{eq:E^h=pi(hat(E^h))}}{=} [Q,E^h(a)] , \]
we get
  \[ D_W\PiM([Q,E^h(a)]) = D_W\PiM(D_Q\PiS(Q\cdot \hat{E}^h(a))) \stackrel{\eqref{eq:PiMh-PiM}}{=} D_Q\PiMh(Q\cdot \hat{E}^h(a)) = \gamma(a) 
  . \]
Since $\beta$ and $\gamma$ are isometric we obtain
  \[ \ndet(D_W\PiM) = |\det(E^h)|^{-1} \stackrel{\eqref{eq:ndet(E^h)}}{=} |\tdet_Y(\W_p)|^{-1} = \tdet_Y(\mcW_p)^{-1}  , \]
where the last equality follows from the positive definiteness of $\mcW_p$.
\end{proof}

\begin{proof}[Proof of Theorem~\ref{thm:norm-Jac}]
It remains to show the claim about the Jacobian of $\Psi$. 

Let us make yet another definition in the lifted setting. 
For $Q\in\hat{\Sigma}$ we define the direction 
$\hat{\nu}(Q)\in T_Q O(n)$ and the map 
$\hat{\Psi}\colon \hat{\Sigma}\times\IR\to O(n)$ via
  \[ \hat{\nu}(Q) := Q \cdot N ,\qquad \hat{\Psi}(Q,t) 
                   = \exp_Q(\arctan t\cdot \hat{\nu}(Q)) , \]
where $N\in\Skew_n$ is, as usual, defined in~\eqref{eq:def-N}.
Abbreviating $\rho:=\arctan t$, we obtain from~\eqref{eq:exp_Q(QU)} that the map $\hat{\Psi}$ is given by 
$\hat{\Psi}(Q,t) = Q\cdot Q_\rho$, where $Q_\rho$ was defined 
in~\eqref{eq:exp_Q(QU)}.
Furthermore, denoting $\tilde{\Pi}:=\PiS\times\id_\IR$, we get the following commutative diagram
\begin{equation}\label{eq:underst-T(Sigma)}
\begin{array}[c]{c}
\begin{tikzpicture}[>=stealth]
\matrix (m) [matrix of math nodes, column sep=14mm, row sep=11mm, 
             text height=1.5ex, text depth=0.25ex]
  {    \hat{\Sigma}\times \IR & O(n) \\
       \Sigma_m\times \IR & \Gr_{n,m} \\
  };
\path[->]
  (m-1-1) edge node[auto]{$\hat{\Psi}$} (m-1-2);
\path[->>]
  (m-2-1) edge node[auto]{$\Psi$} (m-2-2)
  (m-1-1) edge node[auto]{$\tilde{\Pi}$} (m-2-1)
  (m-1-2) edge node[auto]{$\Pi$} (m-2-2);
\end{tikzpicture}
\end{array} .
\end{equation}

We will compute the derivative $D\Psi$ via the lifting $D\hat{\Psi}$ in the following way. 
Let $W\in\Sigma_m$ and let 
$Q=(p,\bar{Q},\nu(p))\in \hat{\Sigma}$ be 
a lifting of~$W$, i.e., $W=\PiS(Q)$. For $(\xi,\dot{t})\in T_{(W,t)}(\Sigma_m\times\IR)$,  
let $\hat{\xi}\in T_Q\hat{\Sigma}$ be a lifting of~$\xi$, i.e., 
$\xi=D\PiS\big(\hat{\xi}\big)$. Using the relations displayed in the 
commutative diagram~\eqref{eq:underst-T(Sigma)}, we get
\begin{equation}\label{eq:DPsi(xi,dot(t))=...}
  D\Psi(\xi,\dot{t})=D\Psi(D\tilde{\Pi}(\hat{\xi},\dot{t})) 
     \stackrel{\eqref{eq:underst-T(Sigma)}}{=} D\Pi(D\hat{\Psi}(\hat{\xi},\dot{t})) 
     .
\end{equation}
From the explicit form $\hat{\Psi}(Q,t) = Q\cdot Q_\rho$, 
taking into account $\rho=\arctan t$, $\tfrac{d\rho}{dt}=(1+t^2)^{-1}$,
and $\frac{d}{d\rho}Q_\rho=Q_\rho N$ (cf.~\eqref{eq:exp_Q(QU)}), 
we get
  \[ D_{(Q,t)}\hat{\Psi}(0,1) = (1+t^2)^{-1}\cdot Q\cdot Q_\rho\cdot N . \]
Using~\eqref{eq:DPsi(xi,dot(t))=...}, we thus have
\begin{align}\label{eq:D_(W,rho)-1}
   D_{(W,t)}\Psi(0,1) & = D\Pi((1+t^2)^{-1}\cdot Q\cdot Q_\rho\cdot N) \nonumber
\\ & \stackrel{\eqref{eq:D_QPi}}{=} (1+t^2)^{-1}\cdot \left[Q\cdot Q_\rho,N\right] .
\end{align}
It remains to compute the derivative of $\Psi$ in the first component.

Taking into account the orthogonal decomposition $T_W\Sigma_m=T_W^v\Sigma_m\oplus T_W^h\Sigma_m$, cf.~Corollary~\ref{cor:Sigma_m}, 
we first consider the vertical space~$T_W^v\Sigma_m$. By~\eqref{eq:T_W^v(Sigma_m)} every element $\xi\in T_W^v\Sigma_m$ is of the form $\xi=[Q,E^v(X)]$ 
for some $X\in\IR^{(n-m-1)\times(m-1)}$. Furthermore, a lifting $\hat{\xi}$ of $\xi$ is given by $\hat{\xi}=Q\cdot E^v(X)$, as $D\PiS([Q,E^v(X)])=Q\cdot E^v(X)$, 
cf.~\eqref{eq:D_QPi}. As for fixed~$\rho$, the map $Q\mapsto Q\cdot Q_\rho$ is linear, we obtain from $\hat{\Psi}(Q,t) = Q\cdot Q_\rho$
  \[ D_{(Q,t)}\hat{\Psi}\big(\hat{\xi},0\big) = 
     Q\cdot E^v(X)\cdot Q_\rho \stackrel{(*)}{=} 
     Q\cdot Q_\rho\cdot E^v(X) , \]
where the equality~$(*)$ follows from the fact that $Q_\rho$ only acts 
on the first and the last columns or rows. This implies via~\eqref{eq:DPsi(xi,dot(t))=...} and~\eqref{eq:D_QPi}
\begin{equation}\label{eq:D_(W,rho)-2}
  D_{(W,t)}\Psi(\xi,0) = \left[Q\cdot Q_\rho \,,\; \pi(E^v(X))\right] = \left[Q\cdot Q_\rho \,,\; E^v(X)\right] ,
\end{equation}
where, as usual, $\pi\colon\Skew_n\to\ol{\Skew_n}$ denotes the orthogonal projection. 

As for the horizontal space, any element $\zeta\in T_W^v\Sigma_m$ is of the form $\zeta=[Q,E^h(a)]$ for some $a\in\IR^{n-2}$, cf.~\eqref{eq:T_W^v(Sigma_m)}. 
A lifting~$\hat{\zeta}$ of~$\zeta$ is given by $\zeta=Q\cdot \hat{E}^h(a)$, as
  \[ D\PiS(\hat{\zeta}) \stackrel{\eqref{eq:D_QPi}}{=} [Q,\pi(\hat{E}^h(a))] \stackrel{\eqref{eq:E^h=pi(hat(E^h))}}{=} [Q,E^h(a)] = \zeta . \]
We thus get from $\hat{\Psi}(Q,t) = Q\cdot Q_\rho$
  \[ D_{(Q,t)}\hat{\Psi}\big(\hat{\zeta},0\big) = Q\cdot \hat{E}^h(a)\cdot Q_\rho 
     = Q\cdot Q_\rho\cdot Q_\rho^T\cdot \hat{E}^h(a)\cdot Q_\rho . \]
This implies via~\eqref{eq:DPsi(xi,dot(t))=...} and~\eqref{eq:D_QPi}
\begin{equation}\label{eq:D_(W,rho)-3}
  D_{(W,t)}\Psi(\zeta,0) = \left[Q\cdot Q_\rho \,,\; 
  \pi\left(Q_\rho^T\cdot \hat{E}^h(a)\cdot Q_\rho\right) \right] .
\end{equation}
To finish the computation, note that we have for $a,b\in\IR^{n-2}$
\begin{equation}\label{eq:D_(W,rho)-4}
  Q_\rho^T\cdot \left(\begin{array}{c|c|c} \scriptstyle0 & -a^T & \scriptstyle0 
\\\hline \rule{0mm}{4mm}a & 0 & -b 
\\\hline \scriptstyle0 & \rule{0mm}{4mm}b^T & \scriptstyle0
\end{array}\right)\cdot Q_\rho = \left(\begin{array}{c|c|c} 
\scriptstyle0 & -ca^T -sb^T & \scriptstyle0 
\\\hline \rule{0mm}{4mm}ca+sb & 0 & sa-cb 
\\\hline \scriptstyle0 & \rule{0mm}{4mm}-sa^T+cb^T & \scriptstyle0
\end{array}\right) ,
\end{equation}
where we use the abbreviations $s:=\sin(\rho)$ and $c:=\cos(\rho)$. 
Using the decompositions 
$a=\left(\begin{smallmatrix} \textstyle a_1\\[1pt]\textstyle a_2\end{smallmatrix}\right)$, 
$b=\left(\begin{smallmatrix} \textstyle b_1\\[1pt]\textstyle b_2\end{smallmatrix}\right)$
with $a_1,b_1\in\IR^{m-1}$, $a_2,b_2\in\IR^{n-m-1}$, we obtain
\begin{align}\label{eq:D_(W,rho)-5}
  & \pi\left( Q_\rho^T \left(\begin{array}{c|c|c} \scriptstyle0 & -a^T & \scriptstyle0 
\\\hline \rule{0mm}{4mm}a & 0 & -b 
\\\hline \scriptstyle0 & \rule{0mm}{4mm}b^T & \scriptstyle0
\end{array}\right) Q_\rho \right) \nonumber
\\ & \hspace{2cm} \stackrel{\eqref{eq:D_(W,rho)-4}}{=} \left(\begin{array}{c|c|c|c}
   \multicolumn{2}{c|}{\multirow{2}{*}{$0$}} & -ca_2^T-sb_2^T & \scriptstyle0 
   \\\cline{3-4} \multicolumn{2}{c|}{} & 0 & \rule{0mm}{4mm} sa_1-cb_1
\\\hline \rule{0mm}{4mm}ca_2+sb_2 & 0 & \multicolumn{2}{c}{\multirow{2}{*}{$0$}}
   \\\cline{1-2} \rule{0mm}{4mm}\scriptstyle0 & \rule{1mm}{0mm}-sa_1^T+cb_1^T\rule{1mm}{0mm}
\end{array}\right) .
\end{align}
Combining~\eqref{eq:D_(W,rho)-5} with the formula in~\eqref{eq:D_(W,rho)-3}, we get an explicit formula for the derivative of~$\Psi$ on the horizontal space.

To summarize the results from~\eqref{eq:D_(W,rho)-1},~\eqref{eq:D_(W,rho)-2}, 
and~\eqref{eq:D_(W,rho)-3}, let us define $W_\rho:=\Pi(Q\cdot Q_\rho)$, and set
  \[ T_{W_\rho}^1 := D\Psi(0\times\IR) ,\qquad 
     T_{W_\rho}^2 := D\Psi(T_W^v\Sigma_m\times0) ,\qquad 
     T_{W_\rho}^3 := D\Psi(T_W^h\Sigma_m\times0) . \]
Recall the orthogonal decomposition $\ol{\Skew_n} = \IR N \oplus \im E^v \oplus \im E^h$, cf.~\eqref{eq:decomp-ol(Skew_n)}. 
This implies the orthogonal decomposition
  \[ T_{W_\rho}\Gr_{n,m}=T_{W_\rho}^1 \oplus T_{W_\rho}^2 \oplus T_{W_\rho}^3 . \]
Furthermore, the corresponding restrictions of $D\Psi$ yield a 
dilation by the factor $(1+t^2)^{-1}$ between $0\times\IR$ and $T_{W_\rho}^1$ (cf.~\eqref{eq:D_(W,rho)-1}), 
an isometry between $T_W^v\Sigma_m\times 0$ and $T_{W_\rho}^2$ (cf.~\eqref{eq:D_(W,rho)-2}), and a 
nontrivial linear map~$\ol{D\Psi}$ between $T_W^h\Sigma_m\times 0$ and $T_{W_\rho}^3$ (cf.~\eqref{eq:D_(W,rho)-3},~\eqref{eq:D_(W,rho)-5}). This implies
\begin{equation}\label{eq:final-step-1}
  \ndet(D_{(W,t)}\Psi) = (1+t^2)^{-1}\cdot |\det\ol{D\Psi}| .
\end{equation}
We determine now $|\det\ol{D\Psi}|$.

We consider the following commutative diagram defining the linear map $\mu\colon\IR^{n-2}\to\IR^{n-2}$:
  \[ \begin{array}[c]{c}
\begin{tikzpicture}[>=stealth]
\matrix (m) [matrix of math nodes, column sep=14mm, row sep=14mm, 
             text height=1.5ex, text depth=0.25ex]
  {    T_W^h\Sigma_m\times0 & T_{W_\rho}^3 \\
       \IR^{n-2} & \IR^{(m-1)+(n-m-1)} \\
  };
\path[->]
  (m-1-1) edge node[auto]{$\ol{D\Psi}$} (m-1-2)
  (m-2-1) edge node[right]{$\beta\circ E^h$} (m-1-1)
  (m-2-2) edge node[right]{$\beta\circ \alpha$} (m-1-2);
\path[->, dashed]
  (m-2-1) edge node[auto]{$\mu$} (m-2-2);
\end{tikzpicture}
\end{array} , \]
where $\beta\colon\Skew_n\to T_W^h\Sigma_m$, $U\mapsto[Q,U]$, and $\alpha\colon\IR^{n-2}\to\Skew_n$ as defined in~\eqref{eq:def-alpha}. 
As $\alpha,\beta$ are isometries, we obtain
\begin{equation}\label{eq:final-step}
  |\det \ol{D\Psi}| = |\det(\beta\circ E^h)^{-1}|\cdot |\det\mu| 
  = |\det E^h|^{-1}\cdot |\det\mu| \stackrel{\eqref{eq:ndet(E^h)}}{=} \tdet_Y(\W_p)^{-1}\cdot |\det\mu| .
\end{equation}

As for the linear map $\mu$, we obtain from~\eqref{eq:D_(W,rho)-5} that $\mu$ has the transformation matrix 
$\begin{pmatrix} c\Lambda_1 - s I_{m-1} & c \Lambda_2 \\ s \Lambda_3 & c I_{n-m-1} + s \Lambda_4 \end{pmatrix}$, as 
$\begin{pmatrix} b_1 \\ b_2 \end{pmatrix} = \begin{pmatrix} \Lambda_1 & \Lambda_2 \\ \Lambda_3 & \Lambda_4 \end{pmatrix}\cdot \begin{pmatrix} a_1 \\ a_2 
\end{pmatrix}$ 
and
\[ \begin{pmatrix} c\Lambda_1 - s I_{m-1} & c \Lambda_2 \\ s \Lambda_3 & c I_{n-m-1} + s \Lambda_4 \end{pmatrix} \cdot \begin{pmatrix} a_1 \\ 
a_2\end{pmatrix} = \begin{pmatrix} -s a_1 + c b_1 \\ c a_2 + s b_2 \end{pmatrix} . 
\]
Using that $\frac{s}{c}=\tan\rho=t$ and $c=\cos\rho=1/\sqrt{1+t^2}$, we obtain
\begin{align*}
   \det\mu & = \det \begin{pmatrix} c\Lambda_1 - s I_{m-1} & c \Lambda_2 \\ s \Lambda_3 & c I_{n-m-1} + s \Lambda_4 \end{pmatrix}
\\ & = c^{n-2}\cdot \det \begin{pmatrix} \Lambda_1 + t I_{m-1} & \Lambda_2 \\ -t \Lambda_3 & I_{n-m-1} - t \Lambda_4 \end{pmatrix}
\\ & = (1+t^2)^{-(n-2)/2}\cdot \ch_Y(\W_p,-t) ,
\end{align*}
where the last equality is a consequence of~\eqref{eq:ch_l(A,t)=...}. From~\eqref{eq:final-step-1} and~\eqref{eq:final-step} we thus get
  \[ \ndet(D_{(W,t)}\Psi) = (1+t^2)^{-1}\cdot \frac{|\det\mu|}{\det_Y\W_p} = (1+t^2)^{-n/2}\cdot \frac{|\ch_Y(\W_p,-t)|}{\det_Y\W_p}  , \]
which finishes the proof of Theorem~\ref{thm:norm-Jac}.
\end{proof}

\appendix

\section{Averaging the twisted characteristic polynomial}\label{sec:ave-tw-charpol}
Recall the setting of Section~\ref{sec:tw-charpol}. 
Let $A\in\IR^{k\times k}$ and 
$\vp\colon\R^k\to\R^k,x\mapsto Ax$. 
Using the notation 
$\ch_\ell(A,t) = \ch_{Y_0}(\vp,t)$, 
where $Y_0 =\R^\ell\times 0$, we have 
\begin{equation}\label{eq:Y-uniform--Q-uniform}
  \underset{Y}{\IE}\left[ \ch_Y(\vp,t)\right] = 
  \underset{Q}{\IE}\left[ \ch_\ell(Q^T\,A\,Q,t)\right] . 
\end{equation}
(The reason is that $QY_0\in\Gr_{k,\ell}$ is uniform random when 
$Q\in O(k)$ uniformly chosen at random.)
This equality allows to prove Theorem~\ref{thm:ave-twist-char-poly} 
with basic matrix calculus. The main idea is to use the multilinearity of the determinant 
and the invariance of the coefficients $\sigma_i(A)$ under similarity 
transformations, to show that the coefficients of 
$\IE[\ch_\ell(Q^T\,A\,Q,t)]$
are linear combinations of 
the~$\sigma_i(A)$. We then compute the coefficients of these  
linear combinations by choosing for~$A$ scalar multiples of the identity matrix.

In the following, we use the notation $[k]:=\{1,\ldots,k\}$, 
and we denote by $\binom{[k]}{i}$ the set of all $i$-element subsets 
of $[k]$. For $A\in\IR^{k\times k}$ and $J\in\binom{[k]}{i}$, we denote 
by $\prm_J(A) := \det(A_J)$ the $J$th \emph{principal minor} of $A$, 
where $A_J$ denotes the submatrix of $A$ 
obtained by selecting the rows and columns of~$A$ whose indices 
lie in~$J$. It is well-known (cf.~for example~\cite[Thm.~1.2.12]{HJ:90})
that $\sigma_i(A)$ 
is the sum of all principal minors of~$A$ of size~$i$, i.e.,
\begin{equation}\label{eq:sigma_i-sum-pm}
  \sigma_i(A) = \sum_{J\in\binom{[k]}{i}} \prm_J(A) \; .
\end{equation}
For $0\leq \ell\leq k$, the $\ell$th \emph{leading principal minor} 
$\lpm_\ell(A) := \prm_{[\ell]}(A)$ is defined as the principal minor for $J=[\ell]$. 
Note that $\lpm_\ell(A) = \tdet_{\IR^\ell\times 0}(\vp)$. 

\begin{lemma}\label{lem:expect-princ-min}
Let $A\in\IR^{k\times k}$, and $Q\in O(k)$ be chosen uniformly at random. 
Then, for all $J\in\binom{[k]}{\ell}$, we have
\begin{equation*}\label{eq:E(Q^TAQ)}
  \underset{Q}{\IE}\left[\prm_J(Q^T A Q)\right] \;=\; 
  \underset{Q}{\IE}\left[\lpm_\ell(Q^T A Q)\right] \;=\;
  \frac{1}{\binom{k}{\ell}} \sigma_\ell(A) \; .
\end{equation*}
\end{lemma}

\begin{proof}
For the first equality let $J=\{j_1,\ldots,j_\ell\}$, $j_1<\ldots<j_\ell$, 
and let $\pi$ be any permutation of~$[k]$ such that $\pi(i)=j_i$ for all 
$i=1,\ldots,\ell$. If $M_\pi$ denotes the permutation matrix corresponding
to $\pi$, i.e.,~$M_\pi e_i=e_{\pi(i)}$, 
then $A_J=(M_\pi^T A M_\pi)_{[\ell]}$, and therefore
$\prm_J(A)=\lpm_\ell( M_\pi^T A M_\pi)$. 
This implies
\begin{align*}
   \underset{Q}{\IE}\left[\prm_J(Q^T A  Q)\right] & 
   = \underset{Q}{\IE}\left[\lpm_\ell(M_\pi^T Q^T A Q  M_\pi)\right] 
   = \underset{Q}{\IE}\left[\lpm_\ell(Q^T A  Q)\right] \; ,
\end{align*}
where we have used the fact that right multiplication by the fixed element 
$M_\pi$ leaves the uniform distribution on $O(k)$ invariant. This implies
\begin{align*}
   \underset{Q}{\IE}\left[\lpm_\ell(Q^T A Q)\right] & = 
   \frac{1}{\binom{k}{\ell}} \sum_{J\in\binom{[k]}{\ell}} 
   \underset{Q}{\IE}\left[\prm_J(Q^T A Q)\right] = 
   \frac{1}{\binom{k}{\ell}}\cdot 
   \underset{Q}{\IE}\Big[\sum_{J\in\binom{[k]}{\ell}} 
                         \prm_J(Q^T A Q)\Big]
\\ & \stackrel{\eqref{eq:sigma_i-sum-pm}}{=} \frac{1}{\binom{k}{\ell}}
   \cdot\underset{Q}{\IE}\left[\sigma_\ell(Q^T A Q)\right] 
   = \frac{1}{\binom{k}{\ell}} \sigma_\ell(A)  \; . \qedhere
\end{align*}
\end{proof}

\begin{proof}[Proof of Theorem~\ref{thm:ave-twist-char-poly}]
The statement~\eqref{eq:E_Y(det_Y)=...} follows from~\eqref{eq:Y-uniform--Q-uniform} 
and Lemma~\ref{eq:E(Q^TAQ)}. 

For proving~\eqref{eq:E_Y[ch_Y]=...},  
we start with a general observation.
By multilinearity, we write the 
determinant of a matrix with columns $v_1,\ldots,v_{i-1},v_i+\alpha\cdot e_i,v_{i+1},\ldots,v_n$
in the form
  \[ \det(v_1,\ldots,v_{i-1},v_i+\alpha\cdot e_i,v_{i+1},\ldots,v_n)
      = \det(v_1,\ldots,v_n) + \alpha\cdot \det(v_1,\ldots,v_{i-1},e_i,v_{i+1},\ldots,v_n) . \]
Using this repeatedly, we obtain
\begin{equation}\label{eq:det-exp}
 \det(A+\diag(\alpha_1,\ldots,\alpha_n)) = \sum_{J\subseteq[k]} \prm_J(A)\cdot \prod_{i\in[k]\setminus J} \alpha_i \; . 
\end{equation}
By~\eqref{eq:ch_l(A,t)=...} we can express the twisted characteristic polynomial as 
$$
   \ch_\ell(A,t)  =t^{k-\ell}\cdot \det(A+\diag(-t,\ldots,-t,1/t,\ldots,1/t)) \; .
$$
Applying~\eqref{eq:det-exp},  we expand this to obtain
\begin{equation}\label{eq:expan-ch_l(A)}
   \ch_\ell(A,t) = \sum_{J\subseteq[k]} (-1)^{c_1(J)}\cdot 
                        \prm_J(A)\cdot t^{c_2(J)} \;, 
\end{equation}
where $c_1,c_2\colon 2^{[k]}\to\IN$ are some integer valued functions 
on the power set of~$[k]$. Averaging the twisted characteristic polynomial 
with the help of Lemma~\ref{lem:expect-princ-min} yields
\begin{align*}
   \underset{Q}{\IE}\left[\ch_\ell(Q^T A Q,t)\right] & = 
     \sum_{J\subseteq[k]} (-1)^{c_1(J)}\cdot 
     \underset{Q}{\IE}\left[\prm_J(Q^T A Q)\right]\cdot t^{c_2(J)}
\\ & = 
     \sum_{J\subseteq[k]} \frac{(-1)^{c_1(J)}}{\binom{k}{|J|}}\cdot 
       \sigma_{|J|}(A) t^{c_2(J)} = 
     \sum_{i,j=0}^k \tilde{d}_{ij}\cdot \sigma_{k-j}(A)\cdot t^{k-i} \; ,
\end{align*}
for some rational constants $\tilde{d}_{ij}$. 
It remains to prove that the $\tilde{d}_{ij}$ coincide with the $d_{ij}$ 
defined in Theorem~\ref{thm:ave-twist-char-poly}. 

To compute the $\tilde{d}_{ij}$, we choose $A=s I_k$. Then,
$\ch_\ell(s I_k,t) = (s-t)^\ell\cdot (1+s t)^{k-\ell}$. 
Since $\sigma_{k-j}(s I_k) = \binom{k}{j} s^{k-j}$, and 
$Q^T s I_k  Q = s I_k$, we get
\begin{align*}
   (s-t)^\ell\cdot (1+s t)^{k-\ell} = \ch_\ell(s I_k,t) & 
     = \underset{Q}{\IE}\left[\ch_\ell(Q^T s I_k Q,t)\right]
     = \sum_{i,j=0}^k \tilde{d}_{ij} \binom{k}{j} s^{k-j} t^{k-i} \; .
\end{align*}
Let us expand the first term: 
\begin{align}\label{eq:chng-summation}
   (s-t)^\ell \cdot (1 + s t)^{k-\ell} & = 
   \left(\sum_{\lambda=0}^\ell\binom{\ell}{\lambda} (-1)^{\ell-\lambda}
      s^\lambda t^{\ell-\lambda}\right)\cdot 
     \left(\sum_{\mu=0}^{k-\ell} \binom{k-\ell}{\mu} s^\mu t^\mu\right) \nonumber
\\ & = \sum_{\lambda=0}^\ell \sum_{\mu=0}^{k-\ell} (-1)^{\ell-\lambda} 
     \binom{\ell}{\lambda} \binom{k-\ell}{\mu} s^{\lambda+\mu} t^{\ell-\lambda+\mu}
\\ & \hspace{-6mm}\stackrel{\substack{i=k-\ell+\lambda-\mu \\ j=k-\lambda-\mu}}{=} 
     \hspace{-2mm}\sum_{\substack{i,j=0\\i+j+\ell\equiv 0\\\pmod 2}}^k 
     (-1)^{\frac{i-j}{2}-\frac{\ell}{2}}\binom{\ell}{\frac{i-j}{2}+\frac{\ell}{2}} 
     \binom{k-\ell}{k-\frac{i+j+\ell}{2}} s^{k-j} t^{k-i} , \nonumber
\end{align}
where again we interpret $\binom{n}{m}=0$ if $m<0$ or $m>n$, i.e., the above 
summation over $i,j$ in fact only runs over the rectangle determined by 
the inequalities $0\leq \frac{i-j}{2}+\frac{\ell}{2}\leq \ell$ and 
$0\leq k-\frac{i+j+\ell}{2}\leq k-\ell$. 
(See Figure~\ref{fig:chng-summation} 
for an illustration of the change of summation.) 
Note that the reverse substitution is given by $\lambda=\frac{i-j+\ell}{2}$ and $\mu=k-\frac{i+j+\ell}{2}$.

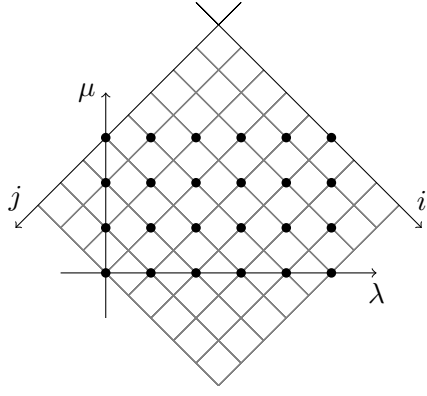
\begin{figure}
  \begin{center}
  \begin{tikzpicture}[scale=0.6]
  \def\ktmp{8}
  \def\ltmp{5}
  \pgfmathsetmacro\kmltmp{\ktmp-\ltmp}
  \pgfmathsetmacro\kpltmp{\ktmp+\ltmp}

  \foreach \i in {1,...,\ktmp}
    \foreach \j in {1,...,\ktmp}
    { \draw[gray] (0.5*\ltmp-0.5*\i,\ktmp-0.5*\ltmp-0.5*\i) -- +(0.5*\ktmp,-0.5*\ktmp);
      \draw[gray] (0.5*\ltmp+0.5*\i,\ktmp-0.5*\ltmp-0.5*\i) -- +(-0.5*\ktmp,-0.5*\ktmp); }

  \draw[->] (-1,0) -- (\ltmp+1,0) node[below]{$\lambda$};
  \draw[->] (0,-1) -- (0,\kmltmp+1) node[left]{$\mu$};
  \foreach \x in {0,1,...,\ltmp}
    \foreach \y in {0,1,...,\kmltmp}
      \fill (\x,\y) circle (3pt);

  \draw[->] (0.5*\ltmp,\ktmp-0.5*\ltmp)
             -- +(0.5,0.5) -- (-0.5*\kmltmp,0.5*\kmltmp)
             -- ++(-0.5,-0.5) node[above=1mm]{$j$};
  \draw[->] (0.5*\ltmp,\ktmp-0.5*\ltmp) -- +(-0.5,0.5) -- (0.5*\kpltmp,0.5*\kmltmp)
             -- ++(0.5,-0.5) node[above=1mm]{$i$};
  \end{tikzpicture}
  \end{center}
  \caption{Illustration of the change of summation in~\eqref{eq:chng-summation} ($k=8$, $\ell=5$).}
  \label{fig:chng-summation}
\end{figure}

Comparing the coefficients of the above two expressions for 
$(s-t)^\ell\cdot (1+s t)^{k-\ell}$ reveals that indeed 
$\tilde{d}_{ij}=d_{ij}$ as defined in Theorem~\ref{thm:ave-twist-char-poly}. 
This completes the proof of~\eqref{eq:E_Y[ch_Y]=...} 

For the last claim~\eqref{star} in Theorem~\ref{thm:ave-twist-char-poly}, 
we define the {\em positive characteristic polynomial} 
\begin{equation*}
     \ch_\ell^+(A,t) = \det \begin{pmatrix} A_1+tI_\ell & A_2 
                          \\ tA_3 & tA_4+I_{k-\ell} \end{pmatrix} 
\end{equation*}
by replacing $-t$ by $t$ in \eqref{eq:ch_l(A,t)=...}.   
By the same reasoning as for~\eqref{eq:expan-ch_l(A)}, we get 
\begin{equation}\label{eq:ch+}
  \ch^+_\ell(A,t) = \sum_{J\subseteq[k]} \prm_J(A)\cdot t^{c_2(J)}  \; ,
\end{equation}
with the same function $c_2\colon 2^{[k]}\to\IN$ as in~\eqref{eq:expan-ch_l(A)}.
Moreover, by arguing as in the proof
of~\eqref{eq:E_Y[ch_Y]=...} before, we show that 
\begin{equation}\label{eq:E(ch_l^+(Q^TAQ))}
 \underset{Q}{\IE}\left[\ch^+_\ell(Q^TAQ,t)\right]  = 
     \sum_{i,j=0}^k |d_{ij}|\cdot \sigma_{k-j}(A)\cdot t^{k-i} \; . 
\end{equation} 
Now assume that $A$ is positive semidefinite. 
Then each of its principal minor is nonnegative, 
i.e., $\prm_J(A)\geq0$ for all $J\subseteq[k]$. 
Therefore, if $t\geq0$, 
we get from~\eqref{eq:expan-ch_l(A)} and \eqref{eq:ch+}
\begin{align*}
   \Big|\ch_\ell(A,t)\Big| & = 
   \Big| \sum_{J\subseteq[k]} (-1)^{c_1(J)}\cdot \prm_J(A)\cdot t^{c_2(J)} \Big|
   \leq \sum_{J\subseteq[k]} \prm_J(A)\cdot t^{c_2(J)} 
    =  \ch^+_\ell(A,t) \; .
\end{align*}
Taking into account~\eqref{eq:E(ch_l^+(Q^TAQ))}
completes the proof of Theorem~\ref{thm:ave-twist-char-poly}.
\end{proof}

\section{Proof of some technical estimations}\label{sec:proof_technicalities}

Recall that the symbol~$\calc$ is used to mark simple estimates, which are easily checked with a computer algebra system.

\begin{proof}[Proof of Lemma~\ref{lem:estimates}]
(1) We make a case distinction by the parity of~$\ell$.
Using $\Gamma(x+1)=x\cdot \Gamma(x)$, we get for odd~$\ell$
\begin{align*}
   \frac{\Gamma(\frac{m+\ell+1}{2})}{\Gamma(\frac{m}{2})} & 
   = \prod_{a=0}^{\frac{\ell-1}{2}} \left(\frac{m}{2}+a\right) 
   \leq \frac{m}{2}\cdot \left(\frac{m+\ell-1}{2}\right)^{\frac{\ell-1}{2}} 
   \leq \sqrt{\frac{m}{2}}\cdot \left(\frac{m+\ell}{2}\right)^{\frac{\ell}{2}} \; .
\end{align*}
Using additionally $\Gamma(x+\frac{1}{2})<\sqrt{x}\cdot \Gamma(x)$, 
we get for even~$\ell$
\begin{align*}
   \frac{\Gamma(\frac{m+\ell+1}{2})}{\Gamma(\frac{m}{2})} & 
   = \prod_{a=0}^{\frac{\ell}{2}-1} \left(\frac{m+1}{2}+a\right) 
     \cdot \frac{\Gamma(\frac{m+1}{2})}{\Gamma(\frac{m}{2})} 
   < \left(\frac{m+\ell}{2}\right)^{\frac{\ell}{2}}\cdot \sqrt{\frac{m}{2}} \; .
\end{align*}

(2) As for the second estimate, we distinguish the cases $i\geq 2k$ 
and $i\leq 2k$. From $0\leq k\leq m-1$ and $0\leq i-k\leq n-m-1$ we get
\[\begin{array}{r@{\;\;\leq\;\;}c@{\;\;\leq\;\;}l}
   1 & m-k+i-k & n-1 \; ,
\\[1mm] 1 & n-(m-k+i-k) & n-1 \; .
\end{array}\]
For $i\geq 2k$ we thus get
\begin{align*}
   \left(\frac{m+i-2k}{n-m-i+2k}\right)^{\frac{i-2k}{2}} & 
   \leq (n-1)^{\frac{i-2k}{2}} < n^{\frac{i}{2}} ,
\intertext{and for $i\leq 2k$}
   \left(\frac{m+i-2k}{n-m-i+2k}\right)^{\frac{i-2k}{2}} & 
   = \left(\frac{n-m+2k-i}{m-2k+i}\right)^{\frac{2k-i}{2}} 
   \leq (n-1)^{\frac{2k-i}{2}} < n^{\frac{i}{2}} \; .
\end{align*}

(3) The $I$-functions have been estimated in~\cite[Lemma~2.2]{BCL:08} 
in the following way. Let $\veps:=\sin(\alpha)=\frac{1}{t}$. For $i<n-2$
\begin{align*}
   I_{n,n-2-i}(\alpha) & = \int_0^\alpha \cos(\rho)^{n-2-i}\cdot \sin(\rho)^i\,d\rho
 \leq \frac{\veps^{i+1}}{i+1} \; ,
\end{align*}
and for $i=n-2$, assuming $n\geq3$,
\begin{align*}
   I_{n,0}(\alpha) & = \int_0^\alpha \sin(\rho)^{n-2}\,d\rho
    \leq \frac{\mcO_{n-1}\cdot \veps^{n-1}}{2\mcO_{n-2}}
    = \frac{\sqrt{\pi}}{2}\cdot \frac{\Gamma(\frac{n-1}{2})}{\Gamma(\frac{n}{2})}
      \cdot \veps^{n-1}
    \stackrel{\calc}{<} \sqrt{\frac{\pi}{2(n-2)}}\cdot \veps^{n-1} \; .
\end{align*}
With these estimates we get
\begin{align*}
   \sum_{i=0}^{n-2} \binom{n-2}{i} & \cdot n^{\frac{i}{2}} \cdot I_{n,n-2-i}(\alpha)
\\ & \leq \sum_{i=0}^{n-2} \binom{n-2}{i}\cdot n^{\frac{i}{2}} 
   \cdot \frac{\veps^{i+1}}{i+1} + \left(\sqrt{\frac{\pi}{2(n-2)}}-\frac{1}{n-1}\right)
   \cdot \veps^{n-1}\cdot n^{\frac{n-2}{2}}
\\ & \stackrel{\calc}{<}\; \veps\cdot \Bigg( \sum_{i=0}^{n-2} \binom{n-2}{i}
   \cdot n^{\frac{i}{2}} \cdot \veps^i + \frac{1.4}{\sqrt{n}}\cdot \veps^{n-2}
   \cdot n^{\frac{n-2}{2}} \Bigg)
\\ & = \veps\cdot \Bigg( \left(1 + \sqrt{n}\cdot \veps\right)^{n-2} + 1.4
   \cdot \veps^{n-2}\cdot n^{\frac{n-3}{2}} \Bigg) \; .
\end{align*}
For $\veps<n^{-\frac{3}{2}}$ we thus get
\begin{align*}
   \sum_{i=0}^{n-2} \binom{n-2}{i}\cdot n^{\frac{i}{2}} \cdot I_{n,n-2-i}(\alpha) & 
   < \veps\cdot \Bigg( \underbrace{\left(1 + \frac{1}{n}\right)^{n-2}}_{<\exp(1)} 
     + 1.4\cdot n^{\frac{3}{2}-n} \Bigg)
   \;\stackrel{\calc}{<}\; 3\cdot \veps \; .
\end{align*}
Similarly we derive
  \[ \sum_{i=0}^{n-2} \binom{n-2}{i} \cdot I_{n,n-2-i}(\alpha) 
     < \veps\cdot \left((1+\veps)^{n-2} 
       + \frac{1.4}{\sqrt{n}}\cdot \veps^{n-2}\right) \; . \]
For $\veps<\frac{1}{m}$ we thus get
\begin{align*}
   \sum_{i=0}^{n-2} \binom{n-2}{i} \cdot I_{n,n-2-i}(\alpha) & 
     < \veps\cdot \left(\left(1+\frac{1}{m}\right)^{n-2} 
       + \frac{1.4}{\sqrt{n}\cdot m^{n-2}}\right)
     \;\stackrel{\calc}{<}\; \veps\cdot \exp\left(\frac{n}{m}\right) \; . \qedhere
\end{align*}
\end{proof}

{\small \bibliography{grassmann}}

\end{document}